\documentclass[letterpaper,11pt]{article}

\usepackage[margin=1in]{geometry}  
\usepackage{xspace,exscale,relsize}
\usepackage{fancybox,shadow}
\usepackage{graphicx}
\usepackage{color}
\usepackage{amsthm,amsfonts}
\usepackage{amssymb}
\usepackage{authblk}
\usepackage[title]{appendix}

\usepackage{extarrows}

\usepackage[noadjust]{cite}

\usepackage{amsmath}
	\makeatletter
	\let\over=\@@over \let\overwithdelims=\@@overwithdelims
	\let\atop=\@@atop \let\atopwithdelims=\@@atopwithdelims
  	\let\above=\@@above \let\abovewithdelims=\@@abovewithdelims
  	\makeatother
\usepackage{rotating}

\usepackage{ifpdf}

\usepackage{subfigure}
\usepackage{psfrag}

\usepackage{prettyref}
\usepackage{enumitem}

\usepackage{tikz}
\usetikzlibrary{arrows}
\tikzstyle{int}=[draw, fill=blue!20, minimum size=2em]
\tikzstyle{dot}=[circle, draw, fill=blue!20, minimum size=2em]
\tikzstyle{init} = [pin edge={to-,thin,black}]

\usepackage[
            CJKbookmarks=true,
            bookmarksnumbered=true,
            bookmarksopen=true,
            colorlinks=true,
            citecolor=red,
            linkcolor=blue,
            anchorcolor=red,
            urlcolor=blue
            ]{hyperref}

\numberwithin{equation}{section}

\usepackage[all]{xy}

\usepackage{mathtools}

\usepackage{ifthen}
\newboolean{aos}
\setboolean{aos}{FALSE}

\numberwithin{equation}{section}

\newcommand{\R}{\mathbb{R}}

\newcommand{\pnorm}[2]{\lVert #1\rVert_{#2}}
\newcommand{\bigpnorm}[2]{\big\lVert#1\big\rVert_{#2}}

\renewcommand{\epsilon}{\varepsilon}

\renewcommand{\d}[1]{\mathrm{d}#1}

\newcommand{\floor}[1]{\left\lfloor #1 \right\rfloor}
\newcommand{\ceil}[1]{\left\lceil #1 \right\rceil}

\renewcommand{\hat}{\widehat}
\renewcommand{\tilde}{\widetilde}

\DeclareMathOperator{\E}{\mathbb{E}}
\DeclareMathOperator{\Prob}{\mathbb{P}}

\DeclareMathOperator{\var}{Var}
\DeclareMathOperator{\cov}{Cov}

\DeclareMathOperator{\emp}{emp}

\DeclareMathOperator{\totreg}{\mathsf{TotRegret}}
\DeclareMathOperator{\reg}{\mathsf{Regret}}
\DeclareMathOperator{\poi}{\mathsf{Poi}}
\DeclareMathOperator{\bin}{\mathsf{Bin}}
\DeclareMathOperator{\hb}{\mathsf{hybrid}}
\DeclareMathOperator{\npmle}{\mathsf{NPMLE}}
\DeclareMathOperator{\rob}{\mathsf{Robbins}}

\DeclareMathOperator*{\argmax}{arg\,max\,}
\DeclareMathOperator*{\argmin}{arg\,min\,}

\newcommand{\beq}{\begin{equation}}
\newcommand{\eeq}{\end{equation}}
\newcommand{\beqa}{\begin{equation} \begin{aligned}}
\newcommand{\eeqa}{\end{aligned} \end{equation}}
\newcommand{\beqas}{\begin{equation*} \begin{aligned}}
\newcommand{\eeqas}{\end{aligned} \end{equation*}}

\newcommand{\bit}{\begin{itemize}}
	\newcommand{\eit}{\end{itemize}}
\newcommand{\bmat}{\begin{bmatrix}}
	\newcommand{\emat}{\end{bmatrix}}

\newcommand\numberthis{\addtocounter{equation}{1}\tag{\theequation}}

\newcommand{\supp}{\ensuremath{\mathrm{supp}}}

\ifx\eqref\undefined
	\newcommand{\eqref}[1]{~(\ref{#1})}
\fi
\ifx\mod\undefined
	\def\mod{\mathop{\rm mod}}
\fi

\usepackage{bm}

\def\argmin{\mathop{\rm argmin}}
\def\argmax{\mathop{\rm argmax}}

\def\exp{\mathop{\rm exp}}

\def\cov{\mathop{\rm cov}}

\newcommand{\stepa}[1]{\overset{\rm (a)}{#1}}
\newcommand{\stepb}[1]{\overset{\rm (b)}{#1}}

\newcommand{\Poi}{\mathsf{Poi}}

\newcommand{\reals}{\mathbb{R}}

\newcommand{\integers}{\mathbb{Z}}

\newcommand{\Expect}{\mathbb{E}}

\newcommand{\iid}{i.i.d.\xspace}

\newcommand{\pth}[1]{\left( #1 \right)}

\newcommand{\iiddistr}{{\stackrel{\text{\iid}}{\sim}}}

\newcommand{\indc}[1]{{\mathbf{1}\left\{{#1}\right\}}}

\definecolor{myblue}{rgb}{.8, .8, 1}
\definecolor{mathblue}{rgb}{0.2472, 0.24, 0.6} 
\definecolor{mathred}{rgb}{0.6, 0.24, 0.442893}
\definecolor{mathyellow}{rgb}{0.6, 0.547014, 0.24}

\newcommand{\calG}{{\mathcal{G}}}

\newcommand{\calP}{{\mathcal{P}}}

\newcommand{\mmse}{\mathsf{mmse}}

\newrefformat{eq}{(\ref{#1})}
\newrefformat{ineq}{(\ref{#1})}
\newrefformat{thm}{Theorem~\ref{#1}}
\newrefformat{th}{Theorem~\ref{#1}}
\newrefformat{chap}{Chapter~\ref{#1}}
\newrefformat{sec}{Section~\ref{#1}}
\newrefformat{seca}{Section~\ref{#1}}
\newrefformat{subsec}{Section~\ref{#1}}
\newrefformat{algo}{Algorithm~\ref{#1}}
\newrefformat{fig}{Fig.~\ref{#1}}
\newrefformat{tab}{Table~\ref{#1}}
\newrefformat{rmk}{Remark~\ref{#1}}
\newrefformat{remark}{Remark~\ref{#1}}
\newrefformat{clm}{Claim~\ref{#1}}
\newrefformat{def}{Definition~\ref{#1}}
\newrefformat{cor}{Corollary~\ref{#1}}
\newrefformat{lmm}{Lemma~\ref{#1}}
\newrefformat{prop}{Proposition~\ref{#1}}
\newrefformat{pr}{Proposition~\ref{#1}}
\newrefformat{app}{Appendix~\ref{#1}}
\newrefformat{apx}{Appendix~\ref{#1}}
\newrefformat{ex}{Example~\ref{#1}}
\newrefformat{exer}{Exercise~\ref{#1}}
\newrefformat{soln}{Solution~\ref{#1}}

\def\unifto{\mathop{{\mskip 3mu plus 2mu minus 1mu%
	\setbox0=\hbox{$\mathchar"3221$}%
	\raise.6ex\copy0\kern-\wd0%
	\lower0.5ex\hbox{$\mathchar"3221$}}\mskip 3mu plus 2mu minus 1mu}}

\ifx\lesssim\undefined
\def\simleq{{{\mskip 3mu plus 2mu minus 1mu%
	\setbox0=\hbox{$\mathchar"013C$}%
	\raise.2ex\copy0\kern-\wd0%
	\lower0.9ex\hbox{$\mathchar"0218$}}\mskip 3mu plus 2mu minus 1mu}}
\else
\def\simleq{\lesssim}
\fi

\ifx\gtrsim\undefined
\def\simgeq{{{\mskip 3mu plus 2mu minus 1mu%
	\setbox0=\hbox{$\mathchar"013E$}%
	\raise.2ex\copy0\kern-\wd0%
	\lower0.9ex\hbox{$\mathchar"0218$}}\mskip 3mu plus 2mu minus 1mu}}
\else
\def\simgeq{\gtrsim}
\fi

\newtheorem{theorem}{Theorem}
\newtheorem{lemma}[theorem]{Lemma}

\newtheorem{proposition}[theorem]{Proposition}

\theoremstyle{definition}

\newtheorem{remark}{Remark}

\newif\ifmapx
{\catcode`/=0 \catcode`\\=12/gdef/mkillslash\#1{#1}}
\edef\jobnametmp{\expandafter\string\csname embayes2_apx\endcsname}
\edef\jobnameapx{\expandafter\mkillslash\jobnametmp}
\edef\jobnameexpand{\jobname}
\ifx\jobnameexpand\jobnameapx
\mapxtrue
\else
\mapxfalse
\fi

\newcommand{\poly}{\mathsf{poly}}
\newcommand{\polylog}{\mathsf{polylog}}

\renewcommand{\hat}{\widehat}
\renewcommand{\tilde}{\widetilde}

\def\mmse{\mathrm{mmse}}

\begin{document}
\ifpdf
\DeclareGraphicsExtensions{.pgf}
\graphicspath{{code/}{plots/}}
\fi

\title{Poisson empirical Bayes estimation: When does $g$-modeling beat $f$-modeling in theory (and in practice)?}
\author{Yandi Shen and Yihong Wu\thanks{
Y.~Shen is with the Department of Statistics, The University of Chicago, Chicago IL, USA, 
\texttt{ydshen@uchicago.edu}.
Y.~Wu is with the Department of Statistics and Data Science, Yale University, New Haven CT, USA, 
\texttt{yihong.wu@yale.edu}.
Y.~Wu is supported in part by the NSF Grant CCF-1900507, an NSF CAREER award CCF-1651588, and an Alfred Sloan fellowship.}}
\date{\today}

\maketitle

\begin{abstract}

Empirical Bayes (EB) is a popular framework for large-scale  inference that aims to find data-driven estimators to compete with the Bayesian oracle that knows the true prior. 
Two principled approaches to EB estimation have emerged over the years: \emph{$f$-modeling}, which constructs an approximate Bayes rule by estimating the marginal distribution of the data, and \emph{$g$-modeling}, which estimates the prior from data and then applies the learned Bayes rule. For the Poisson model, the prototypical examples are the celebrated Robbins estimator and the nonparametric MLE (NPMLE), respectively. It has long been recognized in practice that the Robbins estimator, while being conceptually appealing and computationally simple,  lacks robustness and can be easily derailed by ``outliers'' (data points that were rarely observed before), 
unlike the NPMLE which provides more stable and interpretable fit thanks to its Bayes form. On the other hand, not only do the existing theories shed little light on this phenomenon, but they all point to the opposite, as both methods have recently been shown optimal in terms of the \emph{regret} (excess over the Bayes risk) for compactly supported and subexponential priors with exact logarithmic factors \cite{brown2013poisson,polyanskiy2021sharp}.

In this paper we provide a theoretical justification for the superiority of $g$-modeling over $f$-modeling for heavy-tailed data by considering priors with bounded $p$th moment previously studied for the Gaussian model \cite{jiang2009general}. For the Poisson model with sample size $n$, assuming $p>1$ (for otherwise triviality arises), we show that with mild regularization, any $g$-modeling method that is Hellinger rate-optimal in density estimation achieves a total regret $\tilde \Theta(n^{\frac{3}{2p+1}})$, which is minimax optimal within logarithmic factors; in particular, the special case of NPMLE succeeds without regularization. 
In contrast, there exists an $f$-modeling estimator whose density estimation rate is optimal but whose EB regret is suboptimal by a polynomial factor. 
These results show that the proper Bayes form provides a ``general recipe of success'' for optimal EB estimation that applies to all $g$-modeling (but not $f$-modeling) methods.
As by-products of our analysis, we also obtain 
(a) the minimax Hellinger rate of estimating Poisson mixture over the moment class;
(b) the characterization of the regret suboptimality of the Robbins estimator;
(c) an extension to the compound setting.
\end{abstract}

\newpage

\setcounter{tocdepth}{3}
\tableofcontents



\section{Introduction}

\subsection{Overview}
\label{sec:overview}

Introduced by Robbins \cite{robbins1951asymptotically, robbins1956empirical} in the 1950s, Empirical Bayes (EB) is a meaningful and powerful framework for large-scale inference that allows one to go beyond worst-case analysis and obtain data-driven estimators that adapt to the latent structure in the data.
Under the Poisson EB model, $\theta^n \equiv (\theta_1,\ldots,\theta_n)$ are latent parameters drawn independently from an unknown prior distribution $G$ supported on $\reals_+ \equiv [0,\infty)$, and conditioned on $\theta^n$, the observed $Y^n = (Y_1,\ldots,Y_n)$ are independently distributed as $Y_i|\theta_i \sim \poi(\theta_i)$, the Poisson distribution with parameter $\theta_i$.
Consequently, the marginal distribution of each $Y_i$ is the following Poisson mixture:
\begin{align}\label{def:poi_mixture_intro}
f_G(y) \equiv \int \poi(y;\theta) G(\d \theta), \quad y\in \integers_+,
\end{align}
where $\poi(y;\theta)\equiv \frac{e^{-\theta} \theta^y}{y!}$ denotes the probability mass function (pmf) of $\poi(\theta)$ throughout the paper. Given a class of priors $G$, the goal is to estimate the $n$ latent Poisson means $\theta^n$ with a minimal total risk. The EB problem, along with its twin problem of compound estimation, have found deep connections to and fruitful applications in a number of areas in statistics, including admissibility, adaptive nonparametric estimation, variable selection, multiple testing, as well as practical data analysis. We refer to the review articles \cite{casella1985introduction,zhang2003compound, efron2021empirical} and the monographs \cite{maritz1989empirical,carlin1996bayes,efron2010large} for a systematic treatment of this broad subject. 

For the squared error, the Bayes estimator minimizing the average risk is the posterior mean, given by
\begin{align}\label{def:poi_bayes_form_intro}
\theta_G(y) \equiv \E_G[\theta| Y = y] = (y+1)\frac{f_G(y+1)}{f_G(y)}, 
\end{align} 
and the Bayes risk is denoted by\footnote{Here and below, $\E_G$ and $\Prob_G$ are taken under the prior $G$.}
\begin{align}\label{def:bayes_risk_intro}
\mathsf{mmse}(G) \equiv \inf_{\hat{\theta}} \E_G\big(\hat{\theta}(Y) - \theta\big)^2 = \E_G\big(\theta_G(Y) - \theta\big)^2,
\end{align}
where $\theta\sim G$ and $Y|\theta\sim \poi(\theta)$, and the infimum is taken over all measurable functions of $Y$. 
Clearly, evaluating the Bayes estimator requires the knowledge of the prior $G$. For this reason, we refer to \eqref{def:poi_bayes_form_intro} as the \emph{oracle}.

In the Poisson EB model with $n$ i.i.d.~observations $Y^n$, the oracle applies the Bayes rule (\ref{def:poi_bayes_form_intro}) separately to each $Y_i$ to estimate $\theta_i$, resulting in the minimal total risk $n\cdot \mathsf{mmse}(G)$. Using this as a benchmark, the goal of EB estimation is to find a data-driven estimator $\hat{\theta}^n(Y^n):\integers_+^n \rightarrow \reals_+^n$ without knowing the exact prior that approaches the oracle risk as closely as possible.
To this end, the principal metric is the excess risk, also known as the \emph{regret} in the EB literature (see Section \ref{subsec:regret_prelim} for related definitions):
\begin{align}\label{def:reg_intro}
\totreg_n(\hat{\theta}^n;\mathcal{G}) \equiv \sup_{G\in\mathcal{G}}\Big\{\E_{G}\pnorm{\hat{\theta}^n(Y^n) - \theta^n}{}^2 - n\cdot \mathsf{mmse}(G)\Big\},
\end{align} 
where the supremum is taken over a class $\mathcal{G}$ of priors. Since the typical order of the Bayes risk $\mathsf{mmse}(G)$ is $O(1)$, we say the estimator $\hat{\theta}^n$ is \emph{consistent} over $\mathcal{G}$ if its regret satisfies $\totreg_n(\hat{\theta}^n;\mathcal{G}) = o(n)$ as $n\to\infty$,
so that the amortized regret per observation is vanishing; this is referred to as asymptotic optimality in Robbins' original framework \cite{robbins1956empirical}. 
Since then, significant progress has been achieved in understanding the rate of $\totreg_n(\cdot;\mathcal{G})$ 
for specific procedures as well as their optimality -- cf.~\cite{singh1979empirical,li2005convergence,jiang2009general,brown2013poisson,polyanskiy2021sharp} and the references therein.

From the previous discussion, it is clear that the key to obtaining a small regret is to accurately learn the oracle Bayes rule (\ref{def:poi_bayes_form_intro}) from the observed data. 
The majority of the current EB literature centers around two principled approaches, aptly named ``$f$-modeling" and ``$g$-modeling" \cite{efron2014two}:
\begin{itemize}
\item The $f$-modeling approach is concerned with directly estimating the mixture density $f_G$ in the Bayes rule (\ref{def:poi_bayes_form_intro}). 
For the Poisson model, the leading example in this category is the celebrated  estimator of Robbins \cite{robbins1956empirical}, which substitutes the mixture density 
in \eqref{def:poi_bayes_form_intro} by the empirical frequency:
\begin{align}\label{def:robbins_intro}
\theta^{\rob}(y) \equiv (y+1)\frac{N_n(y+1)}{N_n(y)},
\end{align} 
where $N_n(y) = \sum_{i=1}^n \indc{Y_i = y}$ is the empirical count of $y$ in the sample $Y^n$. 
 Other examples of $f$-modeling, developed for both Poisson and other exponential families, include smoothed (kernel) estimates of the mixture density \cite{good1953population, singh1979empirical, zhang1997empirical,pensky1999nonparametric, li2005convergence, zhang2005general, brown2009nonparametric, efron2019bayes}.

\item The $g$-modeling approach proceeds by first producing an estimator $\hat G$ of the prior $G$ and applying the Bayes rule corresponding to $\hat G$.\footnote{In this sense, one can view $g$-modeling as a special case of $f$-modeling which uses \emph{proper} density estimators that are valid mixture distributions. In contrast, most $f$-modeling approaches apply 
improper density estimates such as empirical distribution or kernel methods.} The leading example in this category is the nonparametric maximum likelihood estimator (NPMLE), originally proposed in \cite{kiefer1956consistency}:
\begin{align}\label{def:npmle_intro}
\hat{G} \equiv \argmax_{G \in \mathcal{G}} \prod_{i=1}^n f_G(Y_i).
\end{align}
After $\hat{G}$ is obtained, we apply the plug-in Bayes rule $\theta_{\hat{G}}(y) = (y+1)f_{\hat{G}}(y+1)/f_{\hat{G}}(y)$ as in (\ref{def:poi_bayes_form_intro}) to each observation $Y_i$. Other notable examples in this category include parametric modeling of the prior \cite{morris1983parametric, casella1985introduction}, and the nonparametric suite of minimum-distance estimators \cite{wolfowitz1953estimation,jana2022poisson}, which contains the NPMLE as a special case.
\end{itemize}

%

From a methodological perspective, 
it is well-recognized that $g$-modeling exhibits the following advantages over $f$-modeling:
\begin{itemize}
\item The Bayes form of 
$g$-modeling estimators leads to more interpretable (e.g.~monotone) and frequently more accurate estimates \cite{koenker2014convex}. 

\item The $g$-modeling approach is more flexible in incorporating knowledge of the prior distribution. For example, the sparse case can be readily dealt with by 
restricting the likelihood optimization to priors with a prescribed atom at zero \cite[Section 5]{efron2014two}.

\item The $f$-modeling approach, exemplified by the Robbins estimator, lacks robustness and exhibits numerical instability in practical settings (see, e.g., \cite{maritz1968smooth}, \cite[Section 1.9]{maritz1989empirical}, \cite[Section 6.1]{efron2021computer}, \cite{jana2022poisson}). In fact, it is easily derailed by ``outliers'', i.e., data points that appear only a few times, for which either the numerator or denominator in (\ref{def:robbins_intro}) is small, causing the estimator to take exceptionally small or large values.
See \prettyref{fig:compare} for an example with heavy-tailed priors.
\end{itemize}
On the other hand, $f$-modeling is widely applied in practice due to its computational simplicity, while $g$-modeling, especially in nonparametric settings and general dimensions, is more expensive to compute.

\begin{figure}[ht]%
\centering
\includegraphics[width=0.7\columnwidth]{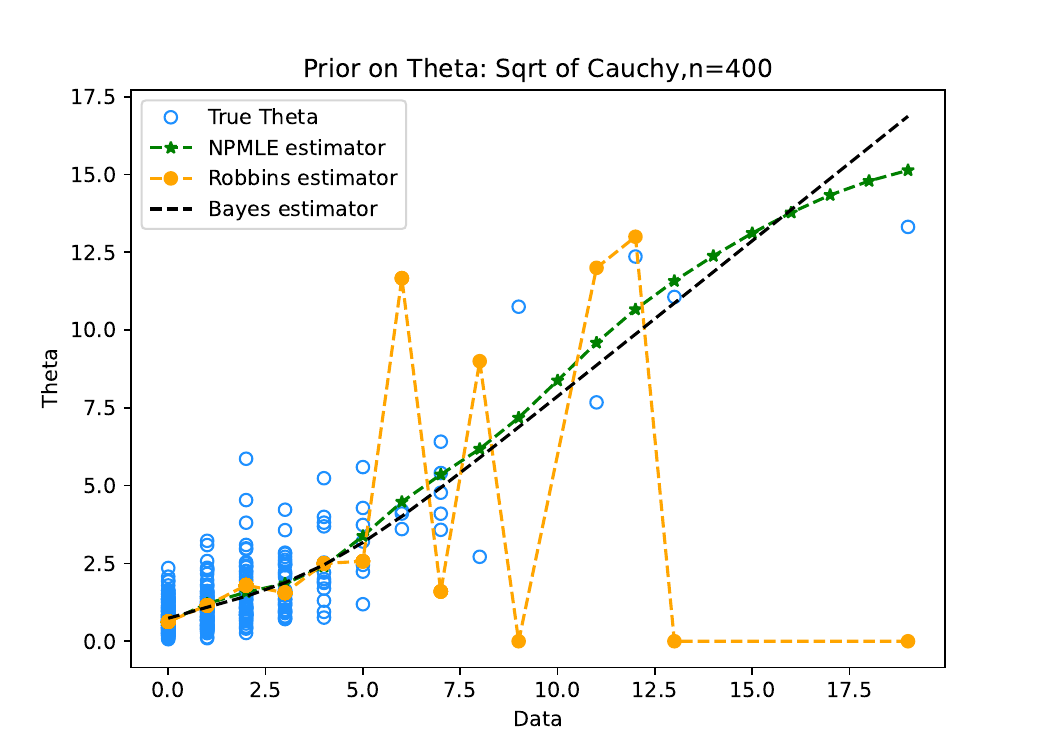}%
\caption{NPMLE vs Robbins vs Bayes estimator for heavy-tailed prior.
The pairs $(X_i,\theta_i)_{i=1}^n$ are shown in blue, where $\theta_i$ are iid copies of the square root of a standard Cauchy variable. 
The Bayes estimator corresponding to the true prior (oracle) and the learned NPMLE (computed using the solver in \cite{jana2022poisson}) are shown in black and green. The Robbins estimator is shown in orange.
}%
\label{fig:compare}%
\end{figure}

Compared to the methodological aspect, theoretical understanding on the Robbins estimator has been limited. In \cite{brown2013poisson, polyanskiy2021sharp}, the authors
studied its regret 
(\ref{def:reg_intro}) for nonparametric class 
$\calG$ of light-tailed priors (compactly supported or subexponential), and proved the surprising conclusion that the Robbins estimator
 achieves the optimal rates of regret with even the \emph{exact} logarithmic factors. 
This result is at odds with the aforementioned nonrobustness of the Robbins estimator (and $f$-modeling methods more generally) that has been widely recognized in practice.

In this paper, we obtain a general theory on $f$-modeling vs $g$-modeling in the Poisson model for the case of heavy-tailed data. To this end, we consider priors with moment constraints, a class previously studied for the Gaussian location model \cite{zhang1997empirical,genovese2000rates,ghosal2001entropies,zhang2009generalized,jiang2009general,kim2022minimax}. 
 This choice is motivated by the empirical observation that the Robbins estimator behaves poorly in the presence of outliers, which are abundant under heavy-tailed priors.
Specifically, for any real $M_p > 0$, consider the moment class
\begin{align}\label{def:mixture_class_intro}
\mathcal{G}_p(M_p) \equiv \{G \in \mathcal{P}(\R_+): ~ m_p(G) \leq M_p\}, \quad  \forall p > 0,
\end{align}
where $\mathcal{P}(\R_+)$ is the set of probability measures on $\R_+$, and $m_p(G) \equiv \int u^pG(\d u)$ is the $p$th moment of a distribution $G$ on $\reals_+$.
Next we give a summary of our main findings. 
To ease exposition, for the rest of the introduction, we shall consider $M_p = 1$ and abbreviate $\mathcal{G}_p(1)$ as $\mathcal{G}_p$.

\subsection{Optimality of $g$-modeling and suboptimality of $f$-modeling}

The main results of this paper are two-fold.
	\begin{itemize}
		\item 
	For $g$-modeling, we show that \emph{for any} rate-optimal (in Hellinger) proper density estimator, the corresponding $g$-modeling EB estimator, with a modicum of regularization, is guaranteed to achieve the optimal rate of regret (up to logarithmic factors). 
	
		\item For $f$-modeling, \emph{there exists} an $f$-modeling estimator whose density estimation rate is optimal but whose EB regret is suboptimal by a polynomial factor. 
	\end{itemize}
	These complementary results show that the \emph{proper} Bayes form is crucial and provides a ``general recipe of success'' for optimal EB estimation that applies to all $g$-modeling (but not $f$-modeling) methods. 

%
To provide more details, fix $p>1$. For any $g$-modeling method with estimated prior $\tilde{G}$ such that $f_{\tilde{G}}$ achieves the minimax Hellinger rate (up to logarithmic factors) of density estimation over $\mathcal{G}_p$, the associated Bayes estimator $\hat{\theta}^{\mathsf{g}} \equiv \theta_{\tilde G}$ given by (\ref{def:poi_bayes_form_intro}) with mild regularization achieves the following regret bound over $\mathcal{G}_p$ (see 
\eqref{def:g_smoothed_individual} for the definition of the regularized estimator and Theorem \ref{thm:regret_MLE_EB} for precise statements): 
\begin{align}\label{eq:npmle_rate_intro}
\totreg_n(\hat{\theta}^{\mathsf{g}};\mathcal{G}_p) = \tilde{O}\big(n^{\frac{3}{2p+1}}\big),
\end{align}
which is shown minimax optimal by Theorem \ref{thm:regret_lower}.
(Here $\tilde{O}(\cdot)$ and $\tilde{\Omega}(\cdot)$ hide polylogarithmic factors; see Section \ref{subsec:notation} for exact definitions.)
Furthermore, for the (important) special case of NPMLE (\ref{def:npmle_intro}), the optimal rate \prettyref{eq:npmle_rate_intro} is achieved without regularization (Theorem \ref{thm:npmle_rho_free}). See Sections \ref{subsec:reg_lower} and \ref{subsec:NPMLE} for a detailed discussion of related regret bounds in the literature. 
Turning to $f$-modeling, we first characterize the regret of the Robbins estimator (see Theorem \ref{thm:regret_robbins_EB} and \prettyref{eq:totreg-robbins} for precise statements): For any $p > 1$,
\begin{align}\label{eq:robbins_rate_intro}
\totreg_n(\hat{\theta}^{\rob};\mathcal{G}_p) = \tilde{\Theta}\big(n^{\frac{3}{p+1}}\big).
\end{align} 
Consequently, the Robbins estimator is inconsistent for $p \in (1,2)$. Furthermore, we show that a natural modification of the Robbins estimator (via interpolation with the MLE $Y^n$) achieves the regret bound $\tilde{\Theta}(n^{\frac{3}{p+2}})$, which is suboptimal by a polynomial factor for all $p>1$ (e.g., $n^{3/4}$ versus the optimal $n^{3/5}$ for $p = 2$). This deficiency is partly explained by the fact that Robbins uses the empirical estimator for $f_G$, whose worst-case Hellinger rate is $\tilde{\Omega}(n^{-\frac{p}{2(p+1)}})$ (see  Proposition \ref{prop:empirical_subopt}), which is strictly sub-optimal by polynomial factors (see (\ref{eq:density_rate_intro}) below). To draw a fair comparison with $g$-modeling, we then demonstrated a $f$-modeling estimator  which achieves the optimal Hellinger rate of density estimation but a strictly sub-optimal regret rate by a polynomial factor; see Theorem \ref{thm:hybrid} for details.  




\subsection{Poisson mixture density estimation}

Nonparametric estimation of mixture densities is a classical problem in statistics.
As an essential step toward the regret bound (\ref{eq:npmle_rate_intro}), we study the problem of estimating Poisson mixture with mixing distributions in the moment class (\ref{def:mixture_class_intro}). 
We show that the NPMLE achieves the following squared Hellinger risk (see Theorem \ref{thm:density_main} for precise statements): For all $p > 0$,
\begin{align}\label{eq:density_rate_intro}
\sup_{G\in\mathcal{G}_p} \E_G H^2(f_{\hat{G}}, f_G) = \tilde{O}(n^{-\frac{2p}{2p+1}}),
\end{align}
where $f_{\hat{G}}$ is the Poisson mixture (\ref{def:poi_mixture_intro}) induced by $\hat{G}$ in (\ref{def:npmle_intro}). 
This result is optimal up to logarithmic factors in view of the minimax lower bound in Theorem \ref{thm:density_lower}. We make the following comments on the rate (\ref{eq:density_rate_intro}), deferring a detailed discussion of the surrounding literature to Section \ref{sec:density_est}:
\begin{itemize}
\item While consistency in regret is only possible for $p > 1$, (\ref{eq:density_rate_intro}) shows that Hellinger consistency of Poisson mixture estimation is possible for all $p > 0$;
\item As discussed below, the Hellinger minimax rate of Gaussian mixture estimation over the class (\ref{def:mixture_class_intro}) is $\tilde{O}(n^{-p/(p+1)})$, which is slower than its Poisson counterpart in (\ref{eq:density_rate_intro});
\item A crucial difference between our regret bound (\ref{eq:npmle_rate_intro}) and the existing results \cite{jiang2009general, brown2013poisson, polyanskiy2021sharp, jana2022poisson} is that in all previously studied settings, the optimal rate of the amortized regret per observation (total regret divided by $n$) coincides with that of density estimation (in $H^2$) up to logarithmic factors; in comparison, (\ref{eq:npmle_rate_intro}) divided by $n$ and (\ref{eq:density_rate_intro}) differ by a polynomial order. This renders the previous reduction from regret to density estimation in, for example, \cite[Theorem 3]{jiang2009general} and \cite{jana2022poisson} not directly applicable and, as a result, the proof of (\ref{eq:npmle_rate_intro}) requires new techniques; 
see Section \ref{subsec:NPMLE} for a detailed discussion. 
\end{itemize}

\subsection{Notation}\label{subsec:notation}

For any positive integer $n$, let $[n]\equiv\{1,\ldots,n\}$. For $a,b \in \R$, $a\vee b\equiv \max\{a,b\}$ and $a\wedge b\equiv\min\{a,b\}$. For $a \in \R$, let $a_\pm \equiv (\pm a)\vee 0$. 
For any $x\in\R$ and $n\in\integers_+$, define the falling factorial $(x)_n \equiv x(x-1)\ldots (x-n+1)$. 
Let $\calP(\reals_+)$ denote the collection of all (Borel) probability measures on $\reals_+$.
For each $G \in \calP(\reals_+)$, let $\supp(G)$ denote its support.
Throughout the paper we adopt the convention $\theta^n \equiv (\theta_1,\ldots,\theta_n)$ for vectors, vector-valued functions, and random vectors.

We use standard asymptotic notation:
For positive sequences $a_n=a_n(x),b_n=b_n(x)$, we write $a_n\lesssim_{x} b_n$ and $b_n\gtrsim_x a_n$ 
(or $a_n= O_x(b_n)$ and $b_n=\Omega_x(a_n)$) if $a_n\leq C_x b$ for some constant $C_x>0$ depending only on $x$; $a_n\asymp_x b$ (or $a_n=\Theta_x(b)$) if both $a_n\lesssim_{x} b$ and $a_n\gtrsim_x b$ (the subscript $x$ is dropped is the constant $C$ is absolute constant);
$a_n= o(b_n)$ if $\lim_{n\rightarrow\infty} (a_n/b_n) = 0$;
$a_n=\poly(n)$ if $a_n=n^{O(1)}$;
$a_n=\polylog(n)$ if $a_n=(\log n)^{O(1)}$.
We will also use the tilde convention to hide polylogarithmic factors, e.g., $a_n=\tilde O(b_n)$ if $a_n= O(b_n \cdot \polylog(n))$.

For a two-sided sequence $\{f(y)\}_{y\in\mathbb{Z}}$, the \emph{forward difference} operator is recursively defined by
\begin{align}\label{eq:forwarddiff}
\Delta^kf(y) \equiv \Delta^{k-1}f(y+1) - \Delta^{k-1}f(y), \quad \Delta^{0}f(y) \equiv f(y),
\end{align}
and the \emph{backward difference} operator is defined by
\begin{align}\label{eq:backwarddiff}
\nabla^k f(y) \equiv \nabla^{k-1}f(y) - \nabla^{k-1}f(y-1), \quad \nabla^{0}f(y) \equiv f(y).
\end{align}
In particular, $\nabla f(y)=\Delta f(y-1)$. 
Expanding these recursive definitions leads to binomial-type expansions of higher-order finite differences, for example,  
\begin{equation}
\nabla^{k} f(y) =  
\sum_{i=0}^{k} (-1)^i \binom{k}{i} f(y-i).
\label{eq:backwarddiff-expand}
\end{equation}
For a one-sided sequence $\{f(y)\}_{y\in\mathbb{Z}_{+}}$, its forward and backward difference operations are understood as first extending the definition by $f(y) \equiv 0$ for all $y < 0$ and then applying the above definitions. 
Finally, recall the summation by parts formula: Provided that $f(-1)=0$,
\begin{align}
\sum_{y=0}^\infty f(y) \cdot \Delta g(y) = - \sum_{y=0}^\infty g(y)\cdot \nabla f(y).
\label{eq:SBP}
\end{align}


\subsection{Organization}
The rest of the paper is organized as follows. Section \ref{sec:density_est} contains results on the Poisson mixture density estimation, and our main results on the regret bounds are presented in Section \ref{sec:regret}. Some concluding remarks are in Section \ref{sec:conclusion}. All major proofs are collected in Sections \ref{sec:proof_density} and \ref{sec:proof_regret}, with some auxiliary results and proofs deferred to the appendices. 

\section{Estimation of Poisson mixture}\label{sec:density_est}
We start by formally introducing the density estimation framework. Let $Y^n = (Y_1,\ldots,Y_n)$ be i.i.d.~observations from the Poisson mixture $f_G$ in (\ref{def:poi_mixture_intro}),
where $G$ is some mixing distribution supported on $\reals_+$. 
We will mainly be interested in the set of mixing distributions defined in (\ref{def:mixture_class_intro}).
For any estimator $\hat{f}$ that is a valid probability mass function, its squared Hellinger error for estimating $f_G$ is
\begin{align}\label{def:hellinger}
H^2(\hat{f}, f_G) \equiv \sum_{y=0}^\infty \big(\sqrt{\hat{f}(y)} - \sqrt{f_G(y)}\big)^2.
\end{align}

We will be chiefly concerned with the nonparametric MLE (NPMLE) \cite{kiefer1956consistency}, defined by
\begin{align}\label{def:npmle}
\hat{G} \equiv \argmax_{G \in \calP(\reals_+)} \prod_{i=1}^n f_G(Y_i).
\end{align}
It is well-known that for the Poisson mixture model, \eqref{def:npmle} has a unique solution with at most $n$ atoms \cite{simar1976maximum}. We refer the readers to the monograph \cite{lindsay1995mixture}  for a systematic treatment of the NPMLE for general exponential families in one dimension and 
\cite{PW20-npmle} for more recent results. 

The following result, proved in
Section \ref{subsec:proof_density_upper}, provides a large-deviations inequality for the Hellinger risk of the NPMLE in density estimation.
\begin{theorem}\label{thm:density_main}
Suppose $Y^n = (Y_1,\ldots,Y_n)$ are i.i.d.~observations from $f_{G}$, where $G\in\mathcal{G}_p(M_p)$ for some $p > 0$ and $M_p^{1/p} \leq n^{10}$. 
Let 
\begin{align}\label{def:density_epsilon_rate}
\epsilon_n \equiv \big(n^{-p/(2p+1)}M_p^{1/(4p+2)} \vee n^{-1/2}\big)(\log n)^4.
\end{align}
Then there exists some $t_* = t_\ast(p)$ such that for all $t\geq t_*$,
\begin{align}\label{ineq:density_upper_main}
\Prob_G\Big(H(f_{\hat{G}}, f_{G}) \geq t\epsilon_n\Big) \leq 2\exp\big(-t^2n\epsilon_n^2/(8\log n)\big) \leq 2\exp(-t^2(\log n)^2/8).
\end{align}
where $\hat{G}$ is the NPMLE in (\ref{def:npmle}). 
Consequently, there exists some $C = C(p) > 0$ such that $\E_{G} H^2(f_{\hat{G}},f_G) \leq C\epsilon_n^2$ uniformly over $G\in \mathcal{G}_p(M_p)$.
\end{theorem}
\begin{remark}
The upper bound condition $M_p^{1/p} \leq n^{10}$ can be strengthened to $M_p^{1/p} \leq n^{\eta}$ for any $\eta = O(1)$, with $t_\ast$ now depending on $\eta$ as well; see Remark \ref{remark:density_lower} below for some related discussion.
\end{remark}

\begin{remark}
A natural question is whether the empirical estimator $\hat{f}^{\emp}(y) \equiv n^{-1}\sum_{i=1}^n \bm{1}_{Y_i = y}$ can achieve the same Hellinger rate. As shown in Proposition \ref{prop:empirical_subopt} in the appendix, the answer is negative. Note that this is in stark contrast to the light-tailed case (i.e., $G$ has bounded support or sub-exponential tail), where the Poisson structure becomes irrelevant and the empirical estimator is already rate-optimal down to exact logarithmic factors \cite{polyanskiy2021sharp}. 
\end{remark}

The next result provides a matching minimax lower bound, proved in Section \ref{subsec:proof_density_lower} based on a construction inspired by the proof of \cite[Theorem 2.3]{kim2022minimax}.
\begin{theorem}\label{thm:density_lower}
For any $p>0$, there exists some $c = c(p) > 0$ such that
\begin{align*}
\inf_{\hat{f}}\sup_{G\in \mathcal{G}_p(M_p)} \E_{G}H^2(\hat{f}, f_{G}) \geq c n^{-2p/(2p+1)}M_p^{1/(2p+1)}(\log n)^{-11}
\end{align*}
provided that $n^{-1/p}(\log n)^{10} \leq M_p^{1/p} \leq n^2(\log n)^2$, where the infimum is taken over all density estimate $\hat f$ measurable with respect to $Y^n\iiddistr f_G$.
\end{theorem}
\begin{remark}\label{remark:density_lower}
For the above lower bound to hold, the assumption for the form $M_p^{1/p} = \tilde{O}(n^2)$ cannot be removed because the Hellinger distance is at most a constant.
\end{remark}

Theorems \ref{thm:density_main} and \ref{thm:density_lower} together determine, subject to some mild assumptions on $M_p$, the minimax rate of estimating Poisson mixture density over the moment class $\mathcal{G}_p(M_p)$ up to logarithmic factors:
\begin{align}\label{ineq:density_minimax_rough}
\inf_{\hat{f}}\sup_{G \in \mathcal{G}_p(M_p)} \E_{G}H^2(\hat{f}, f_{G}) = \tilde\Theta(n^{-\frac{2p}{2p+1}}M_p^{\frac{1}{2p+1}}).
\end{align}
This result is of independent interest, and also plays a crucial role in proving the regret optimality of NPMLE in Section \ref{sec:regret}.

Next we discuss the connection of the minimax rate (\ref{ineq:density_minimax_rough}) to the surrounding literature. Instead of surveying the large collection of results on mixture density estimation and the NPMLE, 
we will only review an incomplete list of results most related to ours. After its original introduction in \cite{kiefer1956consistency}, early results on the consistency of the NPMLE were obtained in \cite{jewell1982mixtures, heckman1984method, pfanzagl1988consistency}, among others; see also \cite{chen2017consistency} for a recent review. More recently, driven by the development of empirical process theory, mixture density estimation via the nonparametric/sieve MLE was studied in \cite{vdgeer1993hellinger, shen1994convergence, wong1995probability, vdgeer1996rates} under generic entropy conditions, and in \cite{genovese2000rates, ghosal2001entropies, ghosal2007posterior, zhang2009generalized, kim2014minimax} specifically under the Gaussian mixture model; see also \cite{saha2020nonparametric} for a multivariate extension. The state of the art on estimating nonparametric Gaussian mixture densities is \cite[Theorem 1]{zhang2009generalized} which considered priors with both light (compactly supported or subgaussian) and heavy tails (moment class):
\begin{align}\label{ineq:gau_density_rate}
\sup_{G\in \mathcal{G}} \E_{G}H^2(f_{\hat{G}}, f_G) \lesssim_p
\begin{cases}
n^{-1}(\log n)^2 & \mathcal{G} = \{G:\supp(G)\subset[-1,1]\} \text{ or } \{G:\int e^{cu^2}G(\d u)\leq 1\}\\
n^{-\frac{p}{p+1}}(\log n)^{\frac{2+3p}{2+2p}} & \mathcal{G} = \mathcal{G}_p(1),\quad \forall p > 0,
\end{cases}
\end{align}
where $c>0$ is a constant. Here, we overloaded the notation to also use $f_G$ to denote the Gaussian mixture density under prior $G$ (convolution between $G$ and standard normal), and $f_{\hat{G}}$ is the mixture density induced by the Gaussian analogue of the NPMLE (\ref{def:npmle}). Up to logarithmic factors, both bounds in (\ref{ineq:gau_density_rate}) are known to be minimax optimal
 \cite{kim2014minimax, kim2022minimax}.
 For the related problem of estimating a \emph{finite} Gaussian mixture density in both fixed and high dimensions, we refer to the works \cite{suresh2014near, ho2016convergence, HK2015, li2017robust,wu2020optimal, doss2020optimal} and the references therein.

In comparison, density estimation under the Poisson mixture model is less studied. 
 Assuming that the prior $G$ has a bounded support, \cite{lambert1984asymptotic} derived a near-parametric rate for the NPMLE under (a variation of) the $\chi^2$-divergence. More recently, \cite{jana2022poisson} studies the performance of the NPMLE when $G$ has a light tail, and obtains the following bounds for light-tailed (compactly-supported and subexponential) priors:
\begin{align*}
\sup_{G\in \mathcal{G}} \E_{G}H^2(f_{\hat{G}}, f_G) \leq C
\begin{cases}
n^{-1}\frac{\log n}{\log \log n} & \mathcal{G} = \big\{G\subset\reals_+:\supp(G)\subset[0,1]\big\}\\
n^{-1}\log n & \mathcal{G} = \big\{G\subset\reals_+:\int e^{cu}G(\d u)\leq 1\big\},
\end{cases}
\end{align*}
where $c > 0$ is a constant and $f_{\hat{G}}$ is the Poisson mixture induced by the NPMLE (\ref{def:npmle}). Both upper bounds are minimax rate-optimal with exact logarithmic factors \cite[Theorem 21]{polyanskiy2021sharp}. 
Complementing this result, (\ref{ineq:density_minimax_rough}) resolves the minimax rate for moment classes up to logarithmic factors.

We close this section with a brief discussion of the technical innovation in the proof of Theorem \ref{thm:density_main}. Following the analysis of NPMLE based on covering entropy in \cite{ghosal2001entropies, ghosal2007posterior, zhang2009generalized}, the key step of the proof is to obtain a tight entropy bound for the mixture class under moment constraint  (\ref{def:mixture_class_intro}) under a truncated $\ell_\infty$-norm, which in turn relies on a discrete approximation of an arbitrary mixing distribution $G$ on $\reals_+$. To this end, our main technical contribution is the following result (see Lemma \ref{lem:moment_matching} for precise statements): For small $\eta > 0$ and large $M$ that is at least $\polylog(1/\eta)$,
there exists a discrete distribution $G_m$ supported on $[0,2M]$ with at most $m = \tilde{O}(\sqrt{M})$ atoms (here ``$\tilde{O}(\cdot)$" hides $\polylog(1/\eta)$ factors), such that
\begin{align}\label{ineq:discrete_approx_rough}
\pnorm{f_G - f_{G_m}}{\infty, M} \equiv \max_{x = 0,\ldots,M} |f_G(x) - f_{G_m}(x)| \leq \eta,
\end{align}
where $f_{G_m}$ is the $m$-component Poisson mixture induced by $G_m$.
The above bound is then applied  in Lemma \ref{lem:entropy} to obtain a tight entropy estimate of the mixture class induced by $\calP(\reals_+)$ . 

The main strength of the bound (\ref{ineq:discrete_approx_rough}) is that over the approximation range $[0,M]$, a discrete distribution with only $\tilde{O}(\sqrt{M})$ atoms is sufficient, while the Gaussian analogue of (\ref{ineq:discrete_approx_rough}) requires $\tilde{O}(M)$ atoms \cite[Lemma 1]{zhang2009generalized}. This difference leads to the faster rate in (\ref{ineq:density_minimax_rough}) compared to the Gaussian rate in (\ref{ineq:gau_density_rate}).
An intuitive explanation is that the Poisson density $\poi(\cdot;\theta)$ resembles locally the density of $\mathcal{N}(\theta, \theta)$ (as opposed to $\mathcal{N}(\theta, 1)$ in the Gaussian location model), so that for large $\theta$, it is possible to reach the same approximation accuracy by matching less moments thanks to the extra ``blurring'' incurred by a large variance. More precisely, (\ref{ineq:discrete_approx_rough}) is proved choosing $G_m$ to match the first $O((\log (1/\eta))^2)$ moments of $G$ locally over each interval of the following \emph{quadratic} partition of $[0,2M]$:
\begin{align}\label{def:quad_scale_partition}
I_i \equiv \big[i^2C\log(1/\eta), (i+1)^2C\log(1/\eta) \wedge 2M\big), \quad 0\leq i\leq N = \tilde{O}(\sqrt{M}).
\end{align}
(See the proof of Lemma \ref{lem:moment_matching} for details.) In contrast, if we follow a linear partition $\big[iC\log(1/\eta), (i+1)C\log(1/\eta) \wedge 2M\big)$ as previously used in \cite[Lemma 1]{zhang2009generalized}, the resulting $G_m$ will again have $\tilde{O}(M)$ atoms. As explained previously, the quadratic scaling in (\ref{def:quad_scale_partition}) is tailored for Poisson distributions (whose variance equals to the mean); a similar partition is also adopted in the lower bound construction of Theorem \ref{thm:density_lower}. Incidentally, this quadratic scaling has previously been used in \cite{han2018local} for estimating distributions and their functionals on large domains based on Poissonized sampling.


\section{Regret bound}\label{sec:regret}

\subsection{Preliminary}\label{subsec:regret_prelim}
As discussed in the Introduction, in the Poisson EB model, our goal is to estimate the Poisson means $\theta^n$ based on the observations $Y^n$ and compete with the Bayes oracle. For any estimator $\hat{\theta}^n: \mathbb{Z}^n_+ \rightarrow \R_+^n$, its performance is measured by the total regret in (\ref{def:reg_intro}).
It turns out that for analysis it will be more convenient to work with the closely related notion of \emph{individual regret} \cite{polyanskiy2021sharp}, formally defined as
\begin{align}\label{def:individual_regret}
\reg_n(\hat{\theta};\mathcal{G}) \equiv \sup_{G\in\mathcal{G}}\Big\{\E_{G}\big(\hat{\theta}(Y^n) - \theta_n\big)^2 - \mathsf{mmse}(G)\Big\}.
\end{align}
where $\hat{\theta}: \integers_+^n\to\reals_+$ is a scalar estimator for $\theta_n$. 
The individual regret (\ref{def:individual_regret}) can be interpreted from the perspective of training/testing data: One may view $Y^{n-1}=(Y_1,\ldots,Y_{n-1})$ as the ``training sample" from which we learn a scalar-valued estimator $\hat{\theta}(Y^{n-1},\cdot)$, and then apply it to the fresh observation $Y_n$ to estimate its mean $\theta_n$. 

For permutation-invariant $\hat{\theta}^n$, i.e., 
\begin{equation}
(\hat{\theta}_1(Y_{\sigma(1)},\ldots, Y_{\sigma(n)}),\ldots,\hat{\theta}_n(Y_{\sigma(1)},\ldots, Y_{\sigma(n)})) = \big(\hat{\theta}_{\sigma(1)}(Y^n),\ldots, \hat{\theta}_{\sigma(n)}(Y^n)\big)
\label{eq:PI}
\end{equation}
for any permutation $\sigma$ of $[n]$, 
it follows from symmetry that
\begin{align}\label{eq:reg_relation_est}
\totreg_n(\hat{\theta}^n;\mathcal{G}) = n \cdot \reg_n(\hat{\theta}_n;\mathcal{G})
\end{align}
where $\hat{\theta}_n$ is the last coordinate of $\hat{\theta}^n$.
In what follows, we will mainly work with the individual regret due to its natural connection to function estimation: with $\theta_G(\cdot)$ the Bayes estimator defined in (\ref{def:poi_bayes_form_intro}), 
\begin{align}\label{eq:reg_function_estimation}
\notag\reg_n(\hat{\theta};\mathcal{G}) &= \sup_{G\in\mathcal{G}}\E_{G}\big(\hat{\theta}(Y_n; Y^{n-1}) - \theta_G(Y_n)\big)^2\\
&=\sup_{G\in\mathcal{G}}\E_{G}\bigpnorm{\hat{\theta}(\cdot; Y^{n-1}) - \theta_G}{\ell_2(f_G)}^2,
\end{align}
where the first identity follows from the orthogonality property of the Bayes estimator, and for any sequence $f$ and pmf $P$ on $\integers_+$, $\pnorm{f}{\ell_2(P)}^2 \equiv \sum_{x \geq 0} f^2(x)P(x)$. 
In other words, the individual regret \prettyref{eq:reg_function_estimation} is precisely the squared error (weighted by the true density $f_G$) of estimating the Bayes rule $\theta_G(\cdot)$ based on $n-1$ i.i.d.~observations. Analogous to the total regret, we say an estimator $\hat{\theta}$ is consistent in estimating $\theta_n$ if $\reg_n(\hat{\theta};\mathcal{G}) = o(1)$.

The fundamental limits of Poisson EB estimation under the two regrets in (\ref{def:reg_intro}) and (\ref{def:individual_regret}) are defined by their minimax analogues:
\begin{align}\label{def:regret_minimax}
\notag\totreg_n(\mathcal{G}) &\equiv \inf_{\hat{\theta}^n}\sup_{G\in\mathcal{G}}\Big\{\E_{G}\bigpnorm{\hat{\theta}^n(Y^n) - \theta^n}{}^2 - n \cdot \mathsf{mmse}(G)\Big\},\\
\reg_n(\mathcal{G}) &\equiv \inf_{\hat{\theta}_n}\sup_{G\in\mathcal{G}}\Big\{\E_{G}\big(\hat{\theta}_n(Y^n) - \theta_n\big)^2 - \mathsf{mmse}(G)\Big\},
\end{align}
where the infimum is taken over estimators measurable with respect to $Y^n\iiddistr f_G$. 
As shown in \cite[Lemma 5]{polyanskiy2021sharp}, the minimax total and individual regrets are in fact 
related by the following identity: for any class $\calG$ of priors,
\begin{align}\label{eq:reg_identity}
\totreg_n(\mathcal{G}) = n\cdot \reg_n(\mathcal{G}).
\end{align}

In the remainder of this section, we study the regret of general $g$- and $f$-modeling methods and determine their optimality and suboptimality by deriving minimax regret bounds.
 
\subsection{Minimax lower bound}\label{subsec:reg_lower}

We first give a minimax lower bound for $\reg_n$ for the prior class in (\ref{def:mixture_class_intro}) with a $p$th moment constraint. Its proof can be found in Section \ref{subsec:proof_regret_lower}.

\begin{theorem}\label{thm:regret_lower}
For any $p \geq 1$, there exists some $c_p >0$ such that the following holds.
\begin{itemize}
	\item For any $p>1$ and $n^{-1/p}(\log n)^{10} \leq M_p^{1/p} \leq n^2(\log n)^2$,
	\begin{align*}
\reg_n\big(\mathcal{G}_p(M_p)\big) \geq c_p n^{-2(p-1)/(2p+1)}M_p^{3/(2p+1)}(\log n)^{-11}.
\end{align*}
	\item For $p=1$,
	\begin{align*}
\reg_n\big(\mathcal{G}_1(M_1)\big) \geq c_1 M_1.
\end{align*}
\end{itemize}
\end{theorem}
\begin{remark}\label{remark:bayes_infinity}
The regret (\ref{def:regret_minimax}) for the moment class $\calG_p(M_p)$ is only well-defined for $p\geq 1$ in the sense that, for any $p<1$ and $M_p > 0$, there exists a prior $G$ with $m_p(G)\leq M_p$ such that the Bayes rule $\theta_G$ in (\ref{def:poi_bayes_form_intro}) is well defined, but
the Bayes risk (and thus the risk of any estimator) is infinite:
\begin{align}\label{eq:bayes_risk_infty}
\mathsf{mmse}(G) = \E_G \big(\theta_G(Y) - \theta\big)^2 = \infty.
\end{align}
See \prettyref{app:mmseinf} for a proof.
%
\end{remark}
%

With a matching upper bound (up to logarithmic factors) of $\reg_n$ in Theorem \ref{thm:regret_MLE_EB} below,  
Theorem \ref{thm:regret_lower} shows an interesting elbow phenomenon for the individual regret at $p=1$:
\begin{itemize}
\item If $p < 1$, the regret is not well-defined as 
the Bayes risk is infinite for certain priors;
\item If $p = 1$, the optimal rate of $\reg_n$ scales with $M_1$ and does not vanish with $n$, which means consistent estimation of a single parameter (taken to be $\theta_n$ in the formulation of $\reg_n$) is impossible as long as $M_1$ does not vanish. Consequently, the MLE $Y$ (or more precisely, $Y_n$ when estimating $\theta_n$), which always satisfies the risk bound $\E_G (Y - \theta)^2 = m_1(G) \leq M_1$, is already minimax rate-optimal;
\item If $p > 1$, the optimal regret decays polynomially in $n$.
As will be shown in the next two sections, a modified version of the NPMLE-based EB estimator achieves this optimal rate, while the Robbins estimator is strictly rate suboptimal. 
\end{itemize}

Let us now discuss the connection of Theorem \ref{thm:regret_lower} to the existing literature. In the seminal paper \cite{jiang2009general}, the Gaussian analogue of the EB model was studied in detail. With $\reg^g_n(\mathcal{G})$ denoting the Gaussian analogue of the individual regret defined in (\ref{def:regret_minimax}) (see Appendix \ref{sec:gau_lower_bound} for precise definitions), it was proved there that
\begin{align}\label{ineq:gau_regret_upper}
\reg^g_n(\mathcal{G}) \leq C\cdot
\begin{cases}
n^{-1}(\log n)^5, & \mathcal{G} = \big\{G:\supp(G)\subset[-1,1] \text{ or } \int e^{cu^2}G(\d u)\leq 1\big\},\\
n^{-\frac{p}{p+1}}(\log n)^{\frac{9p+8}{2p+2}} & \mathcal{G} = \mathcal{G}_p(1), \quad\forall p > 0,
\end{cases}
\end{align}
where $c > 0$ is universal, and $C > 0$ only depends on $p$. In words, the first case of (\ref{ineq:gau_regret_upper}) studies the light-tailed setup (i.e., $G$ has a bounded support/subgaussian tail), and the second case studies the heavy-tailed setup. In addition to the EB setting, \cite[Theorem 5]{jiang2009general} also extended the bounds in (\ref{ineq:gau_regret_upper}) to the so-called compound setting with a slightly modified metric; we refer to Section \ref{sec:compound} (Theorem \ref{thm:npmle_compound}) for detailed definitions and results for the Poisson model in the compound setup. Up to logarithmic factors, the bound $n^{-1}(\log n)^5$ in the first case of (\ref{ineq:gau_regret_upper}) has been shown by \cite[Theorem 1]{polyanskiy2021sharp} to be minimax optimal. By adapting the proof of Theorem \ref{thm:regret_lower}, we show in Theorem \ref{thm:gau_regret_lower}  that the second case of (\ref{ineq:gau_regret_upper}) is also minimax optimal up to logarithmic factors, thereby settling the optimality of (\ref{ineq:gau_regret_upper}) in the Gaussian EB model. It is worth noting that, in the Gaussian EB model, consistency of the individual regret is possible for all $p > 0$, as opposed to the threshold $p > 1$ in the Poisson case.

In the Poisson EB model, the optimal rate of the individual regret was known for compactly supported or subexponential priors, where \cite[Theorem 2]{polyanskiy2021sharp} shows that the minimax individual regrets are $\Theta\big(n^{-1}(\log n/\log\log n)^2\big)$ and $\Theta\big(n^{-1}(\log n)^3\big)$, respectively. As a result, for these 
light-tailed priors, the total excess risk for estimating the $n$ parameters compared with the Bayes oracle is merely $\polylog(n)$. 
In contrast, Theorem \ref{thm:regret_lower} shows that for the heavy-tailed case of moment classes, the total regret is at least $\poly(n)$ which is tight as shown in the next section. 
Finally, we note that in all previous results, the optimal rate of the individual regret coincides with that of density estimation under $H^2$ (see the discussion after Theorem \ref{thm:density_lower}) up to logarithmic factors;
this, after all, is not a universal phenomenon, as we show in this paper (comparing (\ref{ineq:density_minimax_rough}) and Theorem \ref{thm:regret_lower}).



\subsection{Positive results on $g$-modeling}\label{subsec:NPMLE}

In this section, we study the performance of a general $g$-modeling approach (with appropriate regularization) for EB estimation. Specifically, for any $\rho \geq 0$ and prior distribution $G$, let the regularized Bayes rule be
\begin{align}\label{def:reg_bayes}
\theta_G(y;\rho) \equiv (y+1)\Big(\frac{\Delta f_G(y)}{f_G(y) \vee \rho} + 1\Big),
\end{align}
where $\Delta f_G(y) \equiv f_G(y+1) - f_G(y)$ denotes the forward difference per \prettyref{eq:forwarddiff}. Clearly, $\theta_G(\cdot;\rho)$ reduces to the Bayes rule $\theta_G(\cdot)$ in (\ref{def:poi_bayes_form_intro}) when $\rho = 0$.

Following the interpretation of the individual regret (\ref{def:individual_regret}), given $n$ i.i.d. observations $Y_1,\ldots,Y_n$ from $f_G$, a generic $g$-modeling approach typically produces an estimator $H$ of the true $G$ from $Y^{n-1}$, which we then apply to $Y_n$ to produce an estimate for $\theta_n$:
\begin{align}\label{def:g_smoothed_individual}
\hat{\theta}_n^{\mathsf{g}}(Y_n;H,\rho) \equiv \hat{\theta}_n^{\mathsf{g}}(Y_n;Y^{n-1}, H,\rho) \equiv \theta_H(Y_n;\rho).
\end{align}


The following result, whose proof is given in Section \ref{subsec:proof_NPMLE}, bounds the regret of \eqref{def:g_smoothed_individual} uniformly for priors with bounded $p$th moment. In view of the impossibility result in Theorem \ref{thm:regret_lower}, we focus on the case of $p > 1$.
\begin{theorem}\label{thm:regret_MLE_EB}
Fix any $p > 1$. Let $H$ be any (random) distribution on $\R_+$ that only depends on $Y^{n-1}$. For any real $\rho > 0$ and integer $y_0 \geq 1$, let 
\begin{align*}
\mathcal{R}(y_0, \rho) \equiv M_p^{1/p}\exp(-c_0y_0) + y_0\rho^{10} + y_0^2\rho\log^2(1/\rho),
\end{align*}
where $c_0 > 0$ is some universal constant. There exists some universal $K > 0$ such that
\begin{align}\label{ineq:g_main_bound}
\E_{Y_n \sim f_G} \big(\hat{\theta}_n^{\mathsf{g}}(Y_n;H,\rho) - \theta_G(Y_n)\big)^2 \leq K \cdot \inf_{y_0 \geq 1} \bigg[\log^4(1/\rho)\Big(M_p y_0^{-(p-1)} + y_0 H^2(f_G, f_H)\Big) + \mathcal{R}(y_0, \rho)\bigg].
\end{align}
Consequently, if $H$ satisfies $\E_{Y^{n-1}\stackrel{i.i.d.}{\sim} f_G} H^2(f_G, f_H) \leq c_1\big(n^{-2p/(2p+1)}M_p^{1/(2p+1)} \vee n^{-1}\big)(\log n)^{\kappa}$ for some positive $c_1$ and $\kappa$ uniformly over $G\in \mathcal{G}(M_p)$, then upon choosing $\rho = c_2n^{-10}$ for some universal $c_2 > 0$, there exists some universal $C = C(c_1,c_2) > 0$ such that
the estimator in (\ref{def:g_smoothed_individual}) satisfies
\begin{align}
\reg_n\Big(\hat{\theta}^{\mathsf{g}}_n; \mathcal{G}_p(M_p)\Big) \leq C\big(n^{-\frac{2(p-1)}{2p+1}}M_p^{\frac{3}{2p+1}} \vee n^{-1}\big)(\log n)^{\kappa+4}.
\label{eq:regret_MLE_EB}
\end{align}
\end{theorem}

\begin{remark}
\label{rmk:totreg-npmle}
The individual regret bound of Theorem \ref{thm:regret_MLE_EB} can be translated to total regret (\ref{def:reg_intro}) as follows.
For $i\in[n]$, let $\hat{\theta}^{\mathsf{g}}_i(Y^n) = \theta_{H_{(i)}}(Y_i; \rho)$ be defined per (\ref{def:reg_bayes}), where $H_{(i)}$ is an estimator of $G$ trained from the sample $Y_{(i)} = Y^n \backslash Y_i$.
Let
\begin{align*}
\hat{\theta}^{\mathsf{g}, n}(Y^n) \equiv \big(\hat{\theta}^{\mathsf{g}}_1(Y^n),\ldots,\hat{\theta}^{\mathsf{g}}_n(Y^n)\big).
\end{align*}
It is easy to see that this estimator is permutation invariant in the sense of \prettyref{eq:PI}, so combining (\ref{eq:reg_relation_est}) and Theorem \ref{thm:regret_MLE_EB} yields 
\begin{align*}
\totreg_n(\hat{\theta}^{\mathsf{g}, n};  \mathcal{G}_p(M_p)) \leq C(n^{\frac{3}{2p+1}}M_p^{\frac{3}{2p+1}} \vee 1)(\log n)^{\kappa+4},
\end{align*}
whenever $\E_{Y_{(i)} \stackrel{i.i.d.}{\sim} f_G} H^2(f_{H_{(i)}}, f_G) \leq c_1\big(n^{-2p/(2p+1)}M_p^{1/(2p+1)} \vee n^{-1}\big)(\log n)^{\kappa}$ for all $i\in [n]$ and $G\in \mathcal{G}_p(M_p)$.

\end{remark}

In view of the regret minimax lower bound in Theorem \ref{thm:regret_lower} and the density estimation results in Theorem \ref{thm:density_lower}, Theorem \ref{thm:regret_MLE_EB} implies that for a generic $g$-modeling approach, as long as the estimated prior 
 $H$ used therein is Hellinger rate-optimal (up to logarithmic factors) in terms of density estimation, then it is also regret rate-optimal (up to logarithmic factors). Thanks to Theorem \ref{thm:density_main}, a concrete example in this category is the NPMLE. In fact, as we show below, the EB estimator based on the NPMLE (\ref{def:npmle}) trained on the whole dataset $Y^n$ also achieves the optimal regret without explicit regularization. Let
\begin{align}\label{def:npmle_eb}
\hat{\theta}^{\npmle,n}(Y^n) \equiv \big(\theta_{\hat{G}}(Y_1),\ldots, \theta_{\hat{G}}(Y_n)\big),
\end{align}
where we emphasize again that $\hat{G}$ is defined in (\ref{def:npmle}) via the entire $Y^n$. The proof the following result is given in Section \ref{subsec:proof_npmle_special}.
\begin{theorem}\label{thm:npmle_rho_free}
Suppose $M_p^{1/p} \leq n^{10}$. Then for some universal $C>0$ it holds that
\begin{align*}
\totreg_n(\hat{\theta}^{\npmle,n}; \mathcal{G}_p(M_p)) \leq C(\log n)^{13}(n^{\frac{3}{2p+1}}M_p^{\frac{3}{2p+1}} \vee 1).
\end{align*}
\end{theorem}


We now discuss the connection of Theorems \ref{thm:regret_MLE_EB} and \ref{thm:npmle_rho_free} to existing regret bounds for the NPMLE:
\begin{itemize}
\item (Poisson model) \cite{jana2022poisson} showed that when the prior $G$ is either compactly supported or subexponential, 
simple NPMLE 
without truncation or regularization achieves the optimal regret with exact logarithmic factors\footnote{For compact support, one needs to use the support-constrained NPMLE solution. } ; see also \cite{park2022poisson} for some slightly weaker guarantees in the bounded prior case.
Furthermore, compared to the celebrated Robbins' estimator (\ref{def:robbins_intro}) which is also rate-optimal in these two cases, the NPMLE and other minimum-distance estimators are shown to exhibit numerically a much more stable finite-sample performance; see the next section for a detailed study of Robbins' estimator. 
\item (Gaussian model) EB estimation in the Gaussian location model is studied in detail in the seminal work \cite{jiang2009general}; see Appendix \ref{sec:gau_lower_bound} for the exact model.
With the Gaussian counterpart of the individual regret (\ref{def:individual_regret}) denoted by $\reg_n^g(\cdot;\mathcal{G})$, \cite[Theorem 3]{jiang2009general} combined with the density estimation guarantees in \cite{zhang2009generalized} showed that the Gaussian analogue $\hat{\theta}^{\npmle,g}_n$ of (\ref{def:npmle_eb}) achieves the regret bounds in (\ref{ineq:gau_regret_upper}),
which are known to be minimax optimal up to polylogarithmic factors by the discussion thereafter. 

\end{itemize}

Next we comment briefly on the technical innovations required for proving Theorem \ref{thm:regret_MLE_EB} in comparison to existing regret analysis. 
The first major technical result, which is repeatedly used in our analysis of $g$-modeling methods and may also be of independent interest, is the following bound (cf.~Lemma \ref{lem:bayes_form_upper}) on the pointwise fluctuation of the Bayes estimator (\ref{def:poi_bayes_form_intro}): 
For \emph{any} prior $G$,
\begin{align}
|\theta_G(y) - y| \leq \E(|\theta - Y||Y = y) \lesssim \sqrt{y\vee 1}\log \frac{1}{f_G(y)}, \quad \forall y\geq 0.
\label{eq:centeredBayes}
\end{align}
In fact, for the Gaussian model the counterpart of \prettyref{eq:centeredBayes} holds without the $\sqrt{y}$ factor \cite[Lemma A.1]{jiang2009general}
; however, for the Poisson model this is tight.\footnote{To see this, simply consider the special case of $G=\delta_\lambda$ and $y = \lambda + C\sqrt{\lambda}$ for large $\lambda$ and constant $C$. In this case, by Stirling approximation both sides of \prettyref{eq:centeredBayes} agree up to a $\log \lambda$ factor.} 
This $\sqrt{y}$ factor is chiefly responsible for the different rates for the regret in the Poisson model (Theorem \ref{thm:regret_MLE_EB}) and that for the Gaussian model (Equation (\ref{ineq:gau_regret_upper})); see (\ref{ineq:npmle_main}) in the proof for details. 

The second (and much more difficult) step is to obtain the following comparison result (cf.~Proposition \ref{prop:A1_bound} for details), which relates the main term in the regret bound to the Hellinger risk of density estimation: for any two distributions $G_1,G_2$,
\begin{align}\label{ineq:A1_bound_rough}
\sum_{y=0}^{y_0} (y+1)^2\Big(\Delta f_{G_1}(y) - \Delta f_{G_2}(y)\Big)^2 \lesssim y_0\cdot H^2(f_{G_1}, f_{G_2})
+ \text{ negligibly small terms},
\end{align}
where $\Delta f_G(y) = f_G(y+1) - f_G(y)$ is the forward difference. Since $y_0$ will inevitably be chosen to be an appropriate polynomial of $n$, (\ref{ineq:A1_bound_rough}) explains a crucial difference between our regret bound in Theorem \ref{thm:regret_MLE_EB} and the previous regret analysis \cite{jiang2009general, brown2013poisson, polyanskiy2021sharp, jana2022poisson}: in all previously studied settings, the optimal rate of regret and density estimation (under $H^2$) only differ by polylogarithmic factors, while the rates in Theorems \ref{thm:density_main} and \ref{thm:regret_MLE_EB} differ by polynomial factors. 

Our proof of \prettyref{ineq:A1_bound_rough} is influenced by the seminal work of Jiang and Zhang in the Gaussian model \cite{jiang2009general}. Therein, 
to analyze the regret of NPMLE in the Gaussian model, they proved an inequality analogous to
\prettyref{ineq:A1_bound_rough} involving the derivative of the mixture density (see \cite[Lemma 1]{jiang2009general}), by means of a recursive argument of using higher-order derivatives to control the first derivative. 
Directly porting this program to the Poisson model, e.g, replacing the first-order forward difference in \prettyref{ineq:A1_bound_rough} with higher-order ones, does not work and more involved arguments are thus needed. The detailed proof is given in \prettyref{subsec:proof_NPMLE},
which constitutes the technical core of the paper.

Let us also remark that the proof technique in \cite{jana2022poisson} for light-tailed priors is not applicable to the current heavy-tailed setting.
In \cite[Lemma 4]{jana2022poisson}, the reduction from regret to density estimation is achieved via a simple truncation argument using the sample maximum $Y_{\max}=\max(Y_1,\ldots,Y_n)$, which in turn bounds the support of the NPMLE solution $\hat{G}$ \cite{simar1976maximum}. For priors with only moment constraint, bounding the learned Bayes estimator $\theta_{\hat{G}}(\cdot)$ by 
$Y_{\max}$ is too crude compared to the desired (\ref{eq:centeredBayes}), and the reduction from regret to density estimation is achieved by much more delicate arguments including (\ref{ineq:A1_bound_rough}).


Moving on to Theorem \ref{thm:npmle_rho_free}, the key reason that regularization can be removed is that the mixture density $f_{\hat{G}}$ with $\hat{G}$ trained on the entire $Y^n$ is automatically lower bounded at each $Y_i$. For the Gaussian model, such lower bound is $\tilde{\Omega}(n^{-1})$ as first observed by \cite{jiang2009general}. The situation for the Poisson is more complicated, as such lower bound is typically $\tilde{\Omega}(n^{-1} (Y_i\vee 1)^{-1/2})$ (see Lemma \ref{lem:density_lower_npmle}), and hence some careful truncation arguments have to be applied. The other major technical component of Theorem \ref{thm:npmle_rho_free} is some properly defined notion of (total) regret in the compound setting \cite{jiang2009general} and its optimal control, which may be of independent interest; we refer to Section \ref{sec:compound} (Theorem \ref{thm:npmle_compound}) for exact definitions and results. Whether such regularization-free results also hold for more general $g$-modeling approaches in the context of Theorem \ref{thm:regret_MLE_EB} remains an interesting open question.

\subsection{Negative results on $f$-modeling}

In this subsection, we demonstrate that, in order to achieve the optimal regret rate in Theorem \ref{thm:regret_MLE_EB}, the \emph{proper} Bayes form in $g$-modeling cannot be violated in general. To construct such a counterexample, we start with a detailed study of Robbins estimator, whose original form is given in (\ref{def:robbins_intro}). 
In the context of individual regret in (\ref{def:individual_regret}), we will study the following generalization of Robbins estimator for $\theta_n$:

\begin{align}\label{def:robbins}
\hat{\theta}^{\rob}_n(Y^n) \equiv \hat{\theta}^{\rob}_n(Y^n;y_0) \equiv
\begin{cases}
(Y_n+1)\frac{N_{n-1}(Y_n+1)}{N_{n-1}(Y_n) + 1} & Y_n\leq y_0,\\
Y_n & Y_n > y_0,
\end{cases}
\end{align}
where $N_{n-1}(y) = \sum_{i=1}^{n-1} \indc{Y_i = y}$, and $y_0\in\integers_+$ is a tuning parameter to be chosen later. To further simplify the notation, we will also abbreviate the above estimator as $\hat{\theta}_n^{\rob}$. Clearly, the original Robbins estimator (\ref{def:robbins_intro}), when applied to $Y_n$, corresponds to $y_0 = \infty$; however, as we will show next, without truncation, the Robbins estimator can be inconsistent.
 

The following result provides matching upper and lower bounds for the individual regret of the Robbins estimator (\ref{def:robbins}). For the rest of this subsection, for simplicity, we take $M_p = 1$ in (\ref{def:mixture_class_intro}) so that the class $\mathcal{G}_p(1)$ consists of all priors with $p$th moment at most one; nevertheless, the results below hold for any constant $M_p$.


\begin{theorem}\label{thm:regret_robbins_EB}
Fix any $p > 1$. Then there exists some $C = C(p) > 0$ such that
\begin{align}\label{ineq:robbins_main_upper}
\inf_{y_0 \geq 1}\reg_n\Big(\hat{\theta}_n^{\rob}(Y^n;y_0); \mathcal{G}_p(1)\Big) \leq Cn^{-\frac{p-1}{p+2}}(\log n)^{\frac{3(p-1)}{p+2}}.
\end{align}

Conversely, let $y_\ast \equiv \big(n/(\log n)^2\big)^{1/(p+1)}$. Then for any $y_0\geq 1$, there exists some $c = c(p) > 0$ such that
\begin{align}\label{ineq:robbins_main_lower}
\reg_n\Big(\hat{\theta}_n^{\rob}(Y^n;y_0); \mathcal{G}_p(1)\Big) \geq c\cdot \Big(\frac{(y_0\wedge y_\ast)^3}{n} + y_0^{-(p-1)}\Big).
\end{align}
Consequently, the regrets of the untruncated and optimally truncated Robbins estimator satisfy
\begin{align*}
\reg_n\Big(\hat{\theta}_n^{\rob}(Y^n;\infty); \mathcal{G}_p(1)\Big) &\geq cn^{-\frac{p-2}{p+1}}(\log n)^{-6},\\
\inf_{y_0 \geq 1}\reg_n\Big(\hat{\theta}_n^{\rob}(Y^n;y_0); \mathcal{G}_p(1)\Big) &\geq cn^{-\frac{p-1}{p+2}}.
\end{align*}
\end{theorem}

A few remarks on 
\prettyref{thm:regret_robbins_EB} are in order:

\begin{itemize}
	\item Unlike Theorem \ref{thm:regret_MLE_EB}, we do not consider any additional regularization in the formulation (\ref{def:robbins}), since the normalized denominator therein $n^{-1}(N_{n-1}(Y_n) + 1)$ is automatically lower bounded by $n^{-1}$.
	\item 
	Compared with the optimal regret $\reg_n(\mathcal{G}_p(1)) = \tilde \Theta(n^{-2(p-1)/(2p+1)})$ 
determined in  Theorems \ref{thm:regret_lower} and \ref{thm:regret_MLE_EB}, 
the generalized Robbins estimator (\ref{def:robbins}), when tuned with the best possible threshold $y_0$, is consistent for any $p > 1$ but only achieves the suboptimal rate $\tilde O(n^{-(p-1)/(p+2)})$, which cannot be improved in view of the lower bound (\ref{ineq:robbins_main_lower}). Furthermore, the original Robbins estimator with $y_0 = \infty$ is inconsistent for $p \in (1,2)$ (the case $p = 2$ is still open due to the poly-logarithmic gap in Theorem \ref{thm:regret_robbins_EB}).

\item
Using the same leave-one-out argument in \prettyref{rmk:totreg-npmle}, we may define a permutation-invariant estimator
\begin{align*}
\hat{\theta}^{\rob, n}(Y^n;y_0) \equiv \big(\hat{\theta}^{\rob}_1(Y^n;y_0), \ldots, \hat{\theta}^{\rob}_n(Y^n;y_0)\big),
\end{align*} 
where $\hat{\theta}^{\rob}_i(Y^n;y_0)$ applies the truncated Robbins estimator with $Y_{\backslash i}$ as the training data and $Y_i$ as the test data.
(Note that for $y_0=\infty$, each $\hat{\theta}^{\rob}_i$ is the same as applying \eqref{def:robbins_intro} to $Y_i$.)
This translates the individual regret bound in Theorem \ref{thm:regret_robbins_EB} to total regret, in particular,
\begin{align}
\inf_{y_0 \geq 1}\totreg_n\Big(\hat{\theta}^{\rob,n}(Y^n;y_0); \mathcal{G}_p(1)\Big) = \tilde\Theta(n^{\frac{3}{p+2}}).
\label{eq:totreg-robbins}
\end{align}

\item 
Despite the long history and wide application of the Robbins estimator, quantitative regret bounds were only obtained recently \cite{brown2013poisson,polyanskiy2021sharp}. For priors with compact support or a subexponential tail, \cite[Theorem 1]{polyanskiy2021sharp} shows that the original Robbins estimator with $y_0 = \infty$ achieves the optimal regret $O\big(n^{-1}\cdot (\log n/\log\log n)^2\big)$ and $O\big(n^{-1}\cdot (\log n)^3\big)$, respectively, with the exact logarithmic factors. 
This stands in stark contrast to the conclusion of Theorem \ref{thm:regret_robbins_EB}: for the moment class, the Robbins estimator is suboptimal by a polynomial factor.


\item 
As mentioned in \prettyref{sec:overview}, the instability of the Robbins estimator has been well recognized in practice:
it takes on exceptionally small or large values when either of its numerator or denominator is (near) zero (cf.~\prettyref{fig:compare}).
Theorem \ref{thm:regret_robbins_EB} shows that this lack of robustness is not merely a numerical issue but in fact directly related to the 
suboptimality of Robbins' estimator when the underlying prior only has a finite number of moments. 
Indeed, such heavy-tailed distributions give rise to a larger number of small but non-zero counts $N(y)$, which causes the Robbins estimator $\theta^{\rob}(y)$ to vary wildly.

\end{itemize}

The proof of Theorem \ref{thm:regret_robbins_EB} is presented in Section \ref{subsec:proof_robbins}. 
We briefly discuss of the proof technique and the ``least favorable'' priors for Robbins' estimator, which are also used in the proof of Theorem \ref{thm:hybrid} below.
The key is to obtain both upper and lower bounds for the bias and variance of  (\ref{def:robbins}) as a function of the prior $G$; see Lemma \ref{lem:binomial_moments} for details. Then the desired upper bound follows from a uniform control of these quantities using the moment constraint.
The lower bound follows by choosing two special instances of $G$: a ``sparse'' prior of the form $G = (1-\epsilon)\delta_0 + \epsilon\delta_a$ and 
a smooth heavy-tailed prior with density $g(a) \propto a^{-(p+1)}(\log a)^{-2}$, which result in the lower bound 
$y_0^{-(p-1)}$ and 
$(y_0\wedge y_\ast)^3/n$ in (\ref{ineq:robbins_main_lower}), respectively. 

Building on the analysis of Theorem \ref{thm:regret_robbins_EB}, we are now ready to construct a $f$-modeling estimator that is Hellinger rate-optimal in density estimation (up to logarithmic factors) but strictly rate sub-optimal in terms of regret. Note that we cannot directly use $\hat{f}^{\emp}$ and its induced Robbins estimator (\ref{def:robbins_intro}) for the purpose above because, as shown in Proposition \ref{prop:empirical_subopt}, $\hat{f}^{\emp}$ is Hellinger rate sub-optimal as a density estimator.


\begin{theorem}\label{thm:hybrid}
For any $\delta > 0$, there exists some probability mass function $\tilde{f}$ (measurable with respect to $Y^n$) such that
\begin{align*}
\sup_{G \in \mathcal{G}_p(1)} \E_G H^2(\tilde{f}, f_G) \leq Cn^{-\frac{2p}{2p+1}}(\log n)^6,
\end{align*}
and the resulting $f$-modeling estimator $\tilde{\theta}_n = (Y_n+1)\frac{\tilde{f}(Y_n+1)}{\tilde{f}(Y_n)}$ satisfies
\begin{align*}
\reg_n(\tilde{\theta}_n; \mathcal{G}_p(1)) \geq c\frac{n^{-\frac{2p-3}{2p+1} - \delta}}{(\log n)^4},
\end{align*}
where $C,c>0$ only depend on $p$.
\end{theorem}

%

The proof of Theorem \ref{thm:hybrid} is given in Section \ref{subsec:proof_hybrid}. 
It implies that, to achieve optimal regret, the proper Bayes form in $g$-modeling cannot be violated in general, or in other words, the Poisson mixture structure must be exploited during the density estimation stage.


\section{Concluding remarks}\label{sec:conclusion}

In this paper, we studied Poisson EB estimation with priors having a finite $p$th moment, and conducted a detailed comparison of the theoretical properties of $f-$ and $g$-modeling methods. The positive result on $g$-modeling reveals an interesting connection between density and EB estimation: Any $g$-modeling approach that achieves the optimal Hellinger rate of density estimation (up to logarithmic factors) also achieves the optimal regret rate (up to logarithmic factors). In contrast, we demonstrated an $f$-modeling method that achieves the optimal density estimation rate but is strictly regret rate sub-optimal by a polynomial factor. We also showed that the renowned Robbins estimator is sub-optimal by a polynomial factor for both density estimation and regret, which stands in sharp contrast to its optimality under light-tailed priors. 

Since $g$-modeling can been as a special class of $f$-modeling, an interesting topic for future study is to understand which properties of $g$-modeling are truly necessary to achieve regret optimality. One such property that stands out from general $f$-modeling approaches is monotonicity \cite{van1983weak, koenker2014convex, barbehenn2022nonparametric, jana2023poisson}, where it was shown in \cite{jana2023poisson} that an empirical risk minimizer with monotonicity constraint achieves optimal regret rate (down to log factors) when the prior is either bounded or has sub-exponential tails. It remains an interesting question to establish other theoretical guarantees for such monotonicity-constrained estimators, especially with heavy-tailed priors.

\section{Proofs for Section \ref{sec:density_est}}\label{sec:proof_density}

\subsection{Proof of Theorem \ref{thm:density_main}: Upper bound}\label{subsec:proof_density_upper}

\subsubsection{A local moment matching lemma}

The following local moment matching lemma is our main technical contribution in the density estimation upper bound. 
Recall that for any $f:\integers_+\to\reals$, $\pnorm{f}{\infty,M} = \max_{x=0,\ldots,M} |f(x)|$.
\begin{lemma}\label{lem:moment_matching}
Fix any mixing distribution $G$ supported on $\reals_+$, and let $f_G$ be the Poisson mixture density defined in (\ref{def:poi_mixture_intro}). Fix $M > 0$ and $\eta\in (0,10^{-3})$ such that $M \geq (\log(1/\eta))^{\rho_M}$ for some sufficiently large $\rho_M > 0$.
 Then there exists a discrete distribution $G_m$ supported on $[0,2M]$ with at most $m \leq K\sqrt{M}(\log(1/\eta))^{3/2}$ atoms for some universal $K > 0$, such that
\begin{align*}
\pnorm{f_G - f_{G_m}}{\infty, M} \leq \eta,
\end{align*}
where $f_{G_m}$ is the Poisson mixture induced by $G_m$.
\end{lemma}
\begin{proof}[Proof of Lemma \ref{lem:moment_matching}]
Let $f_j(\lambda) \equiv \Poi(j;\lambda) = \lambda^je^{-\lambda}/j!$. For any $j\in[0,M]$, we have
\begin{align}\label{ineq:h_decomp}
\notag|f_G(j) - f_{G_m}(j)| &= \Big|\int f_j(\lambda) \Big(G(\d\lambda) - G_m(\d\lambda)\Big)\Big|\\
&\leq \Big|\int_0^{2M} f_j(\lambda) \Big(G(\d\lambda) - G_m(\d\lambda)\Big)\Big| + \Big|\int_{\lambda > 2M} f_j(\lambda) \Big(G(\d\lambda) - G_m(\d\lambda)\Big)\Big|.
\end{align}
For any $\lambda > 2M$, we have $j \leq \lambda/2$, hence by the Poisson tail bound (see Lemma \ref{lem:poi_basic}\ref{poi_basic1} in \prettyref{app:aux}),
 we have with $X\sim \poi(\lambda)$,
\begin{align*}
f_j(\lambda) \leq \Prob(X - \lambda \leq -\lambda/2) \leq \exp(-\lambda/12) \leq \exp(-M/6) \leq \eta/10,
\end{align*}
using the conditions on $(M,\eta)$. Hence the second term in (\ref{ineq:h_decomp}) is bounded by $\eta/10$. For the first term, let $\bar{\eta} \equiv \log(1/\eta)$, and we consider the following partition of $[0,2M]$: for $0\leq i\leq N$ with $N\equiv \ceil{\sqrt{2M/(C\bar{\eta})} - 1}$,
\begin{align}\label{def:partition}
I_i \equiv [i^2C\bar{\eta}, \big((i+1)^2C\bar{\eta}\big) \wedge 2M).
\end{align}
Let $L_i$ denote the degree of polynomial approximation we will apply on the interval $I_i$.
Let $G_i$ denote $G$ conditioned on $I_i$, namely, $G_i(A) = G(A)/w_i$ for any $A\subset I_i$, where $w_i\equiv G(I_i)$.
By the Carath\'eodory theorem, for each $0\leq i\leq N$, there exists a discrete distribution $G^{(i)}$ supported on $I_i$ with $L_i$ atoms,\footnote{In fact, $\lceil(L_i+1)/2\rceil$ atoms will do.} such that
\begin{align}\label{eq:local_matching}
\int_{I_i} u^k G_i(\d u) = \int_{I_i} u^k G^{(i)}(\d u), \quad \forall k = 1,\ldots, L_i.
\end{align}
Combine $\{G^{(i)}\}_{0\leq i\leq N}$ to obtain
\begin{align*}
G_m \equiv \sum_{i=0}^{N} w_iG^{(i)} + \Big(1-\sum_{i=0}^N w_i\Big)\delta_{2M}, 
\end{align*} 
which is supported on $[0,2M]$ with $m = \sum_{i=0}^N L_i + 1$ atoms. Now the first term in (\ref{ineq:h_decomp}) can be written as
\begin{align*}
S(j) \equiv \int_0^{2M} f_j(\lambda) \Big(G(\d\lambda) - G_m(\d\lambda)\Big) &= \sum_{i=0}^{N}\int_{I_i}f_j(\lambda) \Big(G(\d\lambda) - G_m(\d\lambda)\Big)\\
&= \sum_{i=0}^{N}w_i\cdot \int_{I_i}f_j(\lambda) \Big(G_i(\d\lambda) - G^{(i)}(\d\lambda)\Big).
\end{align*}
We will now bound $|S(j)|$ uniformly over $j\in[0,M]$. Fix any such $j$ so that $j \in I_{i_0}$ for some $i_0 = i_0(j)$. Then for any $i > i_0+1$ and $\lambda \in I_i$ (if such $i$ exists), we have
\begin{align*}
|\lambda - j| \geq |i^2C\bar{\eta} - (i_0+1)^2C\bar{\eta}| \geq (i+1)C\bar{\eta}\geq \sqrt{C\bar{\eta}\cdot \lambda}.
\end{align*}
Hence by the Poisson tail bound in Lemma \ref{lem:poi_basic}\ref{poi_basic1}, for such $\lambda$, we have $f_j(\lambda) \leq \eta$ by choosing $C > 0$ in (\ref{def:partition}) to be a large enough universal constant. A similar argument applies to $i < i_0 - 1$ and $\lambda \in I_i$. Hence
\begin{align*}
\sup_{j\in I_{i_0}}|S(j)| \leq \eta + \max_{i\in\{i_0-1,i_0,i_0+1\}}\bigg|\int_{I_i}f_j(\lambda) \Big(G_i(\d\lambda) - G^{(i)}(\d\lambda)\Big)\bigg|.
\end{align*}

\textbf{Case 1: $i_0(j) \leq M^{1/6}$.} We only bound the term
\begin{align*}
S(j,i_0) \equiv \int_{I_{i_0}}f_j(\lambda) \Big(G_{i_0}(\d\lambda) - G^{(i_0)}(\d\lambda)\Big);
\end{align*}
the other two terms for $I_{i_0-1}$ and $I_{i_0 +1}$ are similar. By Taylor expansion and the moment matching of (\ref{eq:local_matching}) on $I_{i_0}$, we have
\begin{align*}
\big|S(j,i_0)\big| = \Big|\int_{I_{i_0}} R_{L_{i_0};j}(\lambda)\Big(G_{i_0}(\d\lambda) - G^{(i_0)}(\d\lambda)\Big)\Big| \leq \sup_{\lambda\in I_{i_0}} |R_{L_{i_0};j}(\lambda)|,
\end{align*}
where, with $\underline{d}_i \equiv i^2C\bar{\eta}$ (resp.~$\bar{d}_i \equiv (i+1)^2C\bar{\eta}$) denoting the left (resp.~right) end of $I_i$,
\begin{align*}
\big|R_{L_{i_0};j}(\lambda)\big| \equiv \Big|f_j(\lambda) - \sum_{\ell=0}^{L_{i_0}}\frac{f_j^{(\ell)}(\underline{d}_{i_0})(\lambda - \underline{d}_{i_0})^\ell}{\ell!}\Big| \leq \frac{\sup_{\lambda\in[0,\bar{d}_{i_0}]}\Big|f_j^{(L_{i_0}+1)}(\lambda)\Big|}{(L_{i_0} + 1)!}\big|\bar{d}_{i_0} - \underline{d}_{i_0}\big|^{L_{i_0}+1},
\end{align*}
for all $\lambda\in I_{i_0}$. We know that for $j \leq L_{i_0} + 1$, $\sup_{\lambda\in [0, \bar{d}_{i_0}]} \big|f_j^{(L_{i_0}+1)}(\lambda)\big| \leq \sup_{\lambda\in [0, \bar{d}_{i_0}]}e^{-\lambda/2}{L_{i_0+1} \choose j} = {L_{i_0+1} \choose j} \leq (L_{i_0 + 1})^j$ \cite[Equation (3.23)]{wu2020polynomial}. Hence by choosing $L_i$ in (\ref{eq:local_matching}) such that $L_{i_0} \geq C_1 (i_0+1)^2 \bar{\eta}^2$ for some large enough universal $C_1$, we have
\begin{align*}
\big|R_{L_{i_0};j}(\lambda)\big| \leq \frac{(L_{i_0} + 1)^j \big((2i_0+1)C\bar{\eta}\big)^{L_{i_0 + 1}}}{(L_{i_0} + 1)^{L_{i_0} + 1}e^{-(L_{i_0} + 1)}} \leq \frac{\big((2i_0+1)Ce\bar{\eta}\big)^{L_{i_0 + 1}}}{(L_{i_0} + 1)^{(L_{i_0} + 1)/2}} \leq \eta/10.
\end{align*}

\textbf{Case 2: $i_0(j) \geq M^{1/6}$.} 
As in the previous case, we only bound the term $S(j,i_0)$.
Denote by
\begin{align*}
e_L(f, R) \equiv \inf_{\text{deg}(P) \leq L} \sup_{x\in R} |f(x) - P(X)|
\end{align*}
the error of the best degree-$L$ polynomial approximation of a function $f$ on the set $R$.
For any polynomial $P$ of degree at most $L_{i_0}$, the moment matching of (\ref{eq:local_matching}) on $I_{i_0}$ yields that
\begin{align*}
\big|S(j,i_0)\big| &= \Big|\int_{I_{i_0}}\Big(f_j(\lambda) - P(\lambda)\Big)\Big(G_{i_0}(\d\lambda) - G^{(i_0)}(\d\lambda)\Big)\Big|\\
&\leq \sup_{\lambda\in I_{i_0}}\Big|f_j(\lambda) - P(\lambda)\Big| \leq \sup_{\lambda\in[j - K_\eta\sqrt{j}, j + K_\eta\sqrt{j}]}\Big|f_j(\lambda) - P(\lambda)\Big|,
\end{align*}
where $K_\eta \equiv 3\sqrt{C\bar{\eta}}$ and the last inequality follows by $|\lambda - j| \leq (2i_0 + 1)C\bar{\eta} \leq K_\eta \sqrt{j}$ for all $\lambda \in I_{i_0}$.  Optimizing over $P$ yields
\begin{align*}
\big|S(j,i_0)\big| &\leq e_{L_{i_0}}\Big(f_j, [j - K_\eta\sqrt{j}, j + K_\eta\sqrt{j}]\Big) = e_{L_{i_0}}\Big(g_j, [- K_\eta\sqrt{j}, K_\eta\sqrt{j}]\Big),
\end{align*}
where $g_j(\lambda) \equiv f_j(\lambda + j) = (\lambda + j)^j e^{-(\lambda + j)}/j!$. To approximate $g_j$, note that $g_{j}(\lambda) = C(j)\exp(j\log(1 + \lambda/j) - \lambda)$ with $C(j) = (j/e)^{j}/j! \leq 1$ by Stirling approximation. Now we approximate the exponent inside $g_{j}$. Let $H_k(x)$ be the $k$th order Taylor expansion of $x\mapsto \log(1+x)$ around $x = 0$, so that $|\log(1 + x) - H_k(x)| \leq |x|^{k+1}(1/(x+1) \vee 1)$ for all $x\in[-1/2,1/2]$.
 For the given positive integer $k$, let 
\begin{align}\label{def:g_poly}
\bar{g}_{j;k} \equiv C(j)\exp\big(j\cdot H_{k}(\lambda/j) - \lambda\big).
\end{align}
Then for $\lambda \in [-K\sqrt{j}, K\sqrt{j}]$ with $\lambda/j \in [-1/2,1/2]$, we have
\begin{align*}
&|g_{j}(\lambda) - \bar{g}_{j;k}(\lambda)|\\
&\leq C(j)\exp\big(j\log(1 + \lambda/j) - \lambda\big)\cdot \Big|\exp\Big(j \cdot \big(-\log(1 + \lambda/j) + H_{k}(\lambda/j)\big)\Big) - 1\Big|\\
&\lesssim  C(j)\cdot j|\lambda/j|^{k+1} \lesssim K^{k + 1}/j^{k/2}.
\end{align*} 
Let $\bar{H}_k(\lambda) \equiv \bar{H}_k(\lambda;j) \equiv jH_k(\lambda/j) - \lambda$, so that 
\begin{align}\label{ineq:bound_Hk}
\big|\bar{H}_{k}(\lambda) + \lambda^2/(2j)\big| \leq 3|\lambda|^3/j^2, \quad \lambda\in[-K\sqrt{j}, K\sqrt{j}].
\end{align}
Since $j \geq M^{1/6} \geq (\log(1/\eta))^{\rho_M / 6}$ for some sufficiently large $\rho_M \geq 7$, we can choose $k = C_k\log(1/\eta)$ for a large enough universal $C_k > 0$ such that $\sup_{|\lambda| \leq K\sqrt{j}} |g_j(\lambda) - \bar{g}_{j;k}(\lambda)|  \leq \eta^{10}$. This implies
\begin{align*}
\big|S(j,i_0)\big| &\leq \eta^{10} + e_{L_{i_0}}\Big(\bar{g}_{j;k}, [- K_\eta\sqrt{j}, K_\eta\sqrt{j}]\Big)\\
&=  \eta^{10} + C(j) \cdot e_{L_{i_0}}\Big(\exp\big(\bar{H}_{k}(\lambda)\big), [- K_\eta\sqrt{j}, K_\eta\sqrt{j}]\Big)\\
&\stackrel{\rm (a)}{\leq}  \eta^{10} + e_{L_{i_0}/k}\Big(\exp(-\lambda), (-\bar{H}_{k})\big([- K_\eta\sqrt{j}, K_\eta\sqrt{j}]\big)\Big)\\
&\stackrel{\rm (b)}{\leq}  \eta^{10} + e_{L_{i_0}/k}\Big(\exp(-\lambda), [0, K_\eta^2]\Big),
\end{align*} 
where (a) follows as $\bar{H}_{k}(\cdot)$ maps a degree-$m$ polynomial into a polynomial of degree at most $mk$, and (b) follows from the bound in (\ref{ineq:bound_Hk}). Hence using $e_L(\exp(-\lambda), [0,r]) \leq r^{L+1}/(L+1)!$ via Taylor approximation, we have
\begin{align*}
\big|S(j,i_0)\big| \leq  \eta^{10} + \frac{K_\eta^{2(L_{i_0}/k + 1)}}{(L_{i_0}/k + 1)!} \leq \eta/10, 
\end{align*}
as long as we choose $L_i$ in (\ref{eq:local_matching}) such that $L_{i_0}/k \geq C'K_\eta^2 $ for some large universal $C' > 0$.

Combining the above two cases yields that under the partition (\ref{def:partition}), for any $j\in[0,M]$, the first term $S(j)$ of (\ref{ineq:h_decomp}) can be bounded by $C_2'\eta$ with some $G_m$ having
\begin{align*}
m \lesssim \sum_{i=0}^{M^{1/6}} (i+1)^2\bar{\eta}^2 + \sum_{i = M^{1/6} + 1}^{N} K_\eta^4 \asymp \sqrt{M}(\log (1/\eta))^{3/2}.
\end{align*} 
atoms. The proof is complete.
\end{proof}

\subsubsection{Completing the proof}

The rest of the upper bound proof largely follows that of \cite{zhang2009generalized}; see also \cite{ghosal2001entropies, ghosal2007posterior}. 
We provide the complete argument for the convenience of the reader.

For the following lemma, recall that for any $\epsilon > 0$, mixture class $\mathcal{H}$, and semi-norm $\pnorm{\cdot}{}$, $\mathcal{N}(\epsilon, \mathcal{H}, \pnorm{\cdot}{})$ denotes the $(\epsilon, \pnorm{\cdot}{})$-covering number of $\mathcal{H}$ (see, e.g., \cite[Definition 2.1.5]{van1996weak}). Let
\begin{align}\label{def:mixture_class_new}
\mathcal{H}_0 \equiv \big\{f_G: G\in \calP(\reals_+)\big\},\quad \mathcal{H}_p(M_p) \equiv \big\{f_G: G \in \mathcal{G}_p(M_p)\},
\end{align}
where $\calP(\reals_+)$ is the collection of all priors and $\mathcal{G}_p$ is the moment classes in (\ref{def:mixture_class_intro}).
\begin{lemma}\label{lem:entropy}
Fix any $M > 0$ and $\eta \in (0,10^{-3})$ such that $M \geq (\log(1/\eta))^{\rho_M}$ for some sufficiently large $\rho_M > 0$.
 Then there exists some universal $K > 0$ such that
\begin{align*}
\log \mathcal{N}(\eta, \mathcal{H}_0, \pnorm{\cdot}{\infty, M}) \leq K\sqrt{M}(\log(1/\eta))^{3/2}\log(M/\eta).
\end{align*}
\end{lemma}
\begin{proof}[Proof of Lemma \ref{lem:entropy}]
By Lemma \ref{lem:moment_matching}, for any distribution $G$ supported on $\reals_+$, there exists a discrete distribution $G_m$ supported on $[0,2M]$ with $m \leq K\sqrt{M}(\log(1/\eta))^{3/2}$ atoms such that
\begin{align*}
\pnorm{f_G - f_{G_m}}{\infty, M} \leq \eta,
\end{align*}
where $K$ is universal. We first approximate the support of $G_m$ by an $\eta$-grid of $[0,2M]$. Let $G_m = \sum_{i=1}^m w_i\delta_{\mu_i}$ with weights $\{w_i\}_{i=1}^m$ and atoms $\{\mu_i\}_{i=1}^m$. For each $\mu_i$, let $\mu'_i$ be the closet point on the grid so that $|\mu_i - \mu'_i| \leq \eta$. Let $G_{m,\eta}\equiv \sum_{i=1}^m w_i\delta_{\mu_i'}$, then with $f_j(\mu) \equiv \poi(j;\mu)$ denoting the Poisson pmf with mean $\mu$, 
\begin{align*}
\pnorm{f_{G_m} - f_{G_{m,\eta}}}{\infty} \equiv \sup_{j\geq0} \Big|\sum_{i=1}^m w_i \cdot \big(f_j(\mu_i) - f_j(\mu_i')\big)\Big| \leq \eta,
\end{align*}
using $\sup_{j\geq0}\sup_{\lambda > 0}|f_j'(\lambda)| \leq 1$ proved in Lemma \ref{lem:pmf_lip} in \prettyref{app:aux}. Next, let $\mathcal{P}^m \equiv \{w = (w_1,\ldots,w_m): w_i\geq 0, \sum_{i=1}^m w_i = 1\}$ be the probability simplex in $\R^m$, and $\mathcal{P}^{m,\eta}$ be an $\eta$-net in $\ell_1$ distance:
\begin{align*}
\sup_{w\in\mathcal{P}^m}\inf_{w'\in \mathcal{P}^{m,\eta}} \pnorm{w - w'}{1} \leq \eta.
\end{align*}
Then a standard volume comparison shows that $|\mathcal{P}^{m,\eta}| \leq (3/\eta)^{m}$. 
Let $\bar{G}_{m,\eta}$ be the approximation of $G_{m,\eta}$ with $\{w_i\}_{i=1}^m$ therein replaced by its closest point (in $\ell_1$) $\{w_i'\}_{i=1}^m$ in $\mathcal{P}^{m,\eta}$. Then using $f_j(\lambda) \leq 1$, we have
\begin{align*}
\pnorm{f_{G_{m,\eta}} - f_{\bar{G}_{m,\eta}}}{\infty} \leq \sum_{i=1}^m |w_i - w_i'| \leq \eta.
\end{align*}
This implies $\pnorm{f_G - f_{\bar{G}_{m,\eta}}}{\infty, M} \leq 3\eta$. Finally, counting the number of possible realizations of $f_{\bar{G}_{m,\eta}}$ yields 
\begin{align*}
\log \mathcal{N}(3\eta, \mathcal{H}_0, \pnorm{\cdot}{\infty, M}) \leq \log {2M/\eta +1 \choose m}\cdot \Big(\frac{3}{\eta}\Big)^m \leq m\log(9eM/m\eta^2) \lesssim m\log (M/\eta),
\end{align*}
where the last inequality follows from the condition on $M$. The claim now follows by adjusting the constants. 
\end{proof}

For the following lemma, recall the mixture class $\mathcal{H}_p(M_p)$ defined in (\ref{def:mixture_class_new}).
\begin{lemma}\label{lem:truncation_moment}
Suppose that $Y^n = (Y_1,\ldots,Y_n)$ are i.i.d.~observations from some $f_G\in\mathcal{H}_p(M_p)$ for some $p > 0$ and $M_p > 0$. Then for any $0 < \lambda < \min(1,p/2)$, $a > 0$, and $M \geq 1$, 
\begin{align*}
\E \Big\{\prod_{i=1}^n (aY_i)^{\indc{Y_i\geq M}}\Big\}^\lambda \leq \exp\Big[C_p\cdot n(aM)^\lambda \cdot \Big(\exp(-cM) + M^{-p}M_p\Big)\Big].
\end{align*}
Here $c > 0$ is universal and $C_p > 0$ only depends on $p$.
\end{lemma}
\begin{proof}[Proof of Lemma \ref{lem:truncation_moment}]
By independence of $\{Y_i\}$, we have
\begin{align*}
\E \Big(\prod_{i=1}^n (aY_i)^{\indc{Y_i\geq M}}\Big)^\lambda &= \prod_{i=1}^n \E  \Big((aY_i)^{\indc{Y_i\geq M}}\Big)^\lambda \leq \prod_{i=1}^n \E \Big(1 + (aY_i)^\lambda\indc{Y_i\geq M}\Big)\\
&\leq \prod_{i=1}^n \exp\Big(a^\lambda \cdot \E Y_i^\lambda\indc{Y_i\geq M}\Big).
\end{align*}
For each $i\in[n]$, applying $y^\lambda=M^\lambda + \lambda\cdot \int_M^\infty u^{\lambda - 1} \d u$, we have
\begin{align*}
\E Y_i^\lambda\indc{Y_i\geq M} &= \sum_{y=M}^\infty y^\lambda f_G(y) = M^\lambda \Prob_G(Y \geq M) + \lambda\cdot \int_M^\infty u^{\lambda - 1}\Prob_G(Y \geq u)\d u\\
&\stackrel{\rm (a)}{\leq} M^\lambda\Big(\exp(-cM) + (M/2)^{-p}M_p\Big) + \lambda \cdot \int_M^\infty u^{\lambda -1}\Big(\exp(-cu) + (u/2)^{-p}M_p\Big)\d u\\
&\stackrel{\rm (b)}{\lesssim} M^\lambda\exp(-c M) + M^{\lambda - p}M_p,
\end{align*}
where (a) uses the mixture tail bound in Lemma \ref{lem:poi_basic}\ref{poi_basic3}, and (b) uses the fact that $M$ is large enough and $\lambda \leq \min\{1,p/2\}$. Using this estimate, we have
\begin{align*}
\E \Big(\prod_{i=1}^n (aY_i)^{\indc{Y_i\geq M}}\Big)^\lambda \leq \exp\Big[C_p\cdot n(aM)^\lambda \cdot \Big(\exp(-cM) + M^{-p}M_p\Big)\Big],
\end{align*}
as desired.
\end{proof}

We are now ready to complete the proof of Theorem \ref{thm:density_main}.
\begin{proof}[Proof of Theorem \ref{thm:density_main}]
Suppose the true density is $f_{G_0} \in \mathcal{H}_p(M_p)$. For any $r > 0$, let $B(r) \equiv B(r;H,f_{G_0}) \equiv \{f_G\in \mathcal{H}_0: H(f_G, f_{G_0}) \leq r\}$ be the Hellinger ball of radius $r$ centered at the truth $f_{G_0}$ and let $B(r)^c \equiv  \mathcal{H}_0 \backslash B(r)$. For any positive functions $g_1, g_2$ with domain $\mathbb{Z}_+$, let 
\begin{align*}
L(g_1,g_2) \equiv \prod_{i=1}^n \frac{g_1(Y_i)}{g_2(Y_i)}.
\end{align*}
Then for any $t \geq 0$, by definition of $f_{\hat{G}}$, we have
\begin{align*}
\Prob\Big(H(f_{\hat{G}}, f_{G_0}) \geq t\epsilon_n\Big) \leq \Prob\Big(\exists f_G \in B(t\epsilon_n)^c \text{ s.t. } L(f_G, f_{G_0}) \geq 1\Big).
\end{align*}
Fix some $\eta > 0$ and $M > 0$ to be chosen later. Let $\mathcal{N} = \{f_{G_1},\ldots, f_{G_{N}}\}$ be a proper $(\eta,\pnorm{\cdot}{\infty,M})$-net of $B(t\epsilon_n)^c$ (here ``proper'' means $\mathcal{N} \subset B(t\epsilon_n)^c$), with $N = \mathcal{N}(\eta, \mathcal{H}_0, \pnorm{\cdot}{\infty, M})$ and the latter bounded by Lemma \ref{lem:entropy}. 
Let 
\begin{align*}
f_\ast(y) \equiv \eta\indc{y\leq M} + \frac{\eta M^2}{y^2}\indc{y > M}, \quad y\in\mathbb{Z}_+.
\end{align*}
Consequently, for any $f_G\in\mathcal{H}_0$ such that $H(f_G, f_{G_0}) > t\epsilon_n$, there exists some $j\leq N$ such that
\begin{align*}
f_G(y) \leq
\begin{cases}
f_{G_j}(y) + \eta = f_{G_j}(y) + f_\ast(y) & y \leq M,\\
1 & y > M,
\end{cases}
\end{align*}
which implies that
\begin{align*}
L(f_G, f_{G_0}) &= \prod_{i:Y_i\leq M}\frac{f_G(Y_i)}{f_{G_0}(Y_i)} \cdot \prod_{i:Y_i > M}\frac{f_G(Y_i)}{f_{G_0}(Y_i)} \leq \prod_{i:Y_i\leq M}\frac{(f_{G_j} + f_\ast)(Y_i)}{f_{G_0}(Y_i)} \cdot \prod_{i:Y_i > M}\frac{1}{f_{G_0}(Y_i)}\\
&\leq L(f_{G_j} + f_\ast, f_{G_0}) \cdot \prod_{i: Y_i > M} \frac{1}{f_\ast(Y_i)}.
\end{align*}
Taking the supremum over $f_G\notin B(t\epsilon_n)$ yields that
\begin{align*}
&\Prob\Big(H(f_{\hat{G}}, f_{G_0}) \geq t\epsilon_n\Big)\\
&\leq \Prob\Big(\max_{j\leq N}L(f_{G_j} + f_\ast, f_{G_0})\cdot  \prod_{i: Y_i > M} \frac{1}{f_\ast(Y_i)} \geq 1\Big)\\
&\leq \Prob\Big(\max_{j\leq N} L(f_{G_j} + f_\ast, f_{G_0}) \geq \exp(-nt^2\epsilon_n^2/2)\Big) + \Prob\Big( \prod_{i: Y_i > M} \frac{1}{f_\ast(Y_i)} \geq \exp(nt^2\epsilon_n^2/2)\Big)\\
&\equiv (I) + (II).
\end{align*}

To bound $(I)$, we have
\begin{align*}
(I) &\leq N\cdot \max_{j\leq N} \Prob\Big(L(f_{G_j} + f_\ast, f_{G_0}) \geq \exp(-nt^2\epsilon_n^2/2)\Big)\\
&= N\cdot \max_{j\leq N} \Prob\bigg(\prod_{i=1}^n \sqrt{\frac{(f_{G_j} + f_\ast)(Y_i)}{f_{G_0}(Y_i)}} \geq \exp(-nt^2\epsilon_n^2/4)\bigg)\\
&\leq N\cdot \max_{j\leq N} \exp(nt^2\epsilon_n^2/4) \Big(\E_{G_0} \sqrt{\frac{(f_{G_j} + f_\ast)(Y_1)}{f_{G_0}(Y_1)}}\Big)^n\\
&\leq N\cdot \max_{j \leq N} \exp\bigg(nt^2\epsilon_n^2/4 + n\cdot \Big(\E_{G_0} \sqrt{(f_{G_j}+f_\ast)(Y_1)/f_{G_0}(Y_1)} - 1\Big)\bigg),
\end{align*}
using $\log x\leq x -1$ for all $x > 0$ in the last inequality. 
For each $f_{G_j}$, using $\sum_{y=0}^\infty f_\ast(y) \leq 2\eta M$, we have
\begin{align*}
&\E_{G_0} \Big(\sqrt{(f_{G_j}+f_\ast)(Y_1)/f_{G_0}(Y_1)} - 1\Big) \leq \E_{G_0}\sqrt{f_{G_j}(Y_1)/f_{G_0}(Y_1)} - 1 + \sum_{y=0}^\infty \sqrt{f_*(y)f_{G_0}(y)}\\
&\leq -\frac{1}{2}H^2(f_{G_j}, f_{G_0}) + \sqrt{\sum_{y=0}^\infty f_\ast(y)} \leq -\frac{1}{2}(t\epsilon_n)^2 + \sqrt{2\eta M}.
\end{align*}
Now we choose
\begin{align}\label{eq:upper_choice}
\eta = n^{-2}, \quad M = (\log n)^{-5}(n\epsilon_n^2)^2.
\end{align}
By definition of $\epsilon_n$, we have $M \geq (\log(1/\eta))^{\rho_M}$ for some $\rho_M \geq 7$, which allows us to
apply Lemma \ref{lem:entropy} to obtain $\log N = \log \mathcal{N}(\eta, \mathcal{H}_0, \pnorm{\cdot}{\infty,M}) \leq K_p\sqrt{M} (\log n)^{5/2}$. This implies
\begin{align*}
(I) \leq \exp\Big(K_{p}\sqrt{M}(\log n)^{5/2} + nt^2\epsilon_n^2/4 -n(t\epsilon_n)^2/2 + n\sqrt{2\eta M}\Big) \leq \exp\big(-nt^2\epsilon_n^2/8\big)
\end{align*}
for $t\geq t_* = t_\ast(p)$ with some sufficiently large $t_*$.

%

To bound $(II)$, we have by the definition of $f_*$ and $\pnorm{f_G}{\infty} \leq 1$ that
\begin{align*}
(II) &= \Prob\Big(\prod_{i:Y_i > M}\frac{Y_i^2}{\eta M^2}  \geq \exp(nt^2\epsilon_n^2/2)\Big) \leq \exp(-n\lambda t^2\epsilon_n^2/4) \cdot \E \prod_{i: Y_i > M} \Big(\frac{Y_i}{M\sqrt{\eta}}\Big)^\lambda\\
&\leq \exp\bigg(-n\lambda t^2\epsilon_n^2/4 + C_p\cdot n\eta^{-\lambda/2} \cdot \Big(\exp(-cM) + M^{-p}M_p\Big)\bigg),
\end{align*}
using Lemma \ref{lem:truncation_moment} in the last inequality with $a = (M\sqrt{\eta})^{-1}$ and some $0 < \lambda < \min(1,p/2)$. Hence by choosing $\lambda = 1/\log n$, the choice of $(\eta, M)$ in (\ref{eq:upper_choice}) guarantee that
for $t\geq t_*$ with sufficiently large $t_* = t_\ast(p)$ that $(II) \leq  \exp(-n\lambda t^2\epsilon_n^2/8) = \exp\big(-nt^2\epsilon_n^2/(8\log n)\big)$. Combining the estimates of $(I)$ and $(II)$ concludes the proof.
\end{proof}

\subsection{Proof of Theorem \ref{thm:density_lower}: Lower bound}\label{subsec:proof_density_lower}
For each $i\geq 1$, let $I_i \equiv [i^2 (\log n)^2, (i+1)^2(\log n)^2]$. Let $a_0 = 0$, and for $i\geq 1$, $a_i \equiv (\log n)^2\cdot (i^2 + (i+1)^2)/2$ be the center of $I_i$. Fix two positive integers $i_0\leq N/2$ to be chosen later. 
Let $w_i \equiv M_p\big((i+1)^2(\log n)^2\big)^{-(p+1/2)}$ for $i_0\leq i\leq N$, then for large enough $n$,
\begin{align*}
\bar{w} &\equiv \sum_{i=i_0}^N w_i = M_p(\log n)^{-(2p+1)}\sum_{i=i_0}^N (i+1)^{-(2p+1)}\\
&\leq M_p\frac{(i_0+1)^{-2p} - (N+1)^{-2p}}{2p\cdot (\log n)^{2p+1}} \leq 1,
\end{align*}
as long as 
\begin{align}\label{ineq:density_lower_1}
M_p i_0^{-2p} \leq 2p.
\end{align}
Let $w_0 \equiv 1 - \bar{w}$. Let $b_i \equiv a_i + \delta_i$ with $\delta_i^2 = a_i/(nw_i(\log n)^{10})$ for $i_0\leq i \leq N$. Direct calculation shows that as long as 
\begin{align}\label{ineq:density_lower_2}
\big((N+1)\log n\big)^{2p+1} \leq C_pnM_p
\end{align}
 for some small enough $C_p > 0$, we have: (i) $\delta_i \leq |I_i|/100$, yielding that $a_i, b_i \in [(i\log n)^2 + |I_i|/4, (i\log n)^2 + 3|I_i|/4]$, and $\delta_i^2/(i\log n)^2 \leq (1600)^{-1}(\log n)^2$; (ii) $w_{N} = \min_{i_0\leq i\leq N}w_i \geq 2/n$.

For any $\bm{\tau} = (\tau_{i_0},\ldots,\tau_{N})$ with $\tau_i \in\{0,1\}$, define a probability distribution (with convention $\lambda_0 \equiv 0$)
\begin{align}\label{def:lower_bound_construction}
G_{\bm{\tau}} \equiv w_0\delta_0 + \sum_{i=i_0}^N w_i\delta_{\lambda_i}, \text{ where } 
\lambda_i \equiv
\begin{cases}
a_i  & \tau_i = 0\\
 b_i & \tau_i = 1\\
\end{cases}, \quad i_0 \leq i \leq N.
\end{align}
Since $a_i \leq b_i \leq ((i+1)\log n)^2$,
\begin{align*}
m_p(G_{\bm{\tau}}) & \leq \sum_{i=i_0}^N M_p\big((i+1)^2(\log n)^2\big)^{-(p+1/2)}\cdot \big((i+1)^2(\log n)^2\big)^{p}\\
&= M_p\cdot\sum_{i=i_0}^N \frac{1}{(i+1)\log n} \leq M_p^p\cdot\frac{\log\big((N+1)/(i_0+1)\big)}{\log n}.
\end{align*}
Hence under the condition (\ref{ineq:density_lower_2}) and additionally
\begin{align}\label{ineq:density_lower_3}
\log\Big(\frac{N+1}{i_0+1}\Big) \leq \log n,
\end{align}
we have $m_p(G_{\bm{\tau}}) \leq M_p$.

By Assouad's lemma (see, e.g., \cite[Theorem 2.12(iv)]{Tsybakov09}), it suffices to upper bound $\chi^2\big(f_{\bm{\tau}}||f_{\bm{\tau}'}\big)$ for $d(\bm{\tau},\bm{\tau}') = 1$, and lower bound $H^2(f_{\bm{\tau}},f_{\bm{\tau}'})/d(\bm{\tau},\bm{\tau}')$ for all $\bm{\tau} \neq \bm{\tau}'$, where $d(\cdot,\cdot)$ is the Hamming distance. For the first quantity, suppose that $\bm{\tau}$ and $\bm{\tau}'$ only differ at the $i_\ast$-th position. Then
\begin{align}\label{ineq:density_lower_chi}
\notag\chi^2(f_{\bm{\tau}}||f_{\bm{\tau}')} &= \sum_{k = 0}^\infty \frac{\Big(\sum_{i=i_0}^N w_i \big(\poi(k;\lambda_i) - \poi(k;\lambda_i')\big)\Big)^2}{w_0\poi(k;0) + \sum_{i=i_0}^N w_i \poi(k;\lambda_i')}\\
\notag&\leq w_{i_\ast} \cdot\sum_{k = 0}^\infty \frac{\big(\poi(k;\lambda_{i_\ast}) - \poi(k;\lambda_{i_\ast}')\big)\Big)^2}{\poi(k;\lambda_{i_\ast}')}\\
&= w_{i_\ast}\cdot \chi^2(\poi(\lambda_{i_\ast})||\poi(\lambda_{i_\ast}')) = w_{i_\ast}\Big(\exp\Big((\lambda_{i_\ast} - \lambda'_{i_\ast})^2/\lambda'_{i_\ast}\Big) - 1\Big).
\end{align}
Using $(\lambda_{i_\ast} - \lambda'_{i_\ast})^2/\lambda'_{i_\ast} = \delta_{i_\ast}^2/\lambda_{i_\ast}' \leq \delta_{i_\ast}^2/a_{i_\ast} \leq (nw_{i_\ast}(\log n)^{10})^{-1}$, the lower bound $w_{i_\ast} \geq 2/n$, and $\exp(x) - 1\leq 2x$ for $x\in(0,1/2)$, we have $\chi^2(f_{\bm{\tau}}||f_{\bm{\tau}')} \leq 2/\big(n(\log n)^{10}\big)$. 

Next, to lower bound the ratio $H^2(f_{\bm{\tau}},f_{\bm{\tau}'})/d(\bm{\tau},\bm{\tau}')$, we have (recall the convention $\lambda_0\equiv 0$)

\begin{align*}
\frac{1}{2}H^2(f_{\bm{\tau}},f_{\bm{\tau}'}) &= 1 - \sum_{k=0}^\infty \sqrt{\Big(w_0\poi(k;0) + \sum_{i=i_0}^N w_i\poi(k;\lambda_i)\Big)\Big(w_0\poi(k;0) + \sum_{i=i_0}^N w_i\poi(k;\lambda'_i)\Big)}\\
&\geq 1 - w_0\sum_{k=0}^\infty \poi(k;0) - \sum_{i=i_0}^N w_i\cdot \sum_{k=0}^\infty  \sqrt{\poi(k;\lambda_i)\poi(k;\lambda'_i)}\\
&\qquad - \sum_{\substack{i\neq j\\i,j\in\{0\}\cup [i_0,N]}}\sqrt{w_iw_j}\cdot \sum_{k=0}^\infty \sqrt{\poi(k;\lambda_i)\poi(k;\lambda'_j)}\\
&= \sum_{i=i_0}^N w_iH^2\big(\poi(\lambda_i), \poi(\lambda_i')\big) - \sum_{\substack{i\neq j\\i,j\in\{0\}\cup [i_0,N]}}\sqrt{w_iw_j}\cdot \sum_{k=0}^\infty \sqrt{\poi(k;\lambda_i)\poi(k;\lambda'_j)}\\
&\geq d(\bm{\tau},\bm{\tau}')\cdot \min_{i_0\leq i \leq N: \lambda_i \neq \lambda_i'} w_i\cdot H^2\big(\poi(\lambda_i), \poi(\lambda_i')\big) -  \sum_{\substack{i\neq j\\i,j\in\{0\}\cup [i_0,N]}}\sqrt{w_iw_j}\cdot \sum_{k=0}^\infty \sqrt{\poi(k;\lambda_i)\poi(k;\lambda'_j)}.
\end{align*}
For each $\lambda_i \neq \lambda_i'$, we have (see Lemma \ref{lem:poi_divergence} in \prettyref{app:aux})
\begin{align*}
&w_i\cdot H^2\big(\poi(\lambda_i), \poi(\lambda_i')\big) = w_i\Big(1 - \exp\big(-(\sqrt{\lambda_i} - \sqrt{\lambda_i'})^2/2\big)\Big)\\
&\geq w_i\bigg(1 - \exp\Big(-\frac{(\lambda_i - \lambda_i')^2}{8(\lambda_i\vee \lambda_i')}\Big)\bigg) = w_i\bigg(1-\exp\Big(-\frac{\delta_i^2}{8(\lambda_i\vee\lambda_i')}\Big)\bigg)\\
&\gtrsim w_i\frac{\delta_i^2}{a_i} = \frac{1}{n(\log n)^{10}}.
\end{align*}
On the other hand, for any $i\neq j$ and $C > 0$, we have $|\sqrt{\lambda_i} - \sqrt{\lambda_j'}| \geq \sqrt{C\log n}$ for all sufficiently large $n$, hence
\begin{align*}
\sum_{k=0}^\infty \sqrt{\poi(k;\lambda_i)\poi(k;\lambda'_j)} = e^{-\frac{(\sqrt{\lambda_i} - \sqrt{\lambda_j'})^2}{2}} \leq n^{-C/2}. 
\end{align*}
Combining the above two estimates yields that, for any $d(\bm{\tau},\bm{\tau}') \geq 1$, 
\begin{align*}
H^2(f_{\bm{\tau}},f_{\bm{\tau}'}) \geq d(\bm{\tau},\bm{\tau}')\cdot \frac{c_0}{n(\log n)^{10}} - N^2 \cdot n^{-C/2} \gtrsim d(\bm{\tau},\bm{\tau}')/(n(\log n)^{10}),
\end{align*}
as long as
\begin{align}\label{ineq:density_lower_4}
N \leq n^\rho
\end{align}
for some $\rho > 0$, and $C > 0$ is large enough (depending on $\rho$). Finally, by choosing $N+1 = c_pn^{1/(2p+1)}M_p^{1/(2p+1)}/\log n$ and $i_0 = (c_p/3)n^{1/(2p+1)}M_p^{1/(2p+1)}/\log n$ for some small $c_p > 0$, Assouad's Lemma yields that the minimax $H^2$-risk is at least proportional to 
$N/(n(\log n)^{10})\asymp n^{-2p/(2p+1)}M_p^{1/(2p+1)}(\log n)^{-11}$. It remains to note that the above choice of $(i_0,N)$ satisfy the conditions in (\ref{ineq:density_lower_1}), (\ref{ineq:density_lower_2}), (\ref{ineq:density_lower_3}), and (\ref{ineq:density_lower_4}) under the aforementioned condition $n^{-1/p}(\log n)^{10} \leq M_p^{1/p} \leq n^2(\log n)^2$.
\qed

\section{Proofs for Section \ref{sec:regret}}\label{sec:proof_regret}

\subsection{Proof of Theorem \ref{thm:regret_lower}}\label{subsec:proof_regret_lower}


We will divide the lower bound proof into two parts: (i) the lower bound $n^{-2(p-1)/(2p+1)}M_p^{3/(2p+1)}(\log n)^{-11}$ for $p\geq 1$; (ii) a refined lower bound $M_1$ (without the logarithmic factor) for $p = 1$.

\begin{proof}[Proof of Theorem \ref{thm:regret_lower}: $p\geq1$]
The proof is similar to that of Theorem \ref{thm:density_lower}, and uses the same lower construction therein. A technical hurdle for applying Assouad's lemma is that the regret $\reg_n$ involves $\|\cdot\|_{\ell_2(f_G)}$ where the weight $f_G$ depends on the parameter $G$ itself, which, as such, does not satisfy a generalized triangle inequality. To this end, we will relate the $\|\cdot\|_{\ell_2(f_G)}$ loss to a loss function that is independent of $G$ and then apply Assouad's lemma to this new loss. 

We proceed with the proof of Theorem \ref{thm:density_lower} till (\ref{ineq:density_lower_chi}) and continue with the following arguments. 
Recall that $\theta_G(k) \equiv \E_G(\theta|Y = k)$ is the Bayes rule. For any $\theta:\integers_+\rightarrow \reals_+$ and $G,G'\in \{G_{\bm{\tau}}\}$ 
where the prior $G_{\bm{\tau}}$ defined in \eqref{def:lower_bound_construction} is indexed by a binary vector $\bm{\tau}$, define
\begin{align*}
\pnorm{\theta - \theta_{G}}{\ell_2(f_{G'})}^{2,\mathsf{trun}} \equiv \sum_{i=i_0}^{N}\sum_{k\in R_i} \big(\theta(k) - \theta_G(k)\big)^2f_{G'}(k),
\end{align*}
where $R_i\subset I_i$ is to be chosen later. Let $\mathcal{I}\subset [i_0,N]$ such that $\lambda_i\neq \lambda_i'$ for $i\in\mathcal{I}$. Then with the shorthand $f_{\bm{\tau}} = f_{G_{\bm{\tau}}}$, 
\begin{align*}
\pnorm{\theta_{G_{\bm{\tau}}} - \theta_{G_{\bm{\tau}'}}}{\ell_2(f_{\bm{\tau}'})}^{2,\mathsf{trun}} &\geq \sum_{i\in\mathcal{I}}\sum_{k\in I_i} \big(\theta_{G_{\bm{\tau}}}(k) - \theta_{G_{\bm{\tau}'}}(k)\big)^2 \cdot \Big(w_0\poi(k;\lambda_0') + \sum_{j=i_0}^N w_j\poi(k;\lambda_j')\Big)\\
&\geq \sum_{i\in\mathcal{I}}w_i\cdot \sum_{k\in R_i}\big(\theta_{G_{\bm{\tau}}}(k) - \theta_{G_{\bm{\tau}'}}(k)\big)^2\cdot\poi(k;\lambda_i').
\end{align*}
For any $G_{\bm{\tau}}$ and $k$, let $w_j(k;G_{\bm{\tau}}) \equiv \Prob_{G_{\bm{\tau}}}(\lambda = \lambda_j | Y = k)$ be the posterior probability. Then $\theta_{G_{\bm{\tau}}}(k) = \sum_{j=i_0}^N w_j(k;G_{\bm{\tau}})\lambda_j$, and for $k\in R_i\subset I_i$,
\begin{align*}
\big(\theta_{G_{\bm{\tau}}}(k) - \theta_{G_{\bm{\tau}'}}(k)\big)^2 &\geq \Big(w_i(k;G_{\bm{\tau}})\lambda_i - w_i(k;G_{\bm{\tau}'})\lambda_i'\Big)^2 - C\Big(\sum_{j:j\neq i} w_j(k;G_{\bm{\tau}})\lambda_j + w_j(k;G_{\bm{\tau}'})\lambda'_j\Big)^2\\
&\geq (\lambda_i - \lambda_i')^2 - C'\Big(1-w_i(k;G_{\bm{\tau}}) \wedge w_i(k;G_{\bm{\tau}'})\Big) \cdot ((N+1)\log n)^4\\
&= \frac{a_i}{nw_i(\log n)^{10}} - C'\Big(1-w_i(k;G_{\bm{\tau}}) \wedge w_i(k;G_{\bm{\tau}'})\Big) \cdot ((N+1)\log n)^4.
\end{align*} 
Choose $R_i \equiv \{k\in I_i: (k - b_i)^2/k \leq (800)^{-1}(\log n)^2\}\subset I_i$. Then for such $k\in I_i$, we have $(k-a_i)^2/k \leq 2\big((k-b_i)^2 + (a_i - b_i)^2\big)/k \leq (400)^{-1}(\log n)^2 + 2\delta_i^2/(i\log n)^2 \leq (200)^{-1}(\log n)^2$, so that $(k-a_i)^2/k\vee (k-b_i)^2/k \leq (200)^{-1}(\log n)^2$ for $k\in R_i \subset I_i$. We claim that for any $j\in\{0\}\cup[i_0,N]$ that $j\neq i$ and $k\in R_i$, $w_j(k;G) \leq n^{-C}$ for $G\in\{G_{\bm{\tau}}, G_{\bm{\tau}'}\}$. To see this, note that using $w_i \geq 2/n$,
\begin{align}\label{eq:posterior_prob}
w_j(k;G) = \frac{\poi(k;\lambda_j)w_j}{\poi(k;\lambda_0)w_0 + \sum_{\ell=i_0}^N \poi(k;\lambda_\ell)w_\ell} \leq \frac{n\cdot \poi(k;\lambda_j)}{\poi(k;\lambda_i)}
\end{align}
Using the Poisson tail in Lemma \ref{lem:poi_basic}\ref{poi_basic1} and recall that $\lambda_j \in \{a_j,b_j\} \subset [(j\log n)^2 + |I_j|/4, (j\log n)^2 + 3|I_j|/4] \subset I_j$, we have $\poi(k;\lambda_j) = 0$ if $j = 0$, and if $j> 0$, $\poi(k;\lambda_j) \leq \Prob(|\poi(\lambda_j) - \lambda_j| \geq |I_j|/4) \leq 2\exp(-(50)^{-1}(\log n)^2)$. On the other hand, using Stirling approximation, we have for $\lambda_i\in\{a_i, b_i\}$,
\begin{align*}
\poi(k;{\lambda_i}) = \frac{\lambda_i^ke^{-\lambda_i}}{k!} &\geq \exp\bigg(k\log\Big(1+\frac{\lambda_i - k}{k}\Big) + (k-\lambda_i) - \log\sqrt{2\pi k} - 1/(12k)\bigg)\\
&\geq \exp\bigg(-(\lambda_i - k)^2/k - \log\sqrt{2\pi k} - 1/(12k)\bigg)\\
&\geq \exp\bigg(-2(\lambda_i - k)^2/k \bigg)/\sqrt{2\pi k} \geq \exp(-(100)^{-1}(\log n)^2)/\sqrt{2\pi k}.
\end{align*}
Combining the above two estimates yields the claim $w_j(k;G) \leq n^{-C}$ for $G\in\{G_{\bm{\tau}}, G_{\bm{\tau}'}\}$, any $j\neq i$, and $k\in R_i$. By choosing $C$ to be large enough, this implies $\big(\theta_{G_{\bm{\tau}}}(k) - \theta_{G_{\bm{\tau}'}}(k)\big)^2 \geq a_i/(nw_i(\log n)^{10}) - n^{-100}$, and hence
\begin{align}\label{ineq:L2_hamming_ratio}
\notag\pnorm{\theta_{G_{\bm{\tau}}} - \theta_{G_{\bm{\tau}'}}}{\ell_2(f_{\bm{\tau}'})}^{2,\mathsf{trun}} &\geq \sum_{i\in\mathcal{I}}w_i\cdot \sum_{k\in R_i} \Big(\frac{a_i}{nw_i(\log n)^{10}} - n^{-100}\Big)\poi(k;\lambda_i')\\
&\geq \sum_{i\in\mathcal{I}} \frac{a_i}{n(\log n)^{10}}\cdot \Prob\big(\poi(\lambda_i') \in R_i\big) - n^{-100} \gtrsim |\mathcal{I}|\cdot \frac{\min_{i\in\mathcal{I}}a_i}{n(\log n)^{10}}.
\end{align}
Here the last inequality follows by the Poisson tail bound in Lemma \ref{lem:poi_basic}\ref{poi_basic1} and noting that $\bar{R}_i \equiv \{k\in I_i: (k - a_i)^2/k \leq (1600)^{-1}(\log n)^2\}\subset R_i$ so that $\Prob\big(\poi(a_i) \in R_i\big) \geq \Prob\big(\poi(a_i) \in \bar{R}_i\big) \gtrsim 1$ and similarly for $\Prob\big(\poi(b_i) \in R_i\big)$.

Next we establish the ratio bound: for some $M > 0$,
\begin{align}\label{ineq:local_ratio_bound}
\max_{i_0\leq i\leq N}\max_{k\in R_i} \max_{G_{\bm{\tau}},G_{\bm{\tau}'}} \frac{f_{G_{\bm{\tau}}}(k)}{f_{G_{\bm{\tau}'}}(k)} \leq M.
\end{align}
Fix any $i_*\in[i_0, N]$, $k \in I_{i_*}$, and ${\bm{\tau}},{\bm{\tau}'}$. We have
\begin{align*}
\frac{f_{G_{\bm{\tau}}}(k)}{f_{G_{\bm{\tau}'}}(k)} = \frac{w_0\poi(k;\lambda_0) + \sum_{i=i_0}^N w_i\poi(k;\lambda_i)}{w_0\poi(k;\lambda_0') + \sum_{i=i_0}^N w_i\poi(k;\lambda_i')} \leq \frac{\poi(k;\lambda_{i_*})}{\poi(k;\lambda_{i_*}')} + n^{-100},
\end{align*}
where the inequality follows from the computation following (\ref{eq:posterior_prob}). For distinct $\lambda_{i_*}, \lambda_{i_*}'\in\{a_{i_*}, b_{i_*}\}$, 
\begin{align*}
\frac{\poi(k;\lambda_{i_*})}{\poi(k;\lambda_{i_*}')} = \frac{\lambda_{i_*}^{k}/k!\cdot \exp(-\lambda_{i_*})}{(\lambda'_{i_*})^{k}/k!\cdot \exp(-\lambda'_{i_*})} = \exp\Big(k\log(\lambda_{i_*}/\lambda_{i_*}') - (\lambda_{i_*} - \lambda_{i_*}')\Big).
\end{align*}
If $\lambda_{i_*} = b_{i_*}\geq a_{i_*} = \lambda_{i_*}'$, then using $\log(1+x)\leq x$, the exponent can be bounded by
\begin{align*}
\Big|k\frac{b_{i_*}-a_{i_*}}{a_{i_*}} - (b_{i_*} - a_{i_*})\Big| = \Big|\delta_{i_*}\Big(\frac{k}{a_{i_*}} - 1\Big)\Big| \lesssim \frac{\delta_{i_*}}{\sqrt{a_{i_*}}}\frac{|I_{i_*}|}{\sqrt{a_{i_*}}} \lesssim \frac{1}{(\log n)^5}\frac{i_*(\log n)^2}{i_*(\log n)} \lesssim 1. 
\end{align*}
If $\lambda_{i_*} = a_{i_*}\leq b_{i_*} = \lambda_{i_*}'$, then using $\log(1+x) - x \geq -x^2$ for $x\in(0,1/2)$, the exponent can be bounded by
\begin{align*}
\Big|-k\log(\frac{b_{i_*} - a_{i_*}}{a_{i_*}} + 1) + (b_{i_*} - a_{i_*})\Big| &\leq \Big|(b_{i_*} - a_{i_*})\Big(1 - \frac{k}{a_{i_*}}\Big)\Big| + k\Big(\frac{b_{i_*} - a_{i_*}}{a_{i_*}}\Big)^2\\
&\lesssim \frac{\delta_{i_*}}{\sqrt{a_{i_*}}}\frac{|I_{i_*}|}{\sqrt{a_{i_*}}} + \frac{k}{a_{i_*}}\frac{\delta_{i_*}^2}{a_{i_*}} \lesssim 1;
\end{align*}
note that we indeed have $(b_{i_*} - a_{i_*})/a_{i_*} = (nw_{i_*}(\log n)^{10})^{-1/2} \leq 1/2$ for large enough $n$. Putting together the two cases, we have established the claim (\ref{ineq:local_ratio_bound}).

With these preparations, we are ready to apply Assouad's lemma. Using the condition (\ref{ineq:local_ratio_bound}), we have
\begin{align}\label{ineq:minimax_lower_1}
\inf_{\hat{\theta}}\sup_{f_G\in\mathcal{H}_p} \E_{f_G}\pnorm{\hat{\theta} - \theta_G}{\ell_2(f_G)}^2 \gtrsim_M \inf_{\hat{\theta}}\max_{f_G:G\in\{G_{\bm{\tau}}\}} \E_{f_G}\pnorm{\hat{\theta} - \theta_{G}}{\ell_2(f_{G_{\bm{0}}})}^{2,\mathsf{trun}},
\end{align}
where $\bm{0}$ is the zero vector with the same length as $\bm{\tau}$.
For any estimator $\hat{\theta}$, let its associating $\hat{G}$ in $\{G_{\bm{\tau}}\}$ be given by
\begin{align*}
\hat{G} \equiv \argmin_{G\in \{G_{\bm{\tau}}\}} \pnorm{\hat{\theta} - \theta_G}{\ell_2(f_{G_{\bm{0}}})}^{2,\mathsf{trun}}.
\end{align*}
Pick any such $\hat{G}$ if the minimum is not unique. Then for any $G\in \{G_{\bm{\tau}}\}$, 
\begin{align*}
\pnorm{\theta_{\hat{G}} - \theta_G}{\ell_2(f_{G_{\bm{0}}})}^{2,\mathsf{trun}} &= \sum_{i=i_0}^{N} \sum_{k\in R_i} \big(\theta_{\hat{G}}(k) - \theta_G(k)\big)^2f_{G_{\bm{0}}}(k)\\
&\leq 2\pnorm{\theta_{\hat{G}} - \hat{\theta}}{\ell_2(f_{G_{\bm{0}}})}^{2,\mathsf{trun}}  + 2\pnorm{\theta_{G} - \hat{\theta}}{\ell_2(f_{G_{\bm{0}}})}^{2,\mathsf{trun}}  \leq 4\pnorm{\theta_{G} - \hat{\theta}}{\ell_2(f_{G_{\bm{0}}})}^{2,\mathsf{trun}}.
\end{align*}
Continuing with (\ref{ineq:minimax_lower_1}), we have 
\begin{align*}
&\inf_{\hat{\theta}}\sup_{f_G\in\mathcal{H}_p} \E_{f_G}\pnorm{\hat{\theta} - \theta_G}{\ell_2(f_G)}^2 \gtrsim_M \inf_{\hat{G}\in\{G_{\bm{\tau}}\}}\max_{G\in\{G_{\bm{\tau}}\}}\pnorm{\theta_{\hat{G}} - \theta_{G}}{\ell_2(f_{G_{\bm{0}}})}^{2,\mathsf{trun}}\\
&\gtrsim_M N\cdot\min_{\bm{\tau}\neq \bm{\tau}'}\frac{\pnorm{\theta_{G_{\bm{\tau}}} - \theta_{G_{\bm{\tau}'}}}{\ell_2(f_{G_{\bm{\tau}'}})}^{2,\mathsf{trun}}}{d(\bm{\tau}, \bm{\tau}')} \cdot \min_{\bm{\tau},\bm{\tau}':d(\bm{\tau},\bm{\tau}') = 1}\bigg(1 - \sqrt{\frac{n}{2}\chi^2(f_{\bm{\tau}}||f_{\bm{\tau}'})}\bigg),
\end{align*}
where the second inequality follows from Assouad's lemma (see, e.g., \cite[Theorem 2.12(iv)]{Tsybakov09}). By the same choices of $(i_0,N)$ as in Theorem \ref{thm:density_lower}: $N = c_p n^{1/(2p+1)}M_p^{1/(2p+1)}/\log n$ for some small $c_p$ and $i_0 = (c_p/3) n^{1/(2p+1)}M_p^{1/(2p+1)}/\log n$, and combining (\ref{ineq:density_lower_chi}) and (\ref{ineq:L2_hamming_ratio}), we obtain the lower bound rate $n^{-(2p-2)/(2p+1)}M_p^{3/(2p+1)}(\log n)^{-11}$. The proof is complete. 
\end{proof}

\begin{proof}[Proof of Theorem \ref{thm:regret_lower}: $p = 1$] 
The proof is based on a simple two-point argument. Let $a = (M_1^{-1}\vee 1)n^5$, $b = a + \sqrt{a/M_1}/100$, and $G_u \equiv (1-u^{-1})\delta_0 + u^{-1}\delta_{u\cdot M_1}$ for $u\in\{a,b\}$, both with first moment equal to $M_1$. First note that, for $G_u$, the Bayes estimator is
\begin{align*}
\theta_{G_u}(y) = \frac{uM_1\cdot e^{-uM_1}}{(u-1) + e^{-uM_1}}\indc{y = 0} + (uM_1)\indc{y>0},
\end{align*}
so the Bayes risk $\mathsf{mmse}(G_u)$ with $u\in\{a,b\}$ equals
\begin{align*}
&\E_{G_u} \big(\theta_{G_u}(Y) - \theta\big)^2\\
&= (1-u^{-1}) \E_{G_u}\big[\big(\theta_{G_u}(Y) - 0\big)^2|\theta = 0\big] + u^{-1}\E_{G_u}\big[\big(\theta_{G_u}(Y) - uM_1\big)^2|\theta = uM_1\big]\\
&=  (1-u^{-1}) \cdot \Big(\frac{uM_1e^{-uM_1}}{(u-1) + e^{-uM_1}}\Big)^2\\
&\qquad + \frac{1}{u}\Big[e^{-uM_1}\Big(\frac{uM_1e^{-uM_1}}{(u-1) + e^{-uM_1}} - uM_1\Big)^2 + (1-e^{-uM_1})(uM_1 - uM_1)^2\Big]\\
&= \frac{u(u-1)M_1^2e^{-uM_1}}{u-1 + e^{-uM_1}} \leq M_1\cdot (uM_1)e^{-uM_1} = o(M_1),
\end{align*}
where we use the fact that $uM_1 \gtrsim n^5$. Hence in order to prove an $\Omega(M_1)$ lower bound for the regret, it suffices to show the same lower bound for the risk. To this end, we have for $\tilde{\theta}_{Y^{n-1}}$ that is measurable with respect to $Y^{n-1}$, 
\begin{align*}
&\inf_{\tilde{\theta}_{Y^{n-1}}}\sup_G \E_G(\tilde{\theta}_{Y^{n-1}}(Y_n) - \theta_n)^2\\
&\gtrsim \inf_{\tilde{\theta}_{Y^{n-1}}} \Big[\E_{G_a}(\tilde{\theta}_{Y^{n-1}}(Y_n) - \theta_n)^2 + \E_{G_b}(\tilde{\theta}_{Y^{n-1}}(Y_n) - \theta_n)^2\Big]\\
&\gtrsim \frac{1}{a}\inf_{\tilde{\theta}_{Y^{n-1}}} \Big[\E_{Y^{n-1}\sim G_a}\E_{U\sim\poi(aM_1)}(\tilde{\theta}_{Y^{n-1}}(U) - aM_1)^2 + \E_{Y^{n-1}\sim G_b}\E_{U\sim\poi(bM_1)}(\tilde{\theta}_{Y^{n-1}}(U) - bM_1)^2\Big]\\
&\stackrel{\rm (a)}{\gtrsim} \frac{1}{a}\inf_{\tilde{\theta}_{Y^{n-1}}} \Big[\E_{U\sim\poi(aM_1)}(\tilde{\theta}_{Y^{n-1} = 0}(U) - aM_1)^2 + \E_{U\sim\poi(bM_1)}(\tilde{\theta}_{Y^{n-1} = 0}(U) - bM_1)^2\Big]\\
&\geq \frac{1}{a}\inf_{f} \Big[\E_{U\sim\poi(aM_1)}(f(U) - aM_1)^2 + \E_{U\sim\poi(bM_1)}(f(U) - bM_1)^2\Big]\\
&\gtrsim \frac{(a-b)^2M_1^2}{a}\Big(1-\textrm{TV}\big(\poi(aM_1), \poi(bM_1)\big)\Big) \stackrel{\rm (b)}{\gtrsim} M_1.
\end{align*}
Here in (a), we use the fact that under both $G_a$ and $G_b$, the event $\{Y^{n-1} = 0\}$ holds with probability at least $1/2$; in (b), we use Lemma \ref{lem:poi_divergence} in \prettyref{app:aux} along with the inequality $\textrm{TV}(P,Q) \leq H(P,Q)$ for any distributions $P,Q$. The proof is complete. 
\end{proof}

\subsection{Proof of Theorem \ref{thm:regret_MLE_EB}}\label{subsec:proof_NPMLE}

As mentioned near the end of \prettyref{subsec:NPMLE}, a key step of the regret analysis is to introduce a sequence $\{A_k\}$ that facilitates the control of the difficult term (\ref{ineq:A1_bound_rough}) appearing in the regret bound. 
To this end, we start with some notations. For any $y\in \integers_+$, $\rho \in \reals_+$, and two distributions $G_1, G_2$, let 
\begin{align}
w(y)\equiv w(y;G_1,G_2,\rho) \equiv \frac{1}{f_{G_1}(y)\vee \rho + f_{G_2}(y)\vee\rho}.
\label{eq:wstar}
\end{align}
For any $k\geq 0$, define
\begin{align}\label{def:Ak}
A_k^2 \equiv A_k^2(G_1, G_2;\rho) \equiv \sum_{y=0}^{\infty} (y+1)^{k}\Big(\Delta^kf_{G_1}(y) - \Delta^kf_{G_2}(y)\Big)^2w(y).
\end{align}
Here $\Delta^k$ is the $k$th-order forward difference operator defined in (\ref{eq:forwarddiff}), and so $A_k^2$ can be interpreted as a squared distance between the 
$k$th order ``discrete derivatives'' of the Poisson mixture with an appropriate weight function that also (crucially) depends on $k$.
The role of this sequence $\{A_k\}$ is the following:
\begin{itemize}
	\item The $k=0$ term corresponds to the squared Hellinger distance between the mixtures. In fact, it is easy to show that $A_0^2(G_1,G_2;\rho) \lesssim H^2(f_{G_1}, f_{G_2})$. 
	
	\item The $k=1$ term is the key in bounding the regret, which, as will soon become clear, boils down to controlling $A_1(H, G_0;\rho)$, where $G_0$ is the true prior, $H$ is the estimated prior used in $g$-modeling (e.g., the NPMLE (\ref{def:npmle})), and $\rho$ is the regularization parameter in (\ref{def:g_smoothed_individual}).
	Since directly bounding $A_1$ is difficult, we will achieve this goal with the aid of higher-order terms.
	
	\item We show that the growth of the sequence $\{A_k\}$ is at most $A_k \lesssim (Ck)^k/\rho$ (\prettyref{prop:Ak_upper}).
	
	\item We derive a recursive inequality relating each $A_k$ to its neighboring terms (\prettyref{prop:Ak_recursion}), which, combined with the boundary conditions at $k=0$ and $k=\Theta(\log n)$, allows us to tightly control the target $A_1$ with appropriately chosen $\rho=n^{-\Theta(1)}$ (\prettyref{prop:A1_bound}).
\end{itemize}

%
%

In the sequel we prove the pointwise bound and the recursive bound on the sequence $\{A_k\}$ in Sections \ref{subsubsec:Ak_pointwise} and \ref{subsubsec:Ak_recursion} respectively, before finishing the proof of Theorem \ref{thm:regret_MLE_EB} in Section \ref{subsubsec:complete}.


\subsubsection{Pointwise bound on $\{A_k\}$}
\label{subsubsec:Ak_pointwise}

\begin{proposition}\label{prop:Ak_upper}
For any distributions $G_1,G_2$ and $\rho > 0$, the following holds.
\begin{align*}
A_k^2(G_1,G_2;\rho) \leq 4k^{k}/\rho.
\end{align*}
\end{proposition}

Before proceeding to the proof of \prettyref{prop:Ak_upper}, let us first explain the subtleties in the argument. Because of the polynomial factor $(y+1)^k$ in (\ref{def:Ak}), it is not even clear a priori whether $A_k$ is finite for moderate to large $k$. In fact, applying the binomial expansion \prettyref{eq:backwarddiff-expand} 
of the $k$th-order backward difference and the triangle inequality only works when $G_1,G_2$ have finite $k$th moments, which cannot be afforded when $k$ is large as we are working with priors with potentially heavy tails.
This suggests that it is crucial to take into account the cancellation thanks to the finite difference operator $\Delta^k$, which offsets the growth of $(y+1)^k$. 
Indeed, the proof below applies the structure of the Poisson mixture and relates $\Delta^k f_G$ to the discrete orthogonal polynomials under the Poisson weights \cite{szego1975orthogonal}. 
For even $k$, a self-contained proof based on Fourier and Laplace transforms is given in Appendix \ref{sec:second_proof}.
(This suffices for proving the main Proposition \ref{prop:A1_bound} as we can choose $k_0$ there to be an even number.)

For any $\theta > 0$, the \emph{Poisson-Charlier polynomial} is defined as 
\begin{align}\label{def:charlier}
p_k(y;\theta) \equiv \frac{\theta^{k/2}}{\sqrt{k!}}\frac{\nabla^k \poi(y;\theta)}{\poi(y;\theta)}, \quad y\in \integers_+,
\end{align}
where $\{\nabla^k\}_{k\geq 0}$ is the backward difference operator in (\ref{eq:backwarddiff}).
It is well-known \cite[Section 2.81]{szego1975orthogonal} that $\{p_k(y;\theta)\}_{k\geq 0}$ is a system of orthonormal polynomials under the $\poi(\theta)$ distribution:
\begin{align}\label{eq:charlier_ortho}
\sum_{y=0}^\infty p_k(y;\theta)p_\ell(y;\theta)\poi(y;\theta) = \indc{k=\ell}.
\end{align} 
We are now ready to present the proof of Proposition \ref{prop:Ak_upper}.
\begin{proof}[Proof of Proposition \ref{prop:Ak_upper}]
We first show that for every $G$,
\begin{align}\label{ineq:Ak_upper_general}
\sum_{y=k}^\infty (y-k+1)^k \big(\nabla^k f_G(y)\big)^2 \leq 2k!.
\end{align}
For any $G$, let $\alpha \equiv G(\{0\})$ be its mass on $0$ and $\bar{G}$ be its conditional version on $(0,\infty)$, so that
\begin{align*}
G = \alpha\delta_0 + (1-\alpha)\bar{G}. 
\end{align*}
Using the definition in (\ref{def:charlier}), we have for any $y\geq k$,
\begin{align*}
\nabla^k f_G(y) &= \E_{\theta\sim G} \nabla^k \poi(y;\theta)\\
&= \alpha \cdot \nabla^k \poi(y;0) + (1-\alpha)\cdot\E_{\theta\sim \bar{G}}\big[p_k(y;\theta)\sqrt{k!}\theta^{-k/2}\poi(y;\theta)\big]\\
&= \alpha (-1)^k\indc{y=k} + (1-\alpha)\cdot\E_{\theta\sim \bar{G}}\big[p_k(y;\theta)\sqrt{k!}\theta^{-k/2}\poi(y;\theta)\big],
\end{align*}
where the first term in the last step applies the expansion \prettyref{eq:backwarddiff-expand}.
Hence 
\begin{align*}
&\frac{1}{k!}\sum_{y=k}^\infty (y-k+1)^k \big(\nabla^k f_G(y)\big)^2\\
&\leq 1 + \sum_{y=k}^\infty (y-k+1)^k\big(\E_{\theta\sim \bar{G}}[p_k(y;\theta)\theta^{-k/2}\poi(y;\theta)]\big)^2\\
&\leq 1 + \E_{\theta\sim \bar{G}}\Big[\sum_{y=k}^\infty  \big(p_k(y;\theta)\big)^2 \poi(y;\theta) \cdot \theta^{-k}\poi(y;\theta)(y-k+1)^k\Big]\\
&\stackrel{\rm (a)}{\leq} 1 + \E_{\theta\sim \bar{G}}\Big[\sum_{y=k}^\infty  \big(p_k(y;\theta)\big)^2\poi(y;\theta)\Big] \stackrel{\rm (b)}{\leq} 2,
\end{align*}
where (a) follows from the fact that for any $y\geq k$,
\begin{align*}
\theta^{-k}\poi(y;\theta)(y-k+1)^k =  \frac{(y-k+1)^k}{(y-k+1)(y-k+2)\ldots y}\poi(y-k;\theta) \leq 1,
\end{align*}
and (b) follows from (\ref{eq:charlier_ortho}). This proves (\ref{ineq:Ak_upper_general}).

Now we apply (\ref{ineq:Ak_upper_general}) to bound $A_k$. Using $\Delta^k f_G(y) = \nabla^k f_G(y+k)$, we have
\begin{align*}
A_k^2 &\leq \rho^{-1}\Big[\sum_{y=0}^\infty (y+1)^k \big(\Delta^k f_{G_1}(y)\big)^2 + \sum_{y=0}^\infty (y+1)^k \big(\Delta^k f_{G_2}(y)\big)^2\Big]\\
&= \rho^{-1}\Big[\sum_{y=0}^\infty (y+1)^k \big(\nabla^k f_{G_1}(y+k)\big)^2 + \sum_{y=0}^\infty (y+1)^k \big(\nabla^k f_{G_2}(y+k)\big)^2 \Big]\\
&= \rho^{-1}\Big[\sum_{y=k}^\infty (y-k+1)^k \big(\nabla^k f_{G_1}(y)\big)^2 + \sum_{y=k}^\infty (y-k+1)^k \big(\nabla^k f_{G_2}(y)\big)^2\Big]\\
&\leq 4\rho^{-1}k! \leq 4k^k/\rho.
\end{align*}
The proof is complete.
\end{proof}

\subsubsection{Recursive bound on $\{A_k\}$}
\label{subsubsec:Ak_recursion}


The following is a recursive inequality for the sequence $A_k=A_k(G_1,G_2;\rho)$ defined in (\ref{def:Ak}).

\begin{proposition}\label{prop:Ak_recursion}
For any $k\geq 1$,
\begin{align}\label{ineq:Ak_recursion}
A_k^2 \leq L_kA_kA_{k-1} + A_{k-1}A_{k+1},
\end{align}
where $L_k = C\log\frac{1}{\rho} + k$ for some universal constant $C > 0$.
\end{proposition}

We need the following lemma which provides a tight bound for the (centered) Bayes estimator uniformly over all priors. This lemma may be of independent interest, and we present its proof after that of Proposition \ref{prop:Ak_recursion}.
\begin{lemma}\label{lem:bayes_form_upper}
There exists a universal constant $C > 0$ such that the following holds.
For any prior $G$ and any $y\in\integers_+$, 
\begin{align}\label{ineq:bayes_upper_1}
|\theta_G(y)-y| \leq \E(|\theta - Y||Y = y) &\leq C\sqrt{y\vee 1}\log\frac{1}{f_G(y)},
\end{align}
where the expectation is taken over $\theta\sim G$ and $Y\sim \Poi(\theta)$ and $\theta_G(y) = \E[\theta|Y=y]$.
Consequently, for any $\rho\leq 1/e$ and $y\geq 0$,
\begin{align}
\frac{|f_G(y+1)-f_G(y)|}{f_G(y)\vee \rho}
=\frac{|\Delta f_G(y)|}{f_G(y)\vee \rho}
\leq \frac{C+1}{\sqrt{y+1}} \log\frac{1}{\rho}.
\label{eq:R1R3mom}
\end{align}
\end{lemma}

\begin{proof}[Proof of Proposition \ref{prop:Ak_recursion}]
Let $h(y) \equiv f_{G_1}(y) - f_{G_2}(y)$. Applying the summation by parts formula \prettyref{eq:SBP}, we have
\begin{align*}
A_k^2 &= \sum_{y=0}^{\infty} (y+1)^{k}\Delta^k h(y)w(y) \cdot\Big(\Delta^{(k-1)}h(y+1) - \Delta^{(k-1)}h(y)\Big)\\ 
&= \sum_{y=0}^{\infty} y^{k}\Delta^k h(y-1)w(y-1)\Delta^{(k-1)}h(y) -  \sum_{y=0}^{\infty} (y+1)^{k}\Delta^k h(y)w(y)\Delta^{(k-1)}h(y)\\
&= - \sum_{y=0}^{\infty}\Delta^{(k-1)}h(y)\cdot \Delta\Big(y^{k}\Delta^kh(y-1)w(y-1)\Big),
\end{align*}
where 
\begin{align}
\Delta\Big(y^{k}\Delta^kh(y-1)w(y-1)\Big) = (y+1)^{k}\Delta^k h(y)w(y) - y^{k}\Delta^kh(y-1)w(y-1).
\label{eq:Delyk}
\end{align}
Here we use the convention $y^{k}\Delta^kh(y-1)w(y-1) = 0$ when $y = 0$. With 
\[
\bar{w}(y) \equiv \frac{2w(y-1)w(y)}{w(y) + w(y-1)},
\]
 the harmonic mean of $w(y)$ and $w(y-1)$, \prettyref{eq:Delyk} can be bounded by
\begin{align*}
|\Delta\Big(y^{k}\Delta^kh(y-1)w(y-1)\Big)|
&\leq \Big|\big((y+1)^{k} - y^{k}\big) \Delta^kh(y) w(y)\Big| + y^{k}\Big| \Delta^kh(y)w(y) - \Delta^kh(y-1)w(y-1)\Big|\\
&\leq \Big|\big((y+1)^{k} - y^{k}\big) w(y)\Delta^k h(y)\Big| + y^{k}\bigg\{\Big|\Delta^kh(y)\big(w(y) - \bar{w}(y)\big)\Big|\\
&\qquad + \Big|\Big(\Delta^kh(y) - \Delta^kh(y-1)\Big)\bar{w}(y)\Big)\Big| + \Big|\Delta^kh(y-1)\big(w(y-1) - \bar{w}(y)\big)\Big|\bigg\}.
\end{align*}
This implies that $A_k^2 \leq \sum_{i=1}^4 S_i$, where
\begin{align*}
S_1 &\equiv \sum_{y=0}^{\infty} \Big|\Delta^{(k-1)}h(y)\Big|\Big|\big((y+1)^{k} - y^{k}\big) w(y)\Delta^k h(y)\Big|,\\
S_2 &\equiv \sum_{y=0}^{\infty} y^{k}\cdot\Big|\Delta^{(k-1)}h(y) \cdot\Delta^kh(y)\Big|    \cdot  |w(y) - \bar{w}(y)|,\\
S_3 &\equiv \sum_{y=0}^{\infty} y^{k}\cdot\Big|\Delta^{(k-1)}h(y)\Big|\Big|\Delta^{(k+1)}h(y-1)\Big| \cdot \bar{w}(y),\\
S_4 &\equiv \sum_{y=0}^{\infty} y^{k}\cdot\Big|\Delta^{(k-1)}h(y)\Big|\Big|\Delta^kh(y-1)\Big|\cdot |w(y-1) - \bar{w}(y)|.
\end{align*}
Now we bound these four terms separately. Using $(y+1)^k - y^k = \int_y^{y+1} kx^{k-1}\d x\leq k(y+1)^{k-1}$, $S_1$ satisfies
\begin{align*}
S_1 &\leq k\cdot\sum_{y=0}^{\infty} \Big|\Delta^{(k-1)}h(y)\Big|\Big|(y+1)^{k-1} w(y)\Delta^k h(y)\Big|\\
&\leq  k\cdot\sum_{y=0}^{\infty} \Big|\Delta^{(k-1)}h(y)\Big|\Big|(y+1)^{k-1/2} w(y)\Delta^k h(y)\Big|\\
&\leq k\cdot \Big(\sum_{y=0}^{\infty} (y+1)^{k-1}\big(\Delta^{(k-1)}h(y)\big)^2 w(y)\Big)^{1/2}\Big(\sum_{y=0}^{\infty} (y+1)^{k}\big(\Delta^kh(y)\big)^2 w(y)\Big)^{1/2}\\
&= k\cdot A_kA_{k-1}.
\end{align*}

To bound $S_2$, note that 
for any prior $G$, applying \prettyref{eq:R1R3mom} in Lemma \ref{lem:bayes_form_upper} yields, for any $y\geq 1$,
\begin{align}
\big|f_G(y) - f_G(y-1)\big| \lesssim 
\frac{f_G(y-1) \vee \rho}{\sqrt{y}} \log\frac{1}{\rho}.  
\label{eq:fGfG}
\end{align}
Recall from \prettyref{eq:wstar} that $w(y)= \frac{1}{f_{G_1}(y)\vee \rho + f_{G_2}(y)\vee\rho}$. 
Then for any $y\geq 1$,
\begin{align*}
|w(y) - \bar{w}(y)| &= \frac{w(y)|w(y) - w(y-1)|}{w(y) + w(y-1)}\\
& = \frac{w^2(y)w(y-1)}{w(y) + w(y-1)} |f_{G_1}(y)\vee \rho + f_{G_2}(y)\vee\rho - f_{G_1}(y-1)\vee \rho - f_{G_2}(y-1)\vee\rho| \\
& \leq w(y)
\pth{ \frac{|f_{G_1}(y)-f_{G_1}(y-1)|}{f_{G_1}(y-1)\vee \rho} + \frac{|f_{G_2}(y)-f_{G_2}(y-1)|}{f_{G_2}(y-1)\vee \rho}} \\
&\stepa{\lesssim} \frac{Cw(y)}{\sqrt{y}} \log\frac{1}{\rho},
\end{align*}
where (a) applies \prettyref{eq:fGfG}.
Applying this and Cauchy-Schwarz, we have, for some universal $C > 0$, 
\begin{align*}
S_2 &\leq C\log\frac{1}{\rho}\cdot\sum_{y=1}^{\infty} y^{k-1/2}\cdot\Big|\Delta^{(k-1)}h(y)\Delta^kh(y)\Big|w(y)\\
&\leq  C\log\frac{1}{\rho}\cdot\Big(\sum_{y=0}^{\infty} (y+1)^{k-1}\big(\Delta^{(k-1)}h(y)\big)^2w(y)\Big)^{1/2}\Big(\sum_{y=0}^{\infty} (y+1)^k\big(\Delta^kh(y)\big)^2w(y)\Big)^{1/2}\\
&= C\log\frac{1}{\rho}\cdot A_{k}A_{k-1}.
\end{align*}

For $S_3$, using $\bar{w}(y) \leq \sqrt{w(y)w(y-1)}$ we get
\begin{align*}
S_3 &\leq \Big(\sum_{y=0}^{\infty} y^{k-1}\big(\Delta^{(k-1)}h(y)\big)^2w(y)\Big)^{1/2}\Big(\sum_{y=0}^\infty y^{k+1}\big(\Delta^{(k+1)}h (y-1)\big)^2w(y-1)\Big)^{1/2}\\ 
&\leq \Big(\sum_{y=0}^{\infty} (y+1)^{k-1}\big(\Delta^{(k-1)}h(y)\big)^2w(y)\Big)^{1/2}\Big(\sum_{y=0}^\infty (y+1)^{k+1}\big(\Delta^{(k+1)}h (y)\big)^2w(y)\Big)^{1/2}\\ 
&= A_{k-1}A_{k+1}.
\end{align*}

The bound for $S_4$ is similar to $S_2$: using
\begin{align*}
\big|w(y-1) - \bar{w}(y)\big| &= \frac{w(y-1)\big|w(y-1) - w(y)\big|}{w(y) + w(y-1)}\\
&= \frac{w(y)w^2(y-1)}{w(y)+w(y-1)}\big|f_{G_1}(y)\vee\rho + f_{G_2}(y)\vee \rho - f_{G_1}(y-1)\vee\rho - f_{G_2}(y-1)\vee\rho\big|\\
&\leq w^{1/2}(y)w^{1/2}(y-1)\Big(\frac{|f_{G_1}(y) - f_{G_1}(y-1)|}{f_{G_1}(y-1)\vee \rho} + \frac{|f_{G_2}(y) - f_{G_2}(y-1)|}{f_{G_2}(y-1)\vee \rho}\Big)\\
&\lesssim \frac{Cw^{1/2}(y)w^{1/2}(y-1)}{\sqrt{y}}\log\frac{1}{\rho},
\end{align*}
we apply a similar argument as in $S_2$ to obtain
\begin{align*}
S_4 \leq C\log\frac{1}{\rho}\cdot A_{k}A_{k-1}.
\end{align*}
Assembling the estimates of $S_1$--$S_4$ yields the desired recursion \prettyref{ineq:Ak_recursion} for $\{A_k\}$.
\end{proof}

\begin{proof}[Proof of Lemma \ref{lem:bayes_form_upper}]
We start with the decomposition
\begin{align*}
\E_G(|\theta - Y||Y = y) = \E_G(|\theta - Y|\indc{\theta \leq 2y}|Y = y) + \E_G(|\theta - Y|\indc{\theta > 2y}|Y = y).
\end{align*}

We first prove the bound
\begin{align}\label{ineq:bayes_upper_11}
\E_G(|\theta - Y|\indc{\theta \leq 2y}|Y = y) &\leq C \sqrt{y \cdot \log\frac{1}{f_G(y)}}.
\end{align}
When $y = 0$, both sides equal to $0$ and there is nothing to prove, so we assume $y\geq 1$. Fix any $\tau\in (0,1/2]$, then by Jensen's inequality, we have
\begin{align*}
\E_G(|\theta - Y|\indc{\theta \leq 2y}|Y = y) &\leq \sqrt{\frac{2y}{\tau} \E_G\Big[\Big(\frac{\tau(\theta - Y)^2}{2y}\Big)\indc{\theta\leq 2y}|Y = y\Big]}\\
&\leq \sqrt{\frac{2y}{\tau} \log \E_G\Big[\exp\Big(\frac{\tau(\theta - Y)^2}{2y}\Big)\indc{\theta\leq 2y}|Y = y\Big]}.
\end{align*}
Using the Stirling approximation,  the inner expectation can be bounded as
\begin{align*}
&\E_G\Big[\exp\Big(\frac{\tau(\theta - Y)^2}{2y}\Big)\indc{\theta\leq 2y}|Y = y\Big]\\
&= \frac{1}{f_G(y)}\int_{\theta \in (0,2y]} \exp\Big(\frac{\tau(\theta - y)^2}{2y}\Big)\cdot \frac{\theta^ye^{-\theta}}{y!} \d G(\theta)\\
&\lesssim \frac{1}{f_G(y)}\cdot \frac{1}{\sqrt{y}} \cdot \int_{\theta \in(0,2y]} \exp\Big(\frac{\tau(\theta - y)^2}{2y}\Big)\cdot \frac{\theta^ye^{-\theta}}{y^ye^{-y}} \d G(\theta)\\
&= \frac{1}{f_G(y)}\cdot \frac{1}{\sqrt{y}} \cdot \E_G \exp\Big(\frac{\tau(\theta - y)^2}{2y} - (\theta - y) + y\log(\theta/y)\Big)\indc{\theta\in(0, 2y]},
\end{align*}
where $\theta = 0$ is excluded in the integral since the posterior probability of $\theta = 0$ given $y\geq 1$ is zero. Since $y\geq 1$, it suffices to show the above exponent is non-positive for all $\theta\leq 2y$. Let $z\equiv \theta - y$, so that this exponent equals
\begin{align*}
M(z) \equiv \frac{\tau z^2}{2y} - z + y\log\Big(1 + \frac{z}{y}\Big).
\end{align*}
If $z \leq 0$, using $\log(1+x) \leq x - x^2/2$ for $x\in (-1, 0]$ and $z/y = \theta/y - 1 \in (-1, 0]$, 
\begin{align*}
M(z) \leq  \frac{\tau z^2}{2y} - \frac{z^2}{2y} \leq \frac{z^2}{2y}(\tau - 1) \leq 0.
\end{align*}
If $z > 0$, we have
\begin{align*}
M(z) = \int_0^z \Big(\frac{\tau t}{y} - 1 + \frac{y}{y + t}\Big) \d t = \int_0^z t\Big(\frac{\tau}{y} - \frac{1}{y+t}\Big)\d t \leq 0,
\end{align*}
using the fact $\theta\leq 2y$ and $\tau \leq 1/2$ in the last step. This concludes the bound (\ref{ineq:bayes_upper_11}).

Next we prove the bound 
\begin{align}\label{ineq:bayes_upper_2}
\E_G(|\theta - Y|\indc{\theta > 2y}|Y = y) &\leq C\log\frac{1}{f_G(y)}.
\end{align}
We first assume $y\geq 1$. Then we similarly have
\begin{align*}
\E_G(|\theta - Y|\indc{\theta > 2y}|Y = y) &\leq \frac{1}{\tau} \log \E_G\Big[\exp\big(\tau(\theta - Y)\big)\indc{\theta > 2y}|Y = y\Big].
\end{align*}
The inner expectation can be computed as
\begin{align*}
&\E_G\Big[\exp\big(\tau(\theta - Y)\big)\indc{\theta > 2y}|Y = y\Big]\\
&= \frac{1}{f_G(y)}\int_{\theta > 2y} \exp\big(\tau(\theta - y)\big)\cdot \frac{\theta^ye^{-\theta}}{y!} \d G(\theta)\\
&\lesssim \frac{1}{f_G(y)}\cdot \frac{1}{\sqrt{y}} \cdot \int_{\theta > 2y} \exp\big(\tau(\theta - y)\big)\cdot \frac{\theta^ye^{-\theta}}{y^ye^{-y}} \d G(\theta)\\
&= \frac{1}{f_G(y)}\cdot \frac{1}{\sqrt{y}} \cdot \E \exp\Big(-(\tau-1)(\theta - y) + y\log(\theta/y)\Big)\indc{\theta > 2y},
\end{align*}
With $u\equiv (z-y)/y \in (1,\infty)$ and choosing $\tau < 0.1$, the exponent is
\begin{align*}
M(u) = y\big(-(1-\tau)u + \log(1+u)\big) \leq 0,
\end{align*} 
proving the bound (\ref{ineq:bayes_upper_2}) for $y\geq 1$. On the other hand, the bound still holds for $y = 0$ because
\begin{align*}
\E_G\Big[\exp\big(\tau(\theta - Y)\big)\indc{\theta > 0}|Y = 0\Big] = \frac{1}{f_G(y)}\int_{\theta > 0} \exp\big(-(1-\tau)\theta\big)\d G(\theta) \leq \frac{1}{f_G(y)}.
\end{align*}
Combining (\ref{ineq:bayes_upper_11}) and (\ref{ineq:bayes_upper_2}) completes the proof of (\ref{ineq:bayes_upper_1}).

Finally, to show \prettyref{eq:R1R3mom},
\begin{align*}
\Big|(y+1)\frac{\Delta f_G(y)}{f_G(y)\vee \rho}\Big| 
&= \Big|\big(\theta_G(y) - (y+1)\big)\cdot\frac{f_G(y)}{f_G(y)\vee \rho}\Big|\\
&\stepa{\leq}
C \sqrt{y\vee 1} \log\frac{1}{f_G(y)}\cdot \frac{f_G(y)}{f_G(y) \vee \rho} + 1 \\
& \stepb{\leq} C \sqrt{y\vee 1}\log\frac{1}{\rho} + 1 \leq (C+1) \sqrt{y \vee 1}\log\frac{1}{\rho},
\end{align*}
where (a) applies \eqref{ineq:bayes_upper_1}; 
(b) uses the fact that $x\mapsto x\log(1/x)$ is increasing in $(0,1/e)$ so that $\max_{t\geq 0} \frac{t}{t\vee \rho}\log(1/t) = \log\frac{1}{\rho}$ for $\rho\leq 1/e$.
\end{proof}

We are now ready to state and prove our main bound of $A_1$ defined in (\ref{def:Ak}), by combining the estimates from Propositions \ref{prop:Ak_upper} and \ref{prop:Ak_recursion}. Recall that $H^2(\cdot,\cdot)$ is the squared Hellinger distance.
\begin{proposition}\label{prop:A1_bound}
For any $\rho \leq 1/e$, there exists some universal $C > 0$ such that
\begin{align*}
A_1^2(G_1,G_2;\rho) \leq C\big((\log 1/\rho)^4\cdot H^2(f_{G_1},f_{G_2}) + \rho^{10}\big),
\end{align*}
uniformly over distributions $G_1,G_2$. 
\end{proposition}
\begin{proof}[Proof of Proposition \ref{prop:A1_bound}]
Define $\gamma_k \equiv A_k/A_{k-1}$. Then the recursion in Proposition \ref{prop:Ak_recursion} yields that for any $k\geq 1$, with $L_k = C\log\frac{1}{\rho} + k$ for some universal $C > 0$,
\begin{align*}
\gamma_k \leq L_k + \gamma_{k+1}.
\end{align*}
Let $K > 0$ to be chosen later. Define $k_0 \equiv \log (1/\rho)$, and we discuss two cases. 
\paragraph{Case (i): $\gamma_k \leq K$ for some $k\in[k_0]$.} Then
\begin{align*}
\gamma_1 \leq \sum_{\ell=1}^{k-1} L_{\ell} + \gamma_k \leq C(\log\frac{1}{\rho}k \vee k^2) + K\leq C(\log 1/\rho)^2 + K. 
\end{align*}
This implies that $A_1^2 = A_0^2 \gamma_1^2 \lesssim \big((\log 1/\rho)^4 + K^2\big)\cdot H^2(f_{G_1}, f_{G_2})$, by noting that
\begin{align*}
A_0^2 = \sum_{y=0}^{\infty} \frac{\big(f_{G_1}(y) - f_{G_2}(y)\big)^2}{f_{G_1}(y)\vee \rho + f_{G_2}(y) \vee \rho} \lesssim  \sum_{y=0}^{\infty} \big(\sqrt{f_{G_1}(y)} - \sqrt{f_{G_2}(y)}\big)^2 = H^2(f_{G_1}, f_{G_2}).
\end{align*}

\paragraph{Case (ii): $\gamma_k > K$ for all $k\in [k_0]$.} Then $A_{k_0}/A_1 = \prod_{\ell=2}^{k_0} \gamma_\ell \geq K^{k_0 - 1}$, which, when combined with the bound in Proposition \ref{prop:Ak_upper}, implies
\begin{align*}
A_1 \leq K^{-(k_0 - 1)}A_{k_0} \leq (2/\sqrt{\rho}) k_0^{k_0/2} \cdot K^{-(k_0 - 1)} \leq \rho^{10},
\end{align*}
by choosing $K$ to be a large constant multiple of $k_0$. Collecting the two bounds completes the proof.
\end{proof}

\subsubsection{Completing the proof}
\label{subsubsec:complete}

We need a further technical lemma bounding the regret of the MLE $Y$ in the scalar Poisson EB model $\theta \sim G$ and $Y|\theta \sim \poi(\theta)$. Recall that $\theta_G(y) \equiv \E_G[\theta|Y=y]$ is the Bayes estimator associate with the prior $G$.
\begin{lemma}\label{lem:MLE_regret}
There exist some universal $C,c > 0$ such that 
for any $p\geq 1, y_0\geq 1$, and $M_p > 0$,  and any $G$ with $p$th moment $m_p(G) \leq M_p$,
\begin{align*}
\E_{Y\sim f_G}\big[\big(Y - \theta_G(Y)\big)^2\indc{Y\geq y_0}\big] \leq C\Big(M_p^{1/p}\exp(-cy_0) + M_p y_0^{-(p-1)}\Big).
\end{align*}
\end{lemma}
\begin{proof}[Proof of Lemma \ref{lem:MLE_regret}]
Let $\theta\sim G$ and $Y\sim\Poi(\theta)$. 
By the orthogonality property of the Bayes estimator $\theta_G(\cdot)$, we have, for any measurable functions $h(Y)\geq 0$ and $\tilde{\theta}(Y)$, 
\begin{align*}
\E_G\big[(\tilde{\theta}(Y) - \theta)^2h(Y)\big] = \E_G\big[(\tilde{\theta}(Y) - \theta_G(Y))^2h(Y)\big] + \E_G\big[(\theta_G(Y) - \theta)^2h(Y)\big]
\end{align*}
provided that all expectations are finite. 
Applying this with $\tilde{\theta}(y) = y$ and $h(y) = \indc{y\geq y_0}$, we have 
\begin{align}\label{eq:mle_error}
\notag&\E\big[(Y - \theta_G(Y))^2\indc{Y\geq y_0}\big] \leq  \E\big[(Y - \theta)^2\indc{Y\geq y_0}\big]\\
\notag&= \E\big[(Y - \theta)^2\indc{Y\geq y_0}\indc{\theta\leq y_0/2}\big] + \E\big[(Y - \theta)^2\indc{Y\geq y_0}\indc{\theta > y_0/2}\big]\\
\notag&\lesssim \E_{\theta\sim G} \Big[\indc{\theta \leq y_0/2} \cdot \theta \cdot \sqrt{\Prob(Y \geq y_0|\theta)}\Big] + \E\big[(Y - \theta)^2\indc{\theta > y_0/2}\big]\\
&\stackrel{\rm (a)}{\lesssim} M^{1/p}_p\exp(-cy_0) + \E \big(\theta\indc{\theta > y_0/2}\big) \stackrel{\rm (b)}{\lesssim}  M_p^{1/p}\exp(-cy_0) + M_p y_0^{-(p-1)},
\end{align}
where (a) follows from the Poisson tail bound in Lemma \ref{lem:poi_basic}\ref{poi_basic1} and $\Expect_G[\theta] \leq (\Expect_G[\theta^p])^{1/p} \leq M_p^{1/p}$; (b) follows from Markov's inequality and the condition $m_p(G) \leq M_p$. The proof is complete.
\end{proof}

We are now ready to complete the proof of Theorem \ref{thm:regret_MLE_EB}.

\begin{proof}[Proof of Theorem \ref{thm:regret_MLE_EB}]
In the proof, we will abbreviate $\hat{\theta}^{\mathsf{g}}_n(Y_n;H)$ as $\hat{\theta}(Y_n;H)$. Fix an integer $y_0 \geq 1$ to be optimized. We first condition on $Y^{n-1}$ so that $H$ is fixed. Then
\begin{align*}
\E_{Y_n \sim f_G} \big(\hat{\theta}(Y_n;H) - \theta_G(Y_n)\big)^2 &= \Big(\sum_{y=0}^{y_0} + \sum_{y=y_0+1}^\infty\Big) f_G(y)\Big(\theta_H(y;\rho) - \theta_G(y)\Big)^2\\
&\equiv (I) + (II). 
\end{align*}

We first bound $(II)$. We have
\begin{align*}
(II) \lesssim \sum_{y=y_0+1}^\infty  f_G(y)\Big(\theta_H(y;\rho) - y\Big)^2 + \sum_{y=y_0+1}^\infty  f_G(y)\Big(y - \theta_G(y)\Big)^2.
\end{align*}
By Lemma \ref{lem:bayes_form_upper}, 
\begin{align*}
\sum_{y=y_0+1}^\infty  f_G(y)\Big(\theta_H(y;\rho) - y\Big)^2 \leq \log^2(1/\rho) \cdot \E[Y\bm{1}_{Y > y_0}] \lesssim \log^2(1/\rho) \cdot M_py_0^{-(p-1)}.
\end{align*}
On the other hand, the second term is bounded by a constant multiple of $M_p^{1/p}\exp(-cy_0) + M_p y_0^{-(p-1)}$ by Lemma \ref{lem:MLE_regret}, so
\begin{align}\label{ineq:regret_II}
(II) \lesssim M_p^{1/p}\exp(-cy_0) + \log^2(1/\rho) \cdot M_p y_0^{-(p-1)}.
\end{align}


Next we bound $(I)$. We can further decompose it as
\begin{align*}
(I) &=  \sum_{y=0}^{y_0} f_G(y)(y+1)^2\bigg(\frac{\Delta f_{H}(y)}{f_{H}(y) \vee \rho} - \frac{\Delta f_G(y)}{f_G(y)}\bigg)^2\\
&\lesssim {\sum_{y=0}^{y_0} f_G(y)(y+1)^2\bigg(\frac{\Delta f_{H}(y)}{f_{H}(y) \vee \rho} - \frac{\Delta f_G(y)}{f_G(y)\vee \rho}\bigg)^2} + {\sum_{y=0}^{y_0} f_G(y)(y+1)^2\Big[\frac{\Delta f_G(y)}{f_G(y)}\Big(1 - \frac{f_G(y)}{\rho}\Big)_+\Big]^2}\\
&\equiv r_1+r_2.
\end{align*}
For the second term,
\begin{align*}
r_2 \leq 
\sum_{y\in[0,y_0]: f_G(y)\leq \rho} f_G(y)(y+1)^2\Big[\frac{\Delta f_G(y)}{f_G(y)}\Big]^2
&= \sum_{y\in[0,y_0]: f_G(y)\leq \rho} f_G(y)\Big[\theta_G(y) - (y+1)\Big]^2\\
&\stepa{\lesssim} \sum_{y\in[0,y_0]: f_G(y)\leq \rho} f_G(y) 
(y+1) \pth{\log\frac{1}{f_G(y)}}^2 \\
&\stepb{\lesssim} y_0^2 \rho \log^2 \frac{1}{\rho} \numberthis \label{eq:r2}
\end{align*}
where (a) applies \prettyref{lem:bayes_form_upper} and (b) applies the monotonicity of $x \mapsto x (\log \frac{1}{x})^2$ on $[0,\rho]$ for sufficiently small $\rho$.

%
The first term is decomposed as
\begin{align*}
r_1 
&\lesssim \sum_{y=0}^{y_0} f_G(y)(y+1)^2\bigg[\bigg(\frac{\Delta f_{H}(y)}{f_{H}(y)\vee \rho} - \frac{2\Delta f_{H}(y)}{f_{H}(y)\vee\rho + f_G(y)\vee\rho}\bigg)^2 \\
&\qquad +\bigg(\frac{2\big(\Delta f_{H}(y)-\Delta f_G(y)\big)}{f_{H}(y)\vee \rho + f_G(y)\vee \rho}\bigg)^2 + \bigg(\frac{2\Delta f_G(y)}{f_{H}(y)\vee \rho + f_G(y)\vee\rho} - \frac{\Delta f_G(y)}{f_G(y)\vee \rho}\bigg)^2\bigg]\\
&\equiv R_1 + R_2 + R_3.
\end{align*}
For $R_3$, we have
\begin{align*}
R_3 &= \sum_{y=0}^{y_0} f_G(y)(y+1)^2 \bigg(\frac{2\Delta f_G(y)}{f_{H}(y)\vee \rho + f_G(y)\vee\rho} - \frac{\Delta f_G(y)}{f_G(y)\vee \rho}\bigg)^2\\
&=  \sum_{y=0}^{y_0} f_G(y)\Big((y+1)\frac{\Delta f_G(y)}{f_G(y)\vee \rho}\Big)^2 \frac{\big(f_{H}(y)\vee\rho - f_G(y)\vee\rho\big)^2}{\big(f_{H}(y)\vee\rho + f_G(y)\vee\rho\big)^2}\\
&\lesssim \sum_{y=0}^{y_0} \Big((y+1)\frac{\Delta f_G(y)}{f_G(y)\vee \rho}\Big)^2 \cdot \Big(\sqrt{f_{H}(y)} - \sqrt{f_G(y)}\Big)^2,
\end{align*}
where the last inequality follows from the fact that for any $a,b,\rho\in[0,1]$,
\[
a\pth{\frac{a\vee \rho - b \vee \rho}{a\vee \rho + b \vee \rho}}^2 
\leq 
\frac{(a\vee \rho - b \vee \rho)^2}{a\vee \rho + b \vee \rho} \leq 
2 (\sqrt{a\vee \rho} -\sqrt{ b \vee \rho})^2 \leq 2 (\sqrt{a} -\sqrt{ b })^2.
\]
Applying \prettyref{eq:R1R3mom}, we have for any $y\leq y_0$, 
\begin{align*}
\Big|(y+1)\frac{\Delta f_G(y)}{f_G(y)\vee \rho}\Big| 
\lesssim \sqrt{y_0}\log\frac{1}{\rho}, \numberthis \label{eq:R1R3}
\end{align*}
As a result,
\begin{align}\label{ineq:R3_bound}
R_3 \lesssim y_0 \pth{\log \frac{1}{\rho}}^2 \sum_{y=0}^{y_0}  \Big(\sqrt{f_{H}(y)} - \sqrt{f_G(y)}\Big)^2 \leq y_0 \pth{\log \frac{1}{\rho}}^2 H^2(f_{H}, f_G).
\end{align}
By an entirely analogous argument, we also have
\begin{align}\label{ineq:R1_bound}
R_1 \lesssim y_0 \pth{\log \frac{1}{\rho}}^2 H^2(f_{H}, f_G).
\end{align}

Finally we bound $R_2$, which corresponds to the key step outlined in (\ref{ineq:A1_bound_rough}). 
Recall $w(y) = w(y;H,G, \rho) = \big(f_{G}(y)\vee \rho + f_{H}(y)\vee\rho\big)^{-1}$ as defined in (\ref{eq:wstar})
Recall also the sequence $\{A_k\equiv A_k(H,G;\rho)\}$ in (\ref{def:Ak}); in particular, 
$A_1^2 = \sum_{y=0}^{\infty} (y+1)\Big(\Delta f_{G}(y) - \Delta f_{H}(y)\Big)^2w(y)$.
Then
\begin{align*}
R_2 &\lesssim \sum_{y=0}^{y_0} (y+1)^2 \cdot \Big[\big(\Delta f_{H}(y) - \Delta f_G(y)\big)^2 w(y)\Big]\\
&\leq (y_0+1) \cdot \Big[\sum_{y=0}^{\infty} (y+1)\big(\Delta f_{H}(y) - \Delta f_G(y)\big)^2 w(y)\Big] = (y_0+1) \cdot A_1^2(H,G;\rho),
\end{align*}
Applying (the crucial) Proposition \ref{prop:A1_bound} yields
\begin{align}\label{ineq:R2_bound}
R_2 \lesssim y_0(\log 1/\rho)^4\cdot H^2(f_{H}, f_G) + y_0 \cdot \rho^{10}.
\end{align}
Combining \eqref{eq:r2} and (\ref{ineq:R3_bound})--(\ref{ineq:R2_bound}), we obtain
\begin{align*}
(I) \lesssim y_0(\log 1/\rho)^4\cdot H^2(f_{H}, f_G) + y_0 \cdot \rho^{10} + y_0^2 \rho \log^2 \frac{1}{\rho},
\end{align*}
which together with the bound in (\ref{ineq:regret_II}) implies the bound in (\ref{ineq:g_main_bound}).

Consequently, if $\E H^2(f_{\hat{G}}, f_G) \leq c_1n^{-2p/(2p+1)}M_p^{1/(2p+1)}(\log n)^\kappa$ for some positive $c_1$ and $\kappa$ uniformly over $G \in \mathcal{G}_p(M_p)$, by taking expectation of both sides of (\ref{ineq:g_main_bound}) with respect to $Y^{n-1}$ and choosing $\rho \asymp n^{-10}$, we have 
\begin{align}\label{ineq:npmle_main}
\reg(\hat{\theta}^{\mathsf{g}}(Y_n;H); \mathcal{G}_p(M_p)) &\lesssim_{c_1}  (\log n)^{\kappa+4}\cdot \Big[y_0 \cdot \big(n^{-2p/(2p+1)}M_p^{1/(2p+1)}\big) + M_p y_0^{-(p-1)}\Big] + \bar{\mathcal{R}}(y_0)
\end{align}
for any integer $y_0 \geq 1$, where
\begin{align}\label{def:npmle_remainder}
\bar{\mathcal{R}}(y_0) \equiv O\bigg(M_p^{1/p}\exp(-cy_0) + n^{-10} y_0  + n^{-9}y_0^2\bigg).
\end{align}
By choosing $y_0 = \floor{n^{2/(2p+1)}M_p^{2/(2p+1)}}$, the first two terms in \prettyref{ineq:npmle_main} are both bounded by \\ $O(n^{-2(p-1)/(2p+1)}M_p^{3/(2p+1)}(\log n)^{\kappa+4})$. Finally, for $\bar{\mathcal{R}}(y_0)$, we have
\begin{itemize}
\item Under the condition $M_p^{1/p} \geq n^{-1/p}(\log n)^{10}$, the first term satisfies 
\begin{align*}
M_p^{1/p}\cdot \exp(-cy_0) = O(n^{-2(p-1)/(2p+1)}M_p^{3/(2p+1)}(\log n)^{\kappa+4});
\end{align*}
\item Under the same condition, the second and third term are also bounded by the same order as above.
\end{itemize}
The proof is complete.
\end{proof}

\subsection{Proof of Theorem \ref{thm:npmle_rho_free}} \label{subsec:proof_npmle_special}
\begin{proof}
Let $M>1$ be specified later and
\begin{align}\label{eq:rho_choice}
\rho =  \frac{n^{-C}}{\sqrt{M}n}
\end{align}
for some large $C = C(p) > 0$. For any prior $G$ in $\mathcal{G}_p(M_p)$, we have
\begin{align*}
&\E_G \pnorm{\hat{\theta}^{\npmle,n}(Y^n) - \theta^n}{}^2 - \pnorm{\theta_G(Y^n) - \theta^n}{}^2\\
&=  \E_G \pnorm{\theta_{\hat{G}}(Y^n) - \theta_G(Y^n)}{}^2\\
&= \E_G \sum_{i=1}^n \big(\theta_{\hat{G}}(Y_i) - \theta_G(Y_i)\big)^2\bm{1}_{Y_i \leq M} +  \underbrace{\E_G \sum_{i=1}^n \big(\theta_{\hat{G}}(Y_i) - \theta_G(Y_i)\big)^2\bm{1}_{Y_i > M}}_{R_1}\\
&\stackrel{(*)}{=} \E_G \sum_{i=1}^n \big(\theta_{\hat{G}}(Y_i;\rho) - \theta_G(Y_i)\big)^2\bm{1}_{Y_i \leq M} + R_1\\
&\leq 2\E_G \pnorm{\theta_{\hat{G}}(Y^n;\rho) - \theta_G(Y^n;\rho)}{}^2 + 2\underbrace{\E_G \sum_{i=1}^n\big(\theta_G(Y_i;\rho) - \theta_G(Y_i)\big)^2\bm{1}_{Y_i \leq M}}_{R_2} + R_1\\
&\leq 4 \E_G\pnorm{\theta_{\hat{G}}(Y^n;\rho) - \theta_{G_n}(Y^n)}{}^2  + 4\underbrace{\E_{G}  \pnorm{\theta_{G}(Y^n;\rho) - \theta_{G_n}(Y^n)}{}^2 }_{R_3} + R_1+2R_2,
\end{align*}
where $G_n = n^{-1}\sum_{i=1}^n \delta_{\theta_i}$ denotes the empirical distribution of $\theta^n$, and $(*)$ follows from Lemma \ref{lem:density_lower_npmle} and the choice of $\rho$ in (\ref{eq:rho_choice}).
In summary, we have
\begin{align}\label{eq:npmle_reg_decomp}
\totreg_n(\hat{\theta}^{\npmle,n};\mathcal{G}_p(M_p)) \lesssim \sup_{G\in \mathcal{G}_p(M_p)}\Big[\E_G \pnorm{\theta_{\hat{G}}(Y^n;\rho) - \theta_{G_n}(Y^n)}{}^2 + \sum_{i=1}^3 R_i\Big]. 
\end{align}

We first bound $R_1$-$R_3$. For $R_1$, by symmetry we have
\begin{align*}
n^{-1}R_1 \lesssim \E_G\big(\theta_{\hat{G}}(Y_n) - Y_n\big)^2\bm{1}\{Y_n > M\} + \E_G\big(\theta_{G}(Y_n) - Y_n\big)^2\bm{1}\{Y_n > M\}.
\end{align*}
For the first term, using $f_{\hat{G}}(Y_n) \gtrsim (Y_n \vee 1)^{-1/2}n^{-1}$ from Lemma \ref{lem:density_lower_npmle} and Lemma \ref{lem:bayes_form_upper}, 
\begin{align*}
\E_G\big(\theta_{\hat{G}}(Y_n) - Y_n\big)^2\bm{1}\{Y_n > M\} &\lesssim \E_G \Big(\sqrt{Y_n \vee 1}\log \frac{1}{f_{\hat{G}}(Y_n)}\Big)^2\bm{1}\{Y_n > M\}\\
&\lesssim \E_G \big(Y_n\log^2 (nY_n)\big)\bm{1}\{Y_n > M\} \lesssim M_p \cdot M^{-(p-1)}\log^2(nM). 
\end{align*}
The second term is already bounded by Lemma \ref{lem:MLE_regret} of the paper:
\begin{align*}
\E_G\big(\theta_{G}(Y_n) - Y_n\big)^2\bm{1}\{Y_n > M\} \lesssim M_p^{1/p} \cdot \exp(-cM) + M_p \cdot M^{-(p-1)}, 
\end{align*}
so we have
\begin{align*}
R_1 \lesssim n(\log (nM))^2 \cdot (M_p^{1/p} \cdot \exp(-cM) + M_p \cdot M^{-(p-1)}). 
\end{align*}
The term $R_2 \lesssim nM^2 \rho \log^2(1/\rho)$ by (\ref{eq:r2}). For $R_3$, we apply Theorem \ref{thm:regret_MLE_EB} 
to obtain 
\begin{align*}
R_3 &= n \cdot \E_{\theta^n} \E_{Y\sim f_{G_n}} (\theta_G(Y;\rho) - \theta_{G_n}(Y))^2\\
&\lesssim n \cdot \E_{\theta^n} \inf_{y_0 > 1} \Big[\log^4(1/\rho) \cdot \Big(m_p(G_n) y_0^{-(p-1)} + y_0 H^2(f_G, f_{G_n})\Big) + \mathcal{R}(y_0,\rho)\Big]\\
&\leq n \cdot \inf_{y_0>1} \Big[\log^4(1/\rho) \cdot \Big(\E_{\theta^n} m_p(G_n) y_0^{-(p-1)} + y_0 \E_{\theta^n}H^2(f_G, f_{G_n})\Big) + \mathcal{R}(y_0,\rho)\Big],
\end{align*}
where $m_p(G_n) = \int_{\R_+} u^p G_n(\d u)$, and
\begin{align*}
\mathcal{R}(y_0, \rho) = m_p(G_n)^{1/p}\exp(-c_0y_0) + y_0\rho^{10} + y_0^2\rho\log^2(1/\rho).
\end{align*}
Note that $\E_{\theta^n} m_p(G_n) = m_p(G) \leq M_p$, and $\E_{\theta^n} H^2(f_G, f_{G_n})  \lesssim n^{-\frac{2p}{2p+1}}m_p(G)^{\frac{1}{2p+1}}$ by Lemma \ref{lem:emp_density} below, yielding 
\begin{align*}
R_3 \lesssim n\log^4(1/\rho)\cdot (n^{-\frac{2(p-1)}{2p+1}}M_p^{\frac{3}{2p+1}} \vee n^{-1}).
\end{align*} 
In summary, by choosing $M = n^{C_0}$ for some large $C_0$ and $\rho$ as in (\ref{eq:rho_choice}) with a larger $C$, we have
\begin{align}\label{eq:R_bounds}
R_1 + R_2 + R_3 \lesssim (\log n)^4 (n^{\frac{3}{2p+1}}M_p^{\frac{3}{2p+1}} \vee 1).
\end{align}
Lastly, for the other term in (\ref{eq:npmle_reg_decomp}), we have
\begin{align*}
\E_G \pnorm{\theta_{\hat{G}}(Y^n;\rho) - \theta_{G_n}(Y^n;\rho)}{}^2 = \E_{\theta^n}\E_{Y^n|\theta^n}\pnorm{\theta_{\hat{G}}(Y^n;\rho) - \theta_{G_n}(Y^n;\rho)}{}^2,
\end{align*}
where the inner expectation is taken over the compound setup outlined in Section \ref{subsec:compound_regret}. When $m_p(G_n)\leq n^{10p}$, using (\ref{eq:bound_compound}) (see Remark \ref{rmk:l2_regret}) and $p\geq 1$,
\begin{align*}
&\E_{\theta^n} \bm{1}\{m_p(G_n)\leq n^{10p}\}\E_{Y^n|\theta^n} \pnorm{\theta_{\hat{G}}(Y^n;\rho) - \theta_{G_n}(Y^n;\rho)}{}^2 \\&\lesssim \E_{\theta^n} (\log n)^{13}(n^{\frac{3}{2p+1}}m_p(G_n)^{\frac{3}{2p+1}} \vee 1)\\
&\leq (\log n)^{13}n^{\frac{3}{2p+1}}((\E_{\theta^n} m_p(G_n))^{\frac{3}{2p+1}}  \vee 1)\\
&\leq (\log n)^{13}(n^{\frac{3}{2p+1}}M_p^{\frac{3}{2p+1}} \vee 1).
\end{align*}
On the other hand, by Lemma \ref{lem:bayes_form_upper}, we always have
\begin{align}\label{eq:bound_compound_2}
\E_{Y^n|\theta^n} \pnorm{\theta_{\hat{G}}(Y^n;\rho) - \theta_{G_n}(Y^n;\rho)}{}^2 \leq C(\log n)^2 n\cdot (m_1(G_n) \vee 1),
\end{align}
so using $m_1(G_n) \leq m_p(G_n)^{1/p}$, 
\begin{align*}
&\E_{\theta^n} \bm{1}\{m_p(G_n)> n^{10p}\}\E_{Y^n|\theta^n} \pnorm{\theta_{\hat{G}}(Y^n;\rho) - \theta_{G_n}(Y^n;\rho)}{}^2\\
& \lesssim n(\log n)^2\E_{\theta_n} m_p(G_n)^{1/p}\bm{1}_{m_p(G_n)> n^{10p}} \lesssim n(\log n)^2 M_p n^{-10(p-1)} \lesssim (\log n)^{13}n^{\frac{3}{2p+1}}M_p^{\frac{3}{2p+1}},
\end{align*}
using the condition $M_p \leq n^{10p}$ in the last step. Combining the above two bounds with (\ref{eq:R_bounds}) completes the proof.
\end{proof}

\begin{lemma}\label{lem:density_lower_npmle}
Let $\hat{G}$ be given by (\ref{def:npmle}). There exists some universal $c > 0$ such that almost surely, 
\begin{align*}
f_{\hat{G}}(Y_i) \geq \frac{c}{(Y_i \vee 1)^{1/2} n}, \quad \forall i\in[n].
\end{align*}
\end{lemma}
\begin{proof}
Let $\ell_n(G) = \sum_{i=1}^n \log f_G(Y_i)$. By definition of $\hat{G}$, we have $\ell_n(\hat{G}) \geq \ell_n((1-\epsilon)\hat{G} + \epsilon \delta_\theta)$ for any $\epsilon\in[0,1]$ and $\theta \in \R_+$, implying
\begin{align*}
0 \geq \lim_{\epsilon\rightarrow 0}\frac{\d}{\d \epsilon}\ell_n((1-\epsilon)\hat{G} + \epsilon \delta_\theta) = \sum_{i=1}^n \Big(\frac{\poi(Y_i;\theta)}{f_{\hat{G}}(Y_i)} - 1\Big).
\end{align*}
This entails that for each $i\in[n]$, 
\begin{align*}
f_{\hat{G}}(Y_i) \geq \frac{1}{n}\sup_{\theta\in\R_+} \poi(Y_i;\theta) = \frac{\poi(Y_i;Y_i)}{n} \gtrsim \frac{1}{n\sqrt{Y_i \vee 1}},
\end{align*}
using Stirling's approximation in the last step. 
\end{proof}


\begin{lemma}\label{lem:emp_density}
Suppose $m_p(G) < \infty$ for some $p > 1$. Let $\theta_1,\ldots,\theta_n$ be $n$ iid draws from $G$ and $G_n = n^{-1}\sum_{i=1}^n \delta_{\theta_i}$. Then $\E_{\theta^n} H^2(f_G, f_{G_n}) \leq K n^{-\frac{2p}{2p+1}}m_p(G)^{\frac{1}{2p+1}}$ for some universal $K > 0$.  
\end{lemma}
\begin{proof}
For some $y_0 \in \mathbb{Z}_+$ to be fixed later, we have
\begin{align*}
\E H^2(f_G, f_{G_n}) &= \E \sum_{y=0}^\infty \big(\sqrt{f_G(y)} - \sqrt{f_G(y)}\big)^2 \leq \E \sum_{y=0}^{y_0} \frac{(f_{G_n}(y) - f_G(y))^2}{f_G(y)} + 2\sum_{y>y_0} f_G(y)\\
&\leq \sum_{y=0}^{y_0} \frac{1}{n}\frac{\var_{\theta\sim G}\poi(y;\theta)}{f_G(y)} + 2 \frac{m_p(G)}{y^p} \leq n^{-1}\sum_{y=0}^{y_0} (y+1)^{-1/2} + \frac{m_p(G)}{y^p} \lesssim \frac{\sqrt{y_0}}{n} + \frac{m_p(G)}{y_0^p},
\end{align*}
where $(*)$ uses
\begin{align*}
\frac{\var_{\theta\sim G}\poi(y;\theta)}{f_G(y)} \leq \frac{\E_{\theta\sim G} (\poi(y;\theta))^2}{\E_{\theta\sim G} \poi(y;\theta)} \leq \sup_{\theta\geq 0}\poi(y;\theta) = \frac{y^ye^{-y}}{y!} \lesssim (y+1)^{-1/2}. 
\end{align*}
The result follows by choosing $y_0 \sim (nm_p(G))^{\frac{2}{(2p+1)}}$.
\end{proof}

\subsection{Proof of Theorem \ref{thm:regret_robbins_EB}}\label{subsec:proof_robbins}

\subsubsection{Proof of upper bound}

We need two more technical results before the proof of Theorem \ref{thm:regret_robbins_EB}. For the following lemma, $\bin(n,p)$ denotes the binomial distribution with $n$ trials and success probability $p$.

\begin{lemma}\label{lem:binomial_moments}
Suppose $X(n)\equiv X(n;p) \sim \bin(n,p)$ for some $n \in \integers_+$ and $p\in[0,1]$. Then
\begin{align}
\label{binomial_moment1}\E \Big(\frac{n-X(n)}{X(n) + 1}\Big) &= \frac{1-p}{p}\Prob(X(n)\geq 1),\\
\label{binomial_moment2}\E \Big(\frac{n - X(n)}{(X(n) + 1)^2}\Big) &\asymp \frac{1-p}{p^2}\frac{1}{n+1} \Prob\big(X(n+1)\geq 2\big),\\
\label{binomial_moment3}\var\Big(\frac{n - X(n)}{X(n) + 1}\Big) &\lesssim \Big(\frac{1}{n+2}\cdot\frac{1-p}{p^3}\Prob\big(X(n+2) \geq 3\big)\Big) \wedge n^2.
\end{align}

\end{lemma}
\begin{proof}[Proof of Lemma \ref{lem:binomial_moments}]
For (\ref{binomial_moment1}), we have
\begin{align*}
\E \frac{n-X(n)}{X(n)+1} =  \sum_{k=0}^n {n\choose k+1}p^k(1-p)^{n-k} = \frac{1-p}{p}\Prob(X(n)\geq 1).
\end{align*}

For (\ref{binomial_moment2}), using the fact that $(k+2)^{-2} \leq (k+1)^{-2} \leq 2(k+2)^{-2}$ for all $k\in\integers_+$, the left side equals
\begin{align*}
&\sum_{k=0}^n {n\choose k}p^k(1-p)^{n-k}\frac{n-k}{(k+1)^2} \asymp \sum_{k=0}^{n-1} \frac{n!}{(n-k-1)!(k+2)!}p^k(1-p)^{n-k}\\
&= \frac{1}{n+1}\frac{1-p}{p^2} \sum_{k=2}^{n+1} {n+1\choose k} p^k (1-p)^{n+1-k} = \frac{1}{n+1}\frac{1-p}{p^2}\Prob\big(X(n+1) \geq 2\big).
\end{align*}

For (\ref{binomial_moment3}), using (\ref{binomial_moment1}), we have
\begin{align*}
\var\Big(\frac{n - X(n)}{X(n) + 1}\Big) = \E\Big(\frac{n - X(n)}{X(n) + 1}\Big)^2 - \Big(\frac{1-p}{p}\Big)^2\Prob^2(X(n)\geq 1).
\end{align*}
To compute the second moment, with
\begin{align*}
D \equiv \frac{n-k}{k+1} - \frac{n-k-1}{k+2} = \frac{n+1}{(k+1)(k+2)} \asymp \frac{n+1}{(k+2)(k+3)},
\end{align*}
we have
\begin{align*}
&\E\Big(\frac{n - X(n)}{X(n) + 1}\Big)^2 = \sum_{k=0}^{n-1} \frac{n!}{(k+1)!(n-k-1)!}p^k(1-p)^{n-k}\frac{n-k}{k+1}\\
&= \sum_{k=0}^{n-1} \frac{n!}{(k+1)!(n-k-1)!}p^k(1-p)^{n-k}\Big(\frac{n-k-1}{k+2} + D\Big)\\
&\equiv \sum_{k=0}^{n-2} {n\choose k+2}p^k(1-p)^{n-k} + S = \Big(\frac{1-p}{p}\Big)^2\Prob(X(n)\geq 2) + S.
\end{align*}
We claim that $\Prob(X(n)\geq 2) - \Prob^2(X(n)\geq 1) \leq 0$. Indeed, this quantity equals
\begin{align*}
&\big(1 - \Prob(X(n) = 0) - \Prob(X(n) = 1)\big) - \big(1 - \Prob(X(n) = 0)\big)^2\\
&= \Prob(X(n) = 0) - \Prob(X(n) = 1) - \Prob^2(X(n) = 0)\\
&= (1-p)^n - np(1-p)^{n-1} - (1-p)^{2n}\\
&= (1-p)^{n-1}\big[1- (n+1)p - (1-p)^{n+1}\big] \leq 0.
\end{align*}
Hence the desired variance is bounded by $S$, where
\begin{align*}
S &= \sum_{k=0}^{n-1} {n\choose k+1}p^k(1-p)^{n-k} D \asymp \sum_{k=0}^{n-1} \frac{(n+1)!}{(k+3)!(n-k-1)!}p^k(1-p)^{n-k}\\
&= \frac{1}{n+2}\sum_{k=0}^{n-1}{n+2\choose k+3}p^k(1-p)^{n-k} = \frac{1}{n+2}\cdot\frac{1-p}{p^3}\Prob(X(n+2) \geq 3).
\end{align*}
On the other hand, since $(n - X(n))/(X(n) + 1) \leq n$, its variance is trivially bounded by $n^2$. The proof is complete.
\end{proof}

\begin{lemma}\label{lem:mixture_pmf}
For any $y \geq 0$ and distribution $G$, there exists some universal $K > 0$ such that
\begin{align*}
f_G(y+1) \leq  K\cdot\Big(\sqrt{\frac{\log^2 n}{y+1}} \vee 1\Big) f_G(y) + n^{-10}.
\end{align*}
\end{lemma}
\begin{proof}[Proof of Lemma \ref{lem:mixture_pmf}]
For any $y \geq 0$, 
\begin{align*}
f_G(y+1) = \int \frac{a^{y+1}e^{-a}}{(y+1)!}G(\d a) &= \int_{a:|a - (y+1)|\leq 100\sqrt{(y+1)\log^2 n}} \frac{a^{y+1}e^{-a}}{(y+1)!}G(\d a)\\
&\qquad + \int_{a:|a-(y+1)|>100\sqrt{(y+1)\log^2 n}} \frac{a^{y+1}e^{-a}}{(y+1)!}G(\d a).
\end{align*}
The first term can be bounded by
\begin{align*}
\int_{a:|a - (y+1)|\leq 100\sqrt{(y+1)\log^2 n}} \frac{a^{y+1}e^{-a}}{(y+1)!}G(\d a) &\leq \frac{y+1 + \sqrt{100(y+1)\log^2 n}}{y+1}f_G(y)\\
&\asymp \Big(\sqrt{\frac{\log^2 n}{y+1}} \vee 1\Big) f_G(y).
\end{align*}
For the second term, if $a \leq 100(y+1)\log n$, then with $X \sim \poi(a)$, the Poisson tail bound in Lemma \ref{lem:poi_basic}\ref{poi_basic1} yields that 
\begin{align*}
\frac{a^{y+1}e^{-a}}{(y+1)!} &= \Prob(X = y+1) \leq \Prob(|X - a| \geq |y+1 - a|)\\
&\leq \Prob\bigg(|X - a| \geq 100\sqrt{(y+1)\log^2 n}\bigg) \leq \exp\bigg(-C\frac{100^2(y+1)\log^2 n}{a \vee 100\sqrt{(y+1)\log^2n}}\bigg) \leq n^{-10}.
\end{align*}
If $a > 100(y+1)\log n\geq 100\log n$, we have
\begin{align*}
\frac{a^{y+1}e^{-a}}{(y+1)!} &= \Prob(X = y+1) \leq \Prob(X \leq a/2) \leq \Prob(|X-a| \geq a/2) \leq \exp\Big(-Ca\Big) \leq n^{-10}.
\end{align*}
Combining the two estimates yields that
\begin{align*}
\int_{a:|a-(y+1)|>100\sqrt{(y+1)\log^2 n}} \frac{a^{y+1}e^{-a}}{(y+1)!}G(\d a) \leq n^{-10},
\end{align*}
as desired. 
\end{proof}

We are now ready for the bounding the regret of Robbins' estimator.

\begin{proof}[Proof of Theorem \ref{thm:regret_robbins_EB}: Upper bound]
In the sequel, we omit the superscript in $\hat{\theta}^{\rob}$ as defined in \eqref{def:robbins}. We also assume for simplicity that the training data has sample size $n$ instead of $n-1$. Fix any distribution $G$ with $m_p(G)\leq 1$. For a fresh observation $Y$ from $f_G$, we have
\begin{align}\label{eq:robbins_reg_decomp}
\notag\E \big(\hat{\theta}(Y) - \theta_G(Y)\big)^2 &= \sum_{y=0}^\infty f_G(y)\E\Big(\hat{\theta}(y) - (y+1)\frac{f_G(y+1)}{f_G(y)}\Big)^2\\
\notag&= \sum_{y=0}^{y_0} f_G(y)(y+1)^2\E\Big(\frac{N(y+1)}{N(y) + 1} - \frac{f_G(y+1)}{f_G(y)}\Big)^2 \\
&\quad + \sum_{y=y_0+1}^\infty f_G(y)\Big(y - (y+1)\frac{f_G(y+1)}{f_G(y)}\Big)^2 \equiv (I) + (II). 
\end{align}
By Lemma \ref{lem:MLE_regret}, we have
\begin{align}\label{ineq:robbins_upper_II}
(II) = \E_G\big[\big(Y - \theta_G(Y)\big)^2\indc{Y > y_0}\big] \lesssim y_0^{-(p-1)} + \exp(-cy_0).
\end{align}

We will abbreviate $f_G(y)$ as $f(y)$, and use $\bin(n,p)$ to denote the binomial distribution with $n$ trials and success probability $p$. Note that conditioning on $N(y)$, $N(y+1)|N(y) \sim \bin\big(n - N(y), f(y+1)/(1 - f(y))\big)$, and marginally $N(y)\sim \bin(n, f(y))$. 
 Hence we have $(I) = (I_1) + (I_2)  + (I_3)$, where
\begin{align}
(I_1) & \equiv \sum_{y=0}^{y_0} \frac{f(y)f(y+1)\big(1-f(y) - f(y+1)\big)}{\big(1-f(y)\big)^2}(y+1)^2 \E \frac{n - N(y)}{(N(y) + 1)^2}, \label{ineq:robbins_upper_I1}\\
(I_2) &\equiv \sum_{y=0}^{y_0} f(y)(y+1)^2 \cdot \Big(\E{\frac{n-N(y)}{N(y) + 1}}  \frac{f(y+1)}{1-f(y)} - \frac{f(y+1)}{f(y)}\Big)^2,\label{ineq:robbins_upper_I2}\\
(I_3) &\equiv \sum_{y=0}^{y_0} f(y)(y+1)^2 \cdot \Big(\frac{f(y+1)}{1-f(y)}\Big)^2 \cdot \var\Big(\frac{n-N(y)}{N(y) + 1}\Big).\label{ineq:robbins_upper_I3}
\end{align}
By (\ref{binomial_moment2}) in Lemma \ref{lem:binomial_moments},
 with $X_1 \sim \bin(n+1, f(y))$,
\begin{align*}
(I_1) &\asymp \sum_{y=0}^{y_0} \frac{f(y+1)\big(1-f(y) - f(y+1)\big)}{(n+1)\cdot\big(1-f(y)\big)f(y)} (y+1)^2 \cdot \Prob(X_1 \geq 2)\\
&\leq \sum_{y=0}^{y_0} \frac{f(y+1)}{(n+1)f(y)}(y+1)^2 \cdot \Prob(X_1 \geq 2)\big(\indc{f(y) > n^{-1}} + \indc{f(y)\leq n^{-1}}\big).
\end{align*}
If $f(y) > n^{-1}$, by Lemma \ref{lem:mixture_pmf}, we have $f(y+1)/f(y) \lesssim \Big(\sqrt{\log^2 n/(y+1)} \vee 1\Big) + n^{-9} \lesssim \log n$, hence
\begin{align*}
\sum_{y=0}^{y_0} \frac{f(y+1)}{(n+1)f(y)}(y+1)^2 \cdot \Prob(X_1 \geq 2)\indc{f(y)> n^{-1}} \lesssim \frac{\log n}{n}\sum_{y=0}^{y_0} (y+1)^2 \asymp \frac{\log n}{n}y_0^3.
\end{align*}
If $f(y) \leq n^{-1}$, the same lemma yields $f(y+1) \lesssim \log n/n$, hence using $\Prob(X_1\geq 2) \leq 2^{-1}(n+1)f(y)$, we have
\begin{align*}
\sum_{y=0}^{y_0} \frac{f(y+1)}{(n+1)f(y)}(y+1)^2 \cdot \Prob(X_1 \geq 2)\indc{f(y) \leq n^{-1}} \lesssim \sum_{y=0}^{y_0} f(y+1) (y+1)^2 \lesssim \frac{\log n}{n}y_0^3.
\end{align*} 
This concludes $(I_1) \lesssim (\log n /n)y_0^3$. For $(I_2)$, (\ref{binomial_moment1}) in Lemma \ref{lem:binomial_moments} yields that, with $X_2 \sim \bin(n, f(y))$ and $\theta_G(\cdot)$ the Bayes estimator, 
\begin{align*}
(I_2) &= \sum_{y=0}^{y_0} f(y)(y+1)^2 \cdot \Big(\frac{f(y+1)}{f(y)}\Prob(X_2 \geq 1) - \frac{f(y+1)}{f(y)}\Big)^2\\
& = \sum_{y=0}^{y_0} \big(\theta_G(y)\big)^2 \big(1-f(y)\big)^{2n}\cdot f(y) \leq  \sum_{y=0}^{y_0} \big(\theta_G(y)\big)^2 e^{-2nf(y)}\cdot f(y)\\
&\lesssim \sum_{y=0}^{y_0} y^2 e^{-2nf(y)}\cdot f(y) + \sum_{y=0}^{y_0} \big(\E[|\theta-Y||Y=y]\big)^2 e^{-2nf(y)}\cdot f(y)\\
&\stackrel{\rm (a)}{\lesssim} \frac{y_0^3}{n} + y_0\cdot \sum_{y=0}^{y_0} \log^2\big(1/f(y)\big)e^{-2nf(y)}\cdot f(y) \lesssim \frac{y_0^3}{n} + \frac{y_0^2}{n}(\log n)^2,
\end{align*}
where (a) follows from Lemma \ref{lem:bayes_form_upper} and the fact that $\sup_{t>0} t e^{-2nt} \lesssim \frac{1}{n}$.
Finally, using (\ref{binomial_moment3}) in Lemma \ref{lem:binomial_moments} and $(1-p)/p^3 = (1-p)/p^2 + (1-p)^2/p^3$, we have
\begin{align*}
(I_3) &\lesssim \sum_{y=0}^{y_0} f(y)(y+1)^2 \cdot \Big(\frac{f(y+1)}{1-f(y)}\Big)^2 \cdot \frac{1}{n}\Big(\frac{1-f(y)}{f(y)^2} +\frac{\big(1-f(y)\big)^2}{f(y)^3}\Big)\indc{f(y) > n^{-1}}\\
&\qquad + \sum_{y=0}^{y_0} f(y)(y+1)^2 \cdot \Big(\frac{f(y+1)}{1-f(y)}\Big)^2\cdot n^2\indc{f(y) \leq n^{-1}} \equiv (I_{3,1}) + (I_{3,2}).
\end{align*}
For $f(y) > n^{-1}$, we have shown $f(y+1)/f(y)\lesssim \log n$ in the analysis of $(I_1)$, so $(I_{3,1}) \lesssim (\log n)^2\cdot y_0^3/n$. For $f(y)\leq n^{-1}$, we use $f(y) \vee f(y+1) \lesssim \log n/n$ to deduce $(I_{3,2}) \lesssim (\log n)^3\cdot y_0^3/n$, hence $(I_3) \lesssim (\log n)^3 \cdot y_0^3/n$.

In summary, we have shown that $(I) \lesssim (\log n)^3\cdot y_0^3/n + \exp(-cy_0)$. 
Combining this with the estimate of $(II)$ in (\ref{ineq:robbins_upper_II}) yields that
\begin{align*}
\inf_{y_0}\sup_{m_p(G)\leq 1} \pnorm{\hat{\theta}(\cdot;y_0) - \theta_G}{\ell_2(f_G)}^2 &\lesssim \inf_{y_0} \Big\{\frac{(\log n)^3}{n}y_0^3 + \exp(-cy_0) + y_0^{-(p-1)}\Big\}\\
&\asymp n^{-\frac{p-1}{p+2}}(\log n)^{\frac{3(p-1)}{p+2}}.
\end{align*}
This proves the desired upper bound \eqref{ineq:robbins_main_upper}.
\end{proof}

\subsubsection{Proof of lower bound}\label{subsec:proof_robbins_lower}

The following lemma constructs a special prior that will be used in the lower bound.
\begin{lemma}\label{lem:ht_prior}
Fix any $p > 0$. There exists some prior $G$ such that with some some universal $c,C > 0$ only depending on $p$,
\begin{align}\label{ineq:ht_prior_1}
c\cdot\frac{y^{-(p+1)}}{(\log y \vee 1)^2}\leq f_G(y) \leq C\cdot\frac{y^{-(p+1)}}{(\log y \vee 1)^2},
\end{align}
for all $y\geq 0$. Consequently, there exists some $c' = c'(p) > 0$ such that for all $y\geq 0$,
\begin{align}\label{ineq:ht_prior_2}
\frac{f_G(y+1)}{f_G(y)} \wedge \frac{1-f_G(y) - f_G(y+1)}{1-f_G(y)} \geq c'.
\end{align}
\end{lemma}
\begin{proof}[Proof of Lemma \ref{lem:ht_prior}]
Let $g(a) \equiv c_0a^{-(p+1)}(\log a)^{-2}$ on $[e, \infty)$ with $c_0 = c_0(p) > 0$ chosen such that $\int_e^\infty g(a)\d a = 1$. Let $\bar{G}$ be a distribution with density $g$, and
\begin{align*}
G \equiv \epsilon \delta_0 + (1-\epsilon)\bar{G},
\end{align*}
for some $\epsilon \in [0,1]$. Then
\begin{align*}
m_p(G) = c_0(1-\epsilon)\int_e^\infty a^p \cdot a^{-(p+1)}(\log a)^{-2} \d a = c_0(1-\epsilon).
\end{align*}
Note that
\begin{align*}
c_0 = \frac{1}{\int_e^\infty a^{-(p+1)}(\log a)^{-2}\d a} = \frac{\int_e^\infty a^p \cdot a^{-(p+1)}(\log a)^{-2} \d a}{\int_e^\infty a^{-(p+1)}(\log a)^{-2}\d a} > 1,
\end{align*}
hence we may choose $\epsilon = \epsilon(p) \in (0,1)$ such that $m_p(G) = 1$. 
Next we consider this $G$ and it suffices to prove (\ref{ineq:ht_prior_1}) for all sufficiently large $y$. We have
\begin{align}\label{ineq:ht_density}
\notag f_G(y) &= (1-\epsilon)c_0\int_e^\infty \frac{a^ye^{-a}}{y!} \cdot a^{-(p+1)}(\log a)^{-2} \d a\\
\notag&\lesssim \int_e^{y/2}  \frac{a^ye^{-a}}{y!} \cdot a^{-(p+1)}(\log a)^{-2} \d a + \int_{y/2}^\infty  \frac{a^ye^{-a}}{y!} \cdot a^{-(p+1)}(\log a)^{-2} \d a\\
\notag&\stackrel{\rm (a)}{\lesssim} \exp(-cy) + (\log y)^{-2}\Big(\int_0^\infty \frac{a^ye^{-a}}{y!} \cdot a^{-(p+1)} \d a\Big)\\
&=  \exp(-cy) + \frac{\Gamma(y-p)}{y!}(\log y)^{-2} \stackrel{\rm (b)}{\lesssim} y^{-(p+1)}(\log y)^{-2},
\end{align}
where in (a) we use $\sup_{a\in(0,y/2)} \poi(y;a) \lesssim \exp(-cy)$ by Lemma \ref{lem:poi_basic}\ref{poi_basic1}; 
(b) follows from Stirling approximation of the Gamma function. The matching lower bound is analogous, so we have proved (\ref{ineq:ht_prior_1}). The inequality (\ref{ineq:ht_prior_2}) for $\frac{f_G(y+1)}{f_G(y)} $ follows from \eqref{ineq:ht_prior_1} directly. Finally, 
\begin{align*}
\frac{1-f_G(y) - f_G(y+1)}{1-f_G(y)} &= 1 - \frac{f_G(y+1)}{1-f_G(y)} \geq 1 - \frac{f_G(y+1)}{f_G(y+1) + f_G(y+2)}\\
&= \frac{f_G(y+2)}{f_G(y+1) + f_G(y+2)} \gtrsim_p 1.
\end{align*} 
\end{proof}

\begin{proof}[Proof of Theorem \ref{thm:regret_robbins_EB}: Lower bound]
Fix any $y_0 \in[1,\infty]$. First take $G = (1-\epsilon)\delta_0 + \epsilon\delta_a$ with $\epsilon = a^{-p}$ and $a = 2y_0$, then $m_p(G) = \epsilon a^p = 1$ and $\theta_G(y) = a$ for any $y > 0$. Recall the regret decomposition $(I_1)+(I_2)+(I_3)   +(II)$ in  (\ref{eq:robbins_reg_decomp})--(\ref{ineq:robbins_upper_I3}). Then
\begin{align*}
(II) &= \E \big(Y - \theta_G(Y)\big)^2\indc{Y\geq y_0} = \E (Y - a)^2\indc{Y\geq y_0} \geq \epsilon\cdot \E\big[(Y - a)^2\indc{Y\geq y_0}|\theta = a\big]\\
&\gtrsim \epsilon\cdot \E\big[(Y - a)^2|\theta = a\big] = a\epsilon \asymp y_0^{-(p-1)}.
\end{align*}
Next take the prior $G$ in Lemma \ref{lem:ht_prior}. Then $f_G(y) \asymp_p y^{-(p+1)}(\log y)^{-2}$. Thus by setting 
\begin{align*}
y_\ast \equiv  c_p\big(n/(\log n)^2\big)^{1/(p+1)}
\end{align*}
with some appropriate $c_p$, we have $f_{G}(y) \geq 5/n$ for all $y\leq y_\ast$ by the construction in Lemma \ref{lem:ht_prior}.
Hence in the regret decomposition (\ref{eq:robbins_reg_decomp}) with $X\sim \bin(n, f_G(y))$,
\begin{align*}
(I_1) &\asymp \sum_{y=0}^{y_0} \frac{f_G(y+1)\big(1-f_G(y) - f_G(y+1)\big)}{(n+1)\cdot\big(1-f_G(y)\big)f_G(y)} (y+1)^2 \cdot \Prob(X \geq 2)\\
&\geq \sum_{y=0}^{y_0\wedge y_\ast} \frac{f_G(y+1)\big(1-f_G(y) - f_G(y+1)\big)}{(n+1)\cdot\big(1-f_G(y)\big)f_G(y)} (y+1)^2 \cdot \Prob(X \geq 2)\\
&\stackrel{\rm (a)}{\gtrsim} \sum_{y=0}^{y_0\wedge y_\ast} \frac{(y+1)^2}{n} \cdot \frac{f_G(y+1)}{f_G(y)} \cdot \frac{1-f_G(y) - f_G(y+1)}{1-f_G(y)}\\
&\stackrel{\rm (b)}{\gtrsim} \sum_{y=0}^{y_0\wedge y_\ast}  \frac{(y+1)^2}{n} \gtrsim \frac{(y_0\wedge y_\ast)^3}{n},
\end{align*}
where (a) follows since for all $y\leq y_\ast$ we have $f_G(y)\geq 5/n$ so $\Prob(X\geq 2) \gtrsim 1$, and (b) follows from (\ref{ineq:ht_prior_2}) in Lemma \ref{lem:ht_prior}. Averaging over the above two priors yields that
\begin{align*}
\inf_{y_0>1}\sup_{m_p(G)\leq 1} \pnorm{\hat{\theta}(\cdot;y_0) - \theta_G}{\ell_2(f_G)}^2 \gtrsim \inf_{y_0>1} \Big\{y_0^{-(p-1)} + \frac{(y_0\wedge y_\ast)^3}{n}\Big\} \asymp n^{-\frac{p-1}{p+2}}.
\end{align*}
The proof is complete.
\end{proof}

\subsection{Proof of Theorem \ref{thm:hybrid}}\label{subsec:proof_hybrid}

\begin{proof}
We start with the definition of this $f$-modeling estimator. With i.i.d. observations $Y_1,\ldots,Y_n$ from some $f_G$, let
\begin{align*}
\hat{f}^{\emp}(y) \equiv \frac{N_n(y)}{n} = \frac{\sum_{i=1}^n \bm{1}\{Y_i = y\}}{n}, \quad y\in \mathbb{Z}_+,
\end{align*}
be the empirical estimator. For some $y_0\in \mathbb{Z}_+$ to be specified, let
\begin{align*}
\bar{f}(y) \equiv
\begin{cases}
f_{\hat{G}}(y) & y \leq y_0,\\
\hat{f}^{\emp}(y) & y > y_0,
\end{cases}
\end{align*}
where $\hat{G}$ is the NPMLE given by (\ref{def:npmle}). Define a hybrid density estimator
\begin{align}
\hat{f}^{\hb}(y) \equiv \frac{\bar{f}(y)}{a}, \quad \text{where} \quad a \equiv \sum_{y=0}^\infty \bar{f}(y). 
\end{align}
Since $\hat{f}^{\emp}$ and $f_{\hat{G}}$ are both valid probability mass functions, $a$ is well-defined. Correspondingly, the induced EB estimator for $\theta_n$ is
\begin{align*}
\hat{\theta}_n^{\hb}(Y^n) = (Y_n + 1)\frac{\hat{f}^{\hb}(Y_n+1)}{\hat{f}^{\hb}(Y_n)} = 
\begin{cases}
(Y_n+1)\frac{f_{\hat{G}}(Y_n+1)}{f_{\hat{G}}(Y_n)} & Y_n \leq y_0,\\
(Y_n+1)\frac{N_{n-1}(Y_n+1)}{N_{n-1}(Y_n) + 1} & Y_n > y_0,
\end{cases}
\end{align*}
where $N_{n-1}(y) = \sum_{i=1}^{n-1} \bm{1}\{Y_i = y\}$ is the number of occurrences of $y$ among $Y^{n-1}$. In words, $\hat{\theta}_n^{\hb}$ is an interpolation between the NPMLE EB and Robbins estimators, and clearly belongs to the $f$-modeling category because $\hat{f}^{\hb}$ is not a valid Poisson mixture. We will now prove density estimation upper bound and regret lower bound for this estimator.

\paragraph{Density estimation upper bound} We will show that if $y_0 > cn^{2/(2p+1)}$ for some universal $c > 0$, then there exists some $C = C(p)>0$  such that 
\begin{align*}
\sup_{G \in \mathcal{G}_p(1)} \E_G H^2(\hat{f}^{\hb}, f_G) \leq Cn^{-\frac{2p}{2p+1}}(\log n)^6.
\end{align*}
We first prove this result for the un-normalized $\bar{f}(y)$. We have
\begin{align*}
\E \pnorm{\sqrt{\bar{f}} - \sqrt{f_G}}{\ell_2}^2 = \E \sum_{y=0}^\infty \big(\sqrt{\bar{f}(y)} - \sqrt{f_G(y)}\big)^2 \leq \E H^2(f_{\hat{G}}, f_G) + \E \sum_{y > y_0} \big(\sqrt{\hat{f}^{\emp}(y)} - \sqrt{f_G(y)}\big)^2.
\end{align*}
The first term is bounded by $C n^{-\frac{2p}{2p+1}}(\log n)^6$ for some $C = C(p) > 0$ by Theorem \ref{thm:density_main}. Using $\E \hat{f}^{\emp}(y) = f_G(y)$ and Markov inequality,
\begin{align*}
\E \sum_{y > y_0} \big(\sqrt{\hat{f}^{\emp}(y)} - \sqrt{f_G(y)}\big)^2 \lesssim \sum_{y > y_0} f_G(y) \leq y_0^{-p},
\end{align*}
yielding the claim for the un-normalized $\bar{f}(y)$. With $a = \pnorm{\sqrt{\bar{f}}}{\ell_2}^2$, this implies
\begin{align*}
\E (\sqrt{a} - 1)^2 = \E \big(\pnorm{\sqrt{\bar{f}}}{\ell_2} - \pnorm{\sqrt{f_G}}{\ell_2}\big)^2 \leq \E \pnorm{\sqrt{\bar{f}} - \sqrt{f_G}}{\ell_2}^2.
\end{align*}
Then
\begin{align*}
\E H^2(\hat{f}^{\hb}, f_G) = \E \pnorm{\sqrt{\bar{f}}/\sqrt{a} - \sqrt{f_G}}{\ell_2}^2 &\leq 2\E \pnorm{\sqrt{\bar{f}} - \sqrt{f_G}}{\ell_2}^2 + 2\E\pnorm{\sqrt{\bar{f}}}{\ell_2}^2 \cdot (1/\sqrt{a} - 1)^2\\
&= 2\E \pnorm{\sqrt{\bar{f}} - \sqrt{f_G}}{\ell_2}^2 + \E(\sqrt{a} - 1)^2\\
&\leq 3\E \pnorm{\sqrt{\bar{f}} - \sqrt{f_G}}{\ell_2}^2,
\end{align*}  
as desired. 

\paragraph{Regret lower bound} Let $y_\ast = \floor{K_p(n/\log^2 n)^{1/(2p+1)}}$ for some large $K_p > 0$. We will show that there exists some $c = c(p) > 0$ such that
\begin{align*}
\reg_n(\hat{\theta}^{\hb}_n; \mathcal{G}_p(1)) \geq \frac{cn(y_0 \vee y_\ast)^{-(2p-1)}}{(\log n)^4}.
\end{align*}
It is clear that
\begin{align*}
\E_{Y^n} \big[\hat{\theta}^{\hb}_n(Y_n) - \theta_G(Y_n)\big]^2 \geq \E_{Y^{n-1}} \sum_{y > y_0} f_G(y) (y+1)^2 \E\Big(\frac{N_{n-1}(y+1)}{N_{n-1}(y) + 1} - \frac{f_G(y+1)}{f_G(y)}\Big)^2.
\end{align*}
Recall that as in (\ref{eq:robbins_reg_decomp}), the above expectation can be decomposed into three non-negative terms $(I_1)$-$(I_3)$, where, with $X \sim \text{Bin}(n, f_G(y))$,
\begin{align*}
(I_1) &= \sum_{y > y_0} \frac{f_G(y)f_G(y+1)\big(1-f_G(y) - f_G(y+1)\big)}{\big(1-f_G(y)\big)^2}(y+1)^2 \cdot \E \frac{(n-1) - N_{n-1}(y)}{\big(N_{n-1}(y) + 1\big)^2}\\
&\asymp \sum_{y > y_0} \frac{f_G(y+1)\big(1-f_G(y) - f_G(y+1)\big)}{nf_G(y)(1-f_G(y))}(y+1)^2(y+1)^2 \Prob(X \geq 2).
\end{align*}
Now take the prior $G$ constructed in Lemma \ref{lem:ht_prior} where $f_G(y) \asymp_p y^{-(p+1)}(\log y + 1)^{-2}$. Note that the definition of $y_\ast$ implies that $f_G(y) \leq 0.001 n^{-1}$ whenever $y > y_\ast$.
Then the above $(I_2)$ can be further lower bounded by 
\begin{align*}
&\sum_{y > y_0 \vee y_\ast} \frac{f_G(y+1)\big(1-f_G(y) - f_G(y+1)\big)}{nf_G(y)(1-f_G(y))}(y+1)^2 \Prob(X = 2)\\
&\gtrsim_p n\cdot \sum_{y > y_0\vee y_\ast}  (y+1)^2f_G(y)^2 \big(1-f_G(y)\big)^{n-2} \gtrsim \frac{n}{(\log n)^4}\cdot \sum_{y > y_0\vee y_\ast} y^{-2p} \asymp_p \frac{n(y_0 \vee y_\ast)^{-(2p-1)}}{(\log n)^4}.
\end{align*}
To conclude the proof, it remains to choose $y_0 = cn^{2/(2p+1)+\delta'}$ for some small $\delta'$ (depending on $\delta$).
\end{proof}

\appendix

\section{Auxiliary results}
\label{app:aux}

\begin{lemma}\label{lem:poi_basic}
Let $X \sim \poi(\theta)$ for some $\theta > 0$, and $Y \sim f_G$ with $f_G \in \mathcal{H}_p(M_p)$ in (\ref{def:mixture_class_new}) for some $p > 0$ and $M_p > 0$. 
\begin{enumerate}[label=(\alph*)]
\item \label{poi_basic1} (Poisson tail) For any $t > 0$,
\begin{align*}
\Prob(X - \theta > t) \vee \Prob(X - \theta < -t) \leq \exp\Big(-\frac{t^2}{2(\theta + t)}\Big).
\end{align*}
\item \label{poi_basic2} (Poisson centered moments) There exists some universal $C>0$ such that for any $p\geq 1$, $\E|X - \theta|^p \leq (Cp)^{p}(\theta \vee 1)^{p/2}$. 
\item \label{poi_basic3} (Poisson mixture tail) There exists some universal $c > 0$ such that, for any $t > 0$,
\begin{align*}
\Prob(Y \geq t) \leq \exp(-ct) + (t/2)^{-p}M_p.
\end{align*}
\end{enumerate}
\end{lemma}
\begin{proof}[Proof of Lemma \ref{lem:poi_basic}]
For Part \ref{poi_basic1}, we only prove the right tail, since the left tail follows from a similar argument and actually admits the stronger bound $\exp(-t^2/(2\theta))$. Since $\E e^{sX} = \exp(\theta e^s - \theta)$, the Chernoff bound yields that, for any $t > 0$, 
\begin{align*}
\Prob(X - \theta > t) \leq \exp\Big(-\sup_{s\geq 0}(st - \theta e^s + \theta + s\theta)\Big) = \exp\big(-\theta h(t/\theta)\big) \stackrel{\rm (a)}{\leq} \exp\Big(-\frac{t^2}{2(\theta + t)}\Big),
\end{align*}
where $h(u) \equiv (1+u)\log(1+u) - u$ for any $u > -1$, and (a) follows from the fact that $h(u) \geq u^2/(2(1+u))$ for any $u \geq 0$.
Part \ref{poi_basic2} follows directly by integrating the tail estimate in Claim (1); see, e.g. \cite[Theorem 2.3]{boucheron2013concentration}. For Part \ref{poi_basic3}, let $\theta\sim G$ for some 
$G\in \calG_p(M_p)$ and $Y|\theta \sim\poi(\theta)$. For any $t \geq 0$,
\begin{align*}
\Prob(Y \geq t) = \E \Prob(\poi(\theta) \geq  t|\theta) &= \E \Prob(\poi(\theta) \geq  t|\theta)\indc{\theta \leq t/2} + \E \Prob(\poi(\theta) \geq  t|\theta)\indc{\theta > t/2}\\
&\leq \E\Prob(\poi(\theta) - \theta \geq  t/2|\theta)\indc{\theta \leq t/2} + \Prob_{\theta\sim G}(\theta \geq t/2)\\
&\stackrel{\rm (a)}{\leq} \E\exp\bigg(-\frac{(t/2)^2}{2(\theta + t/2)}\bigg)\indc{\theta \leq t/2} + (t/2)^{-p}M_p\\
&\leq \exp(-ct) + (t/2)^{-p}M_p,
\end{align*}
as desired. Here in (a) we use the Poisson tail in Part \ref{poi_basic1}.
\end{proof}

\begin{lemma}\label{lem:pmf_lip}
For any $j\in\integers_+$, let $\poi(j;\lambda) \equiv \lambda^je^{-\lambda}/j!$ be the Poisson density. Then 
\begin{align*}
\sup_{\lambda > 0}\sup_{j\geq 0}\left|\frac{\d\poi(j;\lambda)}{\d \lambda}\right| \leq 1.
\end{align*}
\end{lemma}
\begin{proof}
The claim clearly holds for $j = 0$, and for $j \geq 1$,  
\begin{align*}
\frac{\d \poi(j;\lambda)}{\d\lambda} = \frac{j\lambda^{j-1}e^{-\lambda} - \lambda^je^{-\lambda}}{j!} = \poi(j-1;\lambda) - \poi(j;\lambda) \in [-1,1].
\end{align*}
\end{proof}

\begin{lemma}\label{lem:poi_divergence}
Let $\poi(\lambda)$ and $\poi(\lambda')$ be Poisson distributions 
with means $\lambda$ and $\lambda'$. Then
\begin{align*}
\chi^2\big(\poi(\lambda)||\poi(\lambda')\big) &= \exp\big((\lambda - \lambda')^2/\lambda'\big) - 1,\\
H^2\big(\poi(\lambda), \poi(\lambda')\big) &= 1 - \exp\big(-(\sqrt{\lambda} - \sqrt{\lambda'})^2/2\big).
\end{align*}
Furthermore,
\begin{align*}
\chi^2\big(\mathcal{N}(\theta,1) || \mathcal{N}(\theta', 1)\big) = \exp\big((\theta - \theta')^2\big) - 1.
\end{align*}
\end{lemma}
\begin{proof}
	Straightforward computation.
\end{proof}

\section{Proof of (\ref{eq:bayes_risk_infty}) for $p < 1$}
\label{app:mmseinf}


We take $G$ to be the distribution with density $g(a;\tau) = \tau a^{-2}\indc{a\geq \tau}$. For any given $p\in(0,1)$ and $M_p > 0$, we have $G\in \mathcal{G}_p(M_p)$ by choosing $\tau = \big((1-p)M_p\big)^{1/p}$. Since the following calculation holds for any $\tau > 0$, we only consider $\tau = 1$ for simplicity. We will abbreviate $f_G$ as $f$. 
To show $\mmse(G)=\infty$, it suffices to analyze the conditional expectation. 
By definition, for $\theta\sim G$ and $Y|\theta\sim \poi(\theta)$, we have
\begin{align}\label{eq:bayes_cond_var}
\notag\E_G \big(\theta_G(Y) - \theta\big)^2 &= \E \big[\E(\theta^2|Y) - (\E\theta|Y)^2\big]\\
\notag&= \E (Y+1)(Y+2)\frac{f(Y+2)}{f(Y)} - \Big((Y+1)\frac{f(Y+1)}{f(Y)}\Big)^2\\
&= \sum_{y=0}^\infty \frac{y+1}{f(y)}\cdot \Big((y+2)f(y+2)f(y) - (y+1)f^2(y+1)\Big).
\end{align}
Now using the definition of $G$, the inner term can be computed as
\begin{align*}
&(y+2)f(y+2)f(y) - (y+1)f^2(y+1)\\
&= (y+2)\int_1^\infty \frac{a^ye^{-a}}{(y+2)!}\d a \cdot \int_1^\infty \frac{a^{y-2}e^{-a}}{y!}\d a - (y+1) \Big(\int_1^\infty\frac{a^{y-1}e^{-a}}{(y+1)!}\d a \Big)^2\\
&= \frac{1}{y!(y+1)!}\int_1^\infty\int_1^\infty e^{-(a+b)}\Big(\frac{1}{2}a^yb^{y-2} + \frac{1}{2}b^ya^{y-2} - a^{y-1}b^{y-1}\Big)\d a\d b\\
&= \frac{1}{y!(y+1)!}\int_1^\infty\int_1^\infty e^{-(a+b)}\frac{1}{2}a^{y-2}b^{y-2}(a-b)^2\d a\d b =  \frac{c^2(y)}{y!(y+1)!}\var(U),
\end{align*}
where $U\equiv U(y)$ is a random variable with density $f_U(a) = a^{y-2}e^{-a}/c(y)$ on $[1,\infty)$, with $c(y)\equiv \int_1^\infty a^{y-2}e^{-a}\d a$. We claim that for sufficiently large $y$, $\var(U) \gtrsim y$. To see this, let $\bar{c}(y)\equiv \int_0^1 a^{y-2}e^{-a}\d a \leq 1$, and $V$ be a random variable with density $f_V(a) = a^{y-2}e^{-a}/\bar{c}(y)$ on $[0,1]$. Let $W$ be a Bernoulli variable independent of $(U,V)$ with success probability $q \equiv c(y)/(c(y) + \bar{c}(y))$. Note that $c(y) = (y-2)! - \bar c(y)$ and $\bar{c}(y)\leq 1$, so that $1 - q = \bar{c}(y)/(c(y) + \bar{c}(y)) \leq 1/(y-2)!$. Let $Z \equiv WU + (1-W)V$, so that $Z$ has density
\begin{align*}
f_Z(a) = qf_U(a) + (1-q)f_V(a) = \frac{a^{y-2}e^{-a}}{\int_0^\infty b^{y-2}e^{-b}\d b}, \quad a\geq 0.
\end{align*}
This implies $Z\sim \Gamma(y-1)$ with $\E Z = \var(Z) = y-1$. Moreover, using $\cov(WU, (1-W)V) \leq 0$, we have
\begin{align*}
\var(Z) &= \var(WU + (1-W)V) \leq \var(WU) + \var((1-W)V) \leq q\E U^2 - q^2(\E U)^2 + 1\\
&= q\var(U) + q(1-q)(\E U)^2 + 1 \leq q\var(U) + o(y),
\end{align*}
where the last step follows
from $1 - q  \leq 1/(y-2)!$
and $y-1 = \E Z = q\E U + (1-q)\E V$ so that $\E U = O(y)$.
 This yields $\var(U) \gtrsim \var(Z) = y$, proving the claim. Plugging the above estimate into (\ref{eq:bayes_cond_var}) yields that, for some large $K > 0$,
\begin{align*}
\E_G \big(\theta_G(Y) - \theta\big)^2 &\geq \sum_{y=K}^\infty \frac{y+1}{f(y)}\cdot \Big((y+2)f(y+2)f(y) - (y+1)f^2(y+1)\Big)\\
&\gtrsim \sum_{y=K}^\infty \frac{y}{y^{-2}}y^{-4} =  \sum_{y=K}^\infty y^{-1} = \infty,
\end{align*}
where we use the readily obtainable fact that $f(y) \asymp y^{-2}$
 (see a similar computation in (\ref{ineq:ht_density})). The proof is complete.

\section{A complex-analytic proof of Proposition \ref{prop:Ak_upper} for even $k$}\label{sec:second_proof}
In this section, we provide a second proof of Proposition \ref{prop:Ak_upper} which is self-contained and based on generating functions. We start with two definitions. For a sequence $f: \integers_+ \rightarrow \R$, its generating function is defined by
\begin{align*}
\phi_{f}(z) \equiv \sum_{y = 0}^\infty f(y)z^y, \quad z\in \mathbb{C}.
\end{align*} 
When $\|f\|_{\ell_1} < \infty$, $\phi_f(z)$ is a holomorphic function on the unit disk $D\equiv \{z\in \mathbb{C}: |z| \leq 1\}$. Next, for a signed measure $G$ on $\reals_+$, its generating function (Laplace transform) is defined by
\begin{align*}
\phi_G(z) \equiv \int_{\reals_+} e^{z\theta} G(\d\theta), \quad z \in \mathbb{C}.
\end{align*}
When $\pnorm{G}{\textrm{TV}} \equiv \int_{\reals_+}|G(\d\theta)| < \infty$, $\phi_G(z)$ is a holomorphic function on the half plane $\{z\in \mathbb{C}: \mathfrak{R}(z) \leq 0\}$. The following lemma will be useful. 

\begin{lemma}\label{lem:pgf_basic}
Given any sequence $\{f(y)\}_{y\geq 0}\in \ell_1$, the following hold for $z\in D = \{z\in \mathbb{C}: |z| \leq 1\}$.
\begin{enumerate}
\item (Finite difference) Under the convention $f(y)\equiv 0$ for $y < 0$, for any $k\in \integers_+$, 
\begin{equation}
	\phi_{\nabla^k f} (z) = (1-z)^k \phi_{f}(z),
	\label{eq:GF-Deltak}
	\end{equation}
where $\nabla^k$ denotes the $k$th-order backward difference defined in \prettyref{eq:backwarddiff}.

\item (Derivatives) For any $k\in \integers_+$, define the $f_{[k]}(y) \equiv (y+k)_k f(y+k)$ with $(\cdot)_k$ the falling factorial. Then
	\begin{equation}
	\phi_{f_{[k]}} (z) = \phi_{f}^{(k)}(z).
	\label{eq:GF-deriv}
	\end{equation}

\item (Parseval's identity)
\begin{align*}
\sum_{y=0}^\infty f(y)^2 = \frac{1}{2\pi}\int_0^{2\pi} \big|\phi_f(e^{i\omega})\big|^2 \d\omega.
\end{align*}
\end{enumerate}
\end{lemma}
\begin{proof}
\begin{enumerate}
	\item We prove by induction. The claim clearly holds for $k = 0$. Suppose the claim holds up to $k$, then
\begin{align*}
\phi_{\nabla^{k+1}f} (z) &= \sum_{y=0}^\infty \big(\nabla^{k}f(y) - \nabla^{k}f(y-1)\big)z^y = \phi_{\nabla^{k}f}(z) - z\cdot \sum_{y=-1}^\infty\nabla^{k}f(y)z^y\\
&= \phi_{\nabla^{k}f}(z) - z\cdot \sum_{y=0}^\infty\nabla^{k}f(y)z^y = (1-z)\phi_{\nabla^{k}f}(z) = (1-z)^{k+1} \phi_f(z),
\end{align*}
using the fact that $\nabla^{k}f(-1) = 0$ for all $k\geq 0$. 

\item This holds because
\begin{align*}
\frac{\d^k}{\d z^k} \sum_{y=0}^{\infty} f(y)z^{y} = \sum_{y=k}^\infty (y)_k f(y) z^{y-k}= \sum_{y=0}^\infty (y+k)_{k}f(y+k)z^y = \phi_{f_{[k]}}(z).
\end{align*}

\item The right side equals
\begin{align*}
\sum_{y_1,y_2=0}^{\infty} f(y_1)f(y_2)\frac{1}{2\pi}\int_0^{2\pi} e^{i\omega(y_1-y_2)}\d\omega = \sum_{y_1,y_2=0}^\infty f(y_1)f(y_2)\indc{y_1=y_2} = \sum_{y=0}^\infty f(y)^2,
\end{align*}
where we use the fact that $f\in \ell_1(\integers_+) \subset \ell_2(\integers_+)$ to apply Fubini's theorem.
\end{enumerate}
\end{proof}

When $f$ and $G$ are probability measures, $\phi_f$ and $\phi_G$ correspond to their probability generating function and the moment generating function (Laplace transform).
In particular, we have 
\[
\phi_G^{(k)}(0) = \Expect_{\theta\sim G}[\theta^k]
\]
Thus, if the $k$th moment of $G$ does not exist, we anticipate $|\phi_G^{(k)}(z)|$ to blow up as $z$ approaches the imaginary axis from the left. The following estimate will be useful:
For any $\Re(z) <0$, 
\begin{equation}
|\phi_G^{(k)}(z)| \leq \pth{\frac{k}{e|\Re(z)|}}^k.
\label{eq:Gz-estimate}
\end{equation}
Indeed, let $\Re(z) = - \epsilon <0$. Then
\[
|\phi_G^{(k)}(z)| \leq \int_{\reals_+} |e^{z\theta}| \theta^k G(d\theta) = 
\int_{\reals_+} e^{-\epsilon \theta} \theta^k G(d\theta)
\leq \sup_{\theta\geq 0} \theta^k e^{-\epsilon\theta} = \pth{\frac{k}{e\epsilon}}^k.
\]

For any prior $G$ on $\reals_+$, recall that $f_G$ denotes the corresponding Poisson mixture. The following identity \cite[Eq.~(114)]{PW18-dual2} relates their generating functions:
\begin{equation}
\phi_{f_G}(z) = \phi_{G}(z-1).
\label{eq:PGF}
\end{equation}
Indeed, 
\[
\phi_{f_G}(z) = \Expect_{Y\sim f_G}[z^Y] = \Expect_{\theta\sim G}[\Expect_{Y\sim \Poi(\theta)}[z^Y|\theta] ] = \Expect_{\theta\sim G}[e^{(z-1)\theta} ] = \phi_G(z-1).
\]

We are now ready to give a second proof of Proposition \ref{prop:Ak_upper}. 
We aim to prove the following: For any distribution $G$ and any even $k$,
\begin{equation}
\sum_{y\geq 0} (y+1)^k (\Delta^k f_G(y))^2 \leq 2^{3k} k!
\label{eq:Ak}
\end{equation}
where $\Delta f_G(y) = f_G(y+1)-f_G(y)$ is the forward difference defined in \prettyref{eq:forwarddiff}.
To deduce Proposition \ref{prop:Ak_upper} from here, recall the definition of $A_k$ in (\ref{def:Ak}) and $w(y) \leq 1/(2\rho)$ from \prettyref{eq:wstar}. Then
\begin{align}\label{ineq:Ak_second_proof}
\notag A_k^2 = \sum_{y=0}^\infty (y+1)^{2\ell} \Big(\Delta^k f_{G_1}(y) - \Delta^k f_{G_2}(y)\Big)^2 w(y) \leq \frac{1}{\rho}  2^{3k} k!
\end{align}
which yields Proposition \ref{prop:Ak_upper} in view of the assumption that $\rho \geq n^{-K_\rho}$ and $k\geq  \kappa\log n$.

To prove \prettyref{eq:Ak}, let $k=2\ell$. Using $\Delta^kf(y) = \nabla^{k} f(y+k)$, we have
\begin{align*}
\sum_{y\geq 0} (y+1)^k (\Delta^k f_G(y))^2 
= & ~ \sum_{y\geq 0} \pth{(y+1)^{\ell} \cdot \nabla^{2\ell} f_G(y+2\ell)}^2 \\
\leq & ~ \sum_{y\geq 0} \pth{(y+2\ell)_{\ell} \cdot \nabla^{2\ell} f_G(y+2\ell)}^2 \\
\leq & ~ \sum_{y\geq 0} \Big(\underbrace{(y+\ell)_{\ell}  \cdot \nabla^{2\ell} f_G(y+\ell)}_{=(\nabla^{2\ell} f_G)_{[\ell]}(y) \equiv g(y)}\Big)^2 = \|g\|_2^2.
\end{align*}

Next we show $\|g\|_2^2 \leq 2^{3k} k!$.
Fix $a \in (0,1)$ and let $\tilde g(y) \triangleq g(y) a^y$. Note 
$0 \leq f_G \leq 1$ since $f_G$ is a pmf. 
Applying the binomial expansion of backward difference in \prettyref{eq:backwarddiff-expand}, we have
\[
\nabla^{2\ell} f_G(y) =  
\sum_{i=0}^{2\ell} (-1)^i \binom{2\ell}{i} f_G(y-i).
\]
So
$|g(y)| \leq 2^{2\ell} (y+\ell)_{\ell} \leq 2^{2\ell} (y+\ell)^{\ell}$ and hence $\sum_{y\geq 0}|\tilde g(y)|<\infty$ for any $0<a<1$.

Note that
\begin{equation}
\phi_{\tilde g}(e^{i\omega})  = \phi_{g}(a e^{i\omega}) =
 \left.\frac{d^\ell}{dz^\ell} (\phi_{G}(z -1) (z -1)^{2\ell})\right|_{z = ae^{i\omega}}.
\label{eq:gtilde-fourier}
\end{equation}
where the second identity applies 
\prettyref{eq:GF-Deltak}, \prettyref{eq:GF-deriv}, and \prettyref{eq:PGF}.
By chain rule, we have 
\begin{equation}
\frac{d^\ell}{dz^\ell} (\phi_{G}(z -1) (z -1)^{2\ell})
= \sum_{m=0}^\ell
\binom{\ell}{m} 2\ell(2\ell-1)\cdots(\ell+m+1)  \cdot \phi_{G}^{(m)}(z -1) 
(z-1)^{\ell+m}.
\label{eq:chainrule}
\end{equation}
Crucially, for $z=a e^{j \omega}$, 
\[
0< -\Re(z-1) = 1 - a \cos\omega, \quad 
|z-1| = \sqrt{a^2+1 - 2a\cos\omega}.
\]
Applying the estimate \prettyref{eq:Gz-estimate}, we have for every $0 \leq m\leq \ell$,
\[
|\phi_{G}^{(m)}(ae^{i\omega} -1) (ae^{i\omega}-1)^{\ell+m}| \leq \pth{\frac{m}{e(1 - a \cos\omega)}}^m \cdot (a^2+1 - 2a\cos\omega)^{\frac{\ell+m}{2}}.
\]
If $\cos \omega\geq 0$, $a^2+1 - 2a\cos\omega \leq 2-2\cos\omega$ and $1-a\cos\omega \geq 1-1\cos\omega$;
if $\cos \omega\leq 0$, $a^2+1 - 2a\cos\omega \leq (a+1)^2 \leq 4$ and $1-a\cos\omega \geq 1$.
In all, we have for all $a\in (0,1), \omega \in [0,\pi]$ and $0\leq m \leq \ell$, 
\[
|\phi_{G}^{(m)}(ae^{i\omega} -1) (ae^{i\omega}-1)^{\ell+m}| \leq \pth{\frac{m}{e}}^m 2^{\ell+m}
\]
and hence, in view of \prettyref{eq:gtilde-fourier} and \prettyref{eq:chainrule}, $|\phi_{\tilde g}(e^{i\omega})|$ is bounded uniformly in $a$ and $\omega$.
Thus
\begin{equation}
\|g\|_2^2
=
\lim_{a\uparrow1} 
\|\tilde g\|_2^2
= \lim_{a\uparrow1} \frac{1}{2\pi}\int_0^{2\pi}  \left|\phi_{\tilde g}(e^{i\omega})\right|^2
=  \frac{1}{2\pi}\int_0^{2\pi} \lim_{a\uparrow1} \left|\phi_{\tilde g}(e^{i\omega})\right|^2,
\label{eq:gg}
\end{equation}
where the three equalities follow from the monotone convergence theorem, Parseval's identity, and the dominated convergence theorem, respectively.
To bound the limit inside the integral, note that
\[
\phi_{G}^{(m)}(ae^{i\omega} -1) (ae^{i\omega}-1)^{\ell+m} \xrightarrow{a\uparrow1} 
\pth{\frac{m}{e}}^m 2^{\frac{\ell+m}{2}} (1- \cos\omega)^{\frac{\ell-m}{2}} 
\]
so, applying \prettyref{eq:gtilde-fourier} and \prettyref{eq:chainrule},
\begin{align*}
\lim_{a\uparrow1} 
\left|\phi_{\tilde g}(e^{i\omega})\right|
= & ~ 2^\ell\sum_{m=0}^\ell \binom{\ell}{m} 2\ell(2\ell-1)\cdots(\ell+m+1)  \pth{\frac{m}{e}}^m  \\
\leq & ~ 2^\ell \sum_{m=0}^\ell \binom{\ell}{m} \frac{(2\ell)!}{(\ell+m)!} m!	= 2^\ell \ell! \sum_{m=0}^\ell \binom{2\ell}{\ell+m} \leq 2^{3\ell} \ell!.
\end{align*}
Substituting this into \prettyref{eq:gg} shows that 
$\|g\|_2^2 \leq 2^{6\ell} (\ell!)^2 \leq 2^{3k} k!$, completing the proof of \prettyref{eq:Ak}.

\section{Regret lower bound in the Gaussian EB model}\label{sec:gau_lower_bound}

In the Gaussian EB model, we have latent $\theta^n = (\theta_1,\ldots, \theta_n) \stackrel{i.i.d.}{\sim} G$ for some distribution $G$ on $\R$, and we observe i.i.d.~data $X^n = (X_1,\ldots,X_n)$ such that $X_i|\theta_i \sim \mathcal{N}(\theta_i,1)$. The goal is again to estimate the underlying Gaussian means $\theta^n$. 

With some abuse of notation, we still use $f_G$ to denote the Gaussian mixture density:
\begin{align}\label{def:gau_mixture}
f_G(x) \equiv \int \phi(x-\theta) G(\d\theta), \quad x\in\R,
\end{align} 
where $\phi(\cdot)$ is the standard normal density. Analogous to the definition (\ref{def:regret_minimax}) in the Poisson model, define the individual regret
\begin{align*}
\reg^g_n(\mathcal{G}) \equiv \inf_{\hat{\theta}^n} \sup_{G\in\mathcal{G}} \big\{\E_{G} \big(\hat{\theta}_n(X^n) - \theta_n\big)^2 - \mathsf{mmse}^g(G)\big\},
\end{align*}
where $\mathsf{mmse}^g(G)$ is the Bayes risk under prior $G$ in the Gaussian EB model. In the seminal paper \cite{jiang2009general}, the upper bound in (\ref{ineq:gau_regret_upper}) was proved for $\reg^g_n(\mathcal{G})$.
Up to logarithmic factors, the first bound $n^{-1}(\log n)^5$ of (\ref{ineq:gau_regret_upper}) has been shown by \cite[Theorem 1]{polyanskiy2021sharp} to be minimax optimal. The following result shows that the second bound of (\ref{ineq:gau_regret_upper}) is also minimax optimal up to logarithmic factors. 

\begin{theorem}\label{thm:gau_regret_lower}
For any $p > 0$, there exists some $c = c(p) > 0$ such that
\begin{align*}
\reg^g_n(\mathcal{G}_p(1)) \geq cn^{-\frac{p}{p+1}}(\log n)^{-11}.
\end{align*}
\end{theorem}
\begin{proof}[Proof of Theorem \ref{thm:gau_regret_lower}]
Since the proof is similar to that of Theorem \ref{thm:regret_lower}, we only provide a sketch of the arguments. We adopt a similar lower construction $\{G_{\bm{\tau}}\}$ as in (\ref{def:lower_bound_construction}), with the following modifications. Let $a_0\equiv 0$, and for $i\geq 1$, $I_i \equiv [i(\log n)^2, (i+1)(\log n)^2]$ with $a_i$ being the center of $I_i$. Let $w_i \equiv \big((i+1)(\log n)^2\big)^{-(p+1)}$, and $w_0\equiv 1 - \sum_{i=i_0}^N w_i$. Let $b_i\equiv a_i + \delta_i$, with $\delta_i^2 \equiv \big(nw_i(\log n)^{10}\big)^{-1}$. Then proceeding along the same lines to (\ref{ineq:density_lower_chi}), we have
\begin{align*}
\notag\chi^2\big(f_{\bm{\tau}}||f_{\bm{\tau}'}\big) &= \int \frac{\Big(\sum_{i=i_0}^N w_i \big(\phi(x-\lambda_i) - \phi(x-\lambda_i')\big)\Big)^2}{w_0\phi(x) + \sum_{i=i_0}^N w_i \phi(x-\lambda_i')}\d x\\
&\leq w_{i_*}\cdot \chi^2\big(\mathcal{N}(\lambda_{i_*},1)||\mathcal{N}(\lambda_{i_*}',1)\big) = w_{i_*}\Big(\exp\big((\lambda_{i_*} - \lambda'_{i_*})^2\big) - 1\Big),
\end{align*}
where $\phi(\cdot)$ denotes the standard normal density, and we use the Gaussian calculation in Lemma \ref{lem:poi_divergence}. Hence using $(\lambda_{i_*} - \lambda'_{i_*})^2 = \delta_{i_*}^2 = (nw_{i_*}(\log n)^{10})^{-1}$ and the lower bound $w_{i_*}\geq 2/n$, we have $\chi^2(f_{\bm{\tau}}||f_{\bm{\tau}'}) \leq 2/(n(\log n)^{10})$. 

Then we proceed to the calculations in (\ref{ineq:L2_hamming_ratio}) to obtain
\begin{align*}
\notag\pnorm{\theta_{G_{\bm{\tau}}} - \theta_{G_{\bm{\tau}'}}}{\ell_2(f_{\bm{\tau}'})}^{2,\mathsf{trun}} \gtrsim \frac{|\mathcal{I}|}{n(\log n)^{10}}.
\end{align*}
The rest of the proof is essentially identical to that of Theorem \ref{thm:regret_lower} by applying Assouad's lemma and choosing $N = c_pn^{1/(p+1)}/\log n$ and $i_0 = N/2$. The proof is complete.
\end{proof}

\section{Sub-optimality of empirical estimator}

The following result demonstrates the sub-optimality of the empirical estimator
\begin{align*}
\tilde{f}(y) \equiv \frac{1}{n}\sum_{i=1}^n \bm{1}\{Y_i = y\}, \quad y\in\mathbb{Z}_+,
\end{align*}
in density estimation. 

\begin{proposition}\label{prop:empirical_subopt}
Fix any $p > 0$. There exists some $G \in \mathcal{G}_p(1)$ and $c = c(p) > 0$ such that, with $Y_1,\ldots,Y_n$ i.i.d. observations from $f_G$,
\begin{align*}
\E_G H^2(\tilde{f}, f_G) \geq cn^{-\frac{p}{p+1}}(\log n)^{-\frac{4}{p+1}}.
\end{align*}
\end{proposition}
\begin{proof}
Consider the prior $G$ constructed in Lemma \ref{lem:ht_prior}. Let $y_\ast \equiv  \floor{c_p\big(n/(\log n)^4\big)^{1/(p+1)}}$ for some small enough $c_p$ such that $f_G(y) \geq 100\log n/n$ whenever $y \leq y_\ast$. Then
\begin{align*}
\E_G H^2(\tilde{f}, f_G) &= \sum_{y=0}^\infty \E\Big(\frac{\tilde{f}(y) - f_G(y)}{\sqrt{\tilde{f}(y)} + \sqrt{f_G(y)}}\Big)^2\\
&\asymp \sum_{y=0}^\infty \E\frac{\big[\tilde{f}(y) - f_G(y)\big]^2}{\tilde{f}(y) + f_G(y)}\\
&\geq \sum_{y=0}^{y_\ast} \E\frac{\big[\tilde{f}(y) - f_G(y)\big]^2}{\tilde{f}(y) + f_G(y)}\bm{1}\{\tilde{f}(y) \leq 50f_G(y)\}\\
&\gtrsim \sum_{y=0}^{y_\ast} \frac{\E\Big[\big(\tilde{f}(y) - f_G(y)\big)^2\bm{1}\{\tilde{f}(y) \leq 50f_G(y)\}\Big]}{f_G(y)}\\
&\stackrel{(*)}{\gtrsim} \sum_{y=0}^{y_\ast} \frac{1-f_G(y)}{n} - n^{-10} \asymp \frac{y_\ast}{n} \asymp_p n^{-\frac{p}{p+1}}(\log n)^{-\frac{4}{p+1}},
\end{align*}
where $(*)$ follows from standard binomial concentration.
\end{proof}

\section{Results in the compound setting}\label{sec:compound}

In this section, we collect a few results for the Poisson model in the compound setting, which play an essential role in the proof of Theorem \ref{thm:npmle_rho_free}.

\subsection{Density estimation}

Let $\theta^n = (\theta_1,\ldots,\theta_n) \in \R_+^n$ be a deterministic vector and $Y^n = (Y_1,\ldots,Y_n)$ be independent variables with $Y_i \sim \poi(\theta_i)$ for $1\leq i \leq n$. Let $f_i(\cdot) \equiv \poi(\cdot;\theta_i)$ be the $i$th marginal pmf, $G_n \equiv n^{-1}\sum_{i=1}^n\delta_{\theta_i}$, and the average density be
\begin{align}
f_{G_n}(y) \equiv \frac{1}{n}\sum_{i=1}^n f_i(y), \quad y \in \mathbb{Z}_+.
\end{align}
For any distribution $G$ on $\R_{+}$ and $p>0$, let
\begin{align*}
m_p(G) \equiv \int_{\R_+} u^p G(\d u).
\end{align*}
Let $\hat{G}$ be the NPMLE given by (\ref{def:npmle}). The $\Prob_{Y^n}$ and $\E_{Y^n}$ below are under the randomness of $Y^n$ described above.

\begin{proposition}\label{prop:compound_density}
Suppose $p > 0$ and $m_p(G_n)^{1/p} \leq n^{10}$. Let
\begin{align}\label{eq:eps_compound}
\epsilon_n \equiv \big(n^{-p/(2p+1)}m_p(G_n)^{1/(4p+2)} \vee n^{-1/2}\big) (\log n)^{4}. 
\end{align}
Then there exists some $t_\ast = t_\ast(p)$ such that for all $t \geq t_\ast$,
\begin{align*}
\Prob_{Y^n}\Big(H(f_{\hat{G}}, f_{G_n}) \geq t\epsilon_n\Big) \leq 2\exp\Big(-t^2n\epsilon_n^2/(8\log n)\Big) \leq 2\exp\Big(-t^2(\log n)^2/8\Big).
\end{align*}
Consequently, there exists some $C = C(p) > 0$ such that $\E_{Y^n} H^2(f_{\hat{G}}, f_{G_n}) \leq C\epsilon_n^2$.
\end{proposition}
%

\begin{proof}[Proof of Proposition \ref{prop:compound_density}]
The proof is very similar to Theorem \ref{thm:density_main} and we only sketch the minor difference below. Using the same notation as in Theorem \ref{thm:density_main} and following the proof there, we have
\begin{align*}
&\Prob_{Y^n}\Big(H(f_{\hat{G}}, f_{G_n}) \geq t\epsilon_n\Big)\\
&\leq \Prob_{Y^n}\Big(\max_{j\leq N} L(f_{H_j} + f_\ast, f_{G_n}) \geq \exp(-nt^2\epsilon_n^2/2)\Big) + \Prob_{Y^n}\Big(\prod_{i: Y_i > M} \frac{1}{f_\ast(Y_i)} \geq \exp(nt^2\epsilon_n^2/2)\Big)\\
&\equiv (I) + (II).
\end{align*}
Here, for some fixed $\eta > 0$ and $M > 0$ to be chosen later, $\{f_{H_j}: 1\leq j\leq N\}$ is a proper $(\eta, \pnorm{\cdot}{\infty,M})$-net of $B(t\epsilon_n)^c$, and by Lemma \ref{lem:entropy} there exists some universal $K > 0$ such that
\begin{align*}
N \leq K\sqrt{M}\big(\log (1/\eta)\big)^{3/2}\log(M/\eta). 
\end{align*}
To bound $(I)$, we have
\begin{align*}
(I) &\leq N\cdot \max_{j\leq N} \Prob_{Y^n}\Big(\prod_{i=1}^n \sqrt{\frac{(f_{H_j} + f_\ast)(Y_i)}{f_{G_n}(Y_i)}} \geq \exp(-nt^2\epsilon_n^2/4)\Big)\\
&\leq N\cdot \max_{j\leq N} \exp\bigg(nt^2\epsilon_n^2/4 + \sum_{i=1}^n \log \E_{Y^n} \sqrt{\frac{(f_{H_j} + f_\ast)(Y_i)}{f_{G_n}(Y_i)}}\bigg)\\
&\leq N\cdot \max_{j\leq N} \exp\bigg(nt^2\epsilon_n^2/4 + n \cdot \Big(\sum_{y=0}^\infty \sqrt{(f_{H_j} + f_\ast)(y) f_{G_n}(y)} - 1\Big)\bigg).
\end{align*}
Here in the last step, since $\log x \leq x - 1$ for all $x > 0$, we have
\begin{align*}
&\sum_{i=1}^n \log \E_{Y^n} \sqrt{\frac{(f_{H_j} + f_\ast)(Y_i)}{f_{G_n}(Y_i)}} \leq \sum_{i=1}^n \E_{Y^n} \sqrt{\frac{(f_{H_j} + f_\ast)(Y_i)}{f_{G_n}(Y_i)}} - n\\
&= \sum_{i=1}^n \sum_{y = 0}^\infty f_i(y) \sqrt{\frac{(f_{H_j} + f_\ast)(y)}{f_{G_n}(y)}} - n = n \cdot \Big(\sum_{y=0}^\infty \sqrt{(f_{H_j} + f_\ast)(y) f_{G_n}(y)} - 1\Big).
\end{align*}
Now the same argument as in Theorem \ref{thm:density_main} implies that for sufficiently large $t$ (depending only on $p$),
\begin{align*}
(I) \leq \exp\Big(K_p\sqrt{M}(\log n)^{5/2} + nt^2\epsilon_n^2/4 - n(t\epsilon_n)^2/2 + n\sqrt{\eta M}\Big) \leq \exp(-nt^2\epsilon_n^2/8).
\end{align*}
The rest of the proof is the same, upon noting that Lemma \ref{lem:truncation_moment} can also be extended to the compound setting using the same argument as above. 
\end{proof}

\subsection{Regret bounds}\label{subsec:compound_regret}
In the compound estimation setting there are multiple definitions of regret \cite{jiang2009general,greenshtein2009asymptotic,saha2020nonparametric,polyanskiy2021sharp}; see \cite[Proposition 3]{polyanskiy2021sharp} for a comparison of these with the empirical Bayes regret.
Following \cite{jiang2009general, polyanskiy2021sharp}, we consider the following notion of (total) regret in the compound setup. For any estimator $\hat{\theta}^n: \mathbb{Z}_+^n \rightarrow \R_+^n$, its total regret at $\theta^n$ is defined by
\begin{align}\label{eq:totreg-compound}
\totreg_n(\hat{\theta}^n; \theta^n) =  \E_{Y^n} \pnorm{\hat{\theta}^n(Y^n) - \theta^n}{}^2 - \E_{Y^n}\pnorm{\theta_{G_n}(Y^n) - \theta^n}{}^2,
\end{align}
where $G_n$ denotes the empirical distribution of $\theta^n$. 
The interpretation of \prettyref{eq:totreg-compound} is the excess risk with respect to the \emph{best separable} oracle, which is simply the Bayes rule
$\theta_{G_n}(Y^n)=\theta_{G_n}(Y_1),\ldots,\theta_{G_n}(Y_n)$ with the empirical distribution $G_n$ as the prior.
 Recall that for any distribution $G$ on $\R_+$ and $p > 0$, $m_p(G) = \int_{\R_+} u^p G(\d u)$. Recall that $\hat{\theta}^{\npmle,n}$ is given by (\ref{def:npmle_eb}).

\begin{theorem}\label{thm:npmle_compound}
Suppose $p > 1$ and $\theta^n\in\R^n$ is such that $m_p(G_n) \leq n^{10p}$. Then 
\begin{align}\label{eq:bound_compound}
\totreg_n(\hat{\theta}^{\npmle,n}; \theta^n) \leq C(\log n)^{13}\Big[n^{\frac{3}{2p+1}}m_p(G_n)^{\frac{3}{2p+1}} \vee 1\Big]
\end{align}
for some universal $C > 0$. Moreover, \prettyref{eq:bound_compound} continues to hold if $\theta_{\hat{G}}(Y^n)$ (resp.~$\theta_{G_n}(Y^n)$) in the definition \prettyref{eq:totreg-compound}
 of $\totreg_n(\hat{\theta}^{\npmle,n}; \theta^n)$ is replaced by the regularized version $\theta_{\hat{G}}(Y^n;\rho)$ (resp.~$\theta_{G_n}(Y^n;\rho)$) for any $\rho \leq n^{-C_\rho}$ with some large universal $C_\rho > 0$.
\end{theorem}

\begin{remark}\label{rmk:l2_regret}
A related definition of total regret at $\theta^n$ (see e.g., \cite{saha2020nonparametric}) is
\begin{align}
\totreg_n'(\hat{\theta}^n; \theta^n) = \E_{Y^n} \pnorm{\hat{\theta}^n(Y^n) - \theta_{G_n}(Y^n)}{}^2.
\label{eq:totgret-compound2}
\end{align}
Note that without orthogonality principle, it is unclear whether \prettyref{eq:totreg-compound} and \prettyref{eq:totgret-compound2} coincide.
Nevertheless, as we show in Section \ref{subsec:proof_remark_regret}, under the same conditions as in Theorem \ref{thm:npmle_compound}, the bound (\ref{eq:bound_compound}) also holds for $\totreg_n'(\hat{\theta}^{\npmle,n}; \theta^n)$, even when $\theta_{\hat{G}}(Y^n)$ and 
$\theta_{G_n}(Y^n)$  are replaced by their regularized versions $\theta_{\hat{G}}(Y^n;\rho)$ and $\theta_{G_n}(Y^n;\rho)$, respectively. 
\end{remark}

\begin{proof}[Proof of Theorem \ref{thm:npmle_compound}]
We only consider the regularized version $\theta_{\hat{G}}(Y^n;\rho)$ and $\theta_{G_n}(Y^n;\rho)$.
The result for the unregularized $\theta_{\hat{G}}(Y^n)$ and $\theta_{G_n}(Y^n)$ follows from the same steps and error bounds leading up to (\ref{eq:npmle_reg_decomp}).

Fix some $M > 0$ to be chosen later. We have
\begin{align*}
&\E_{Y^n} \pnorm{\theta_{\hat{G}}(Y^n;\rho) - \theta^n}{}^2\\
&= \E_{Y^n} \sum_{i=1}^n \big(\theta_{\hat{G}}(Y_i;\rho) - \theta_i\big)^2\bm{1}_{Y_i \leq M} + \underbrace{\E_{Y^n} \sum_{i=1}^n \big(\theta_{\hat{G}}(Y_i;\rho) - \theta_i\big)^2\bm{1}_{Y_i > M}}_{\zeta_1}\\
&\leq \E_{Y^n} \sum_{i=1}^n \big(\theta_{\hat{G}}(Y_i;\rho) - \theta_i\big)^2\bm{1}_{Y_i \leq M, |Y_i - \theta_i| \leq \sqrt{M}(\log n)^2} + \zeta_1\\
&\quad + \underbrace{\E_{Y^n} \sum_{i=1}^n \big(\theta_{\hat{G}}(Y_i;\rho) - \theta_i\big)^2\bm{1}_{Y_i \leq M, \theta_i > 2M}}_{\zeta_2} + \underbrace{\E_{Y^n} \sum_{i=1}^n \big(\theta_{\hat{G}}(Y_i;\rho) - \theta_i\big)^2\bm{1}_{Y_i \leq M, \theta_i\leq 2M, |Y_i - \theta_i| > \sqrt{M}(\log n)^2}}_{\zeta_3}.
\end{align*}
Let $A_i$ denote the event $\{Y_i \leq M, |Y_i - \theta_i| \leq \sqrt{M}(\log n)^2\}$ for $i\leq n$, and $E$ denote the event $H(f_{\hat{G}}, f_{G_n}) \leq t_\ast\epsilon_n$, where $\epsilon_n$ is given by (\ref{eq:eps_compound}). Then
\begin{align*}
\E_{Y^n} \sum_{i=1}^n \big(\theta_{\hat{G}}(Y_i;\rho) - \theta_i\big)^2\bm{1}_{A_i} = \E_{Y^n} \sum_{i=1}^n \big(\theta_{\hat{G}}(Y_i;\rho) - \theta_i\big)^2\bm{1}_{A_i}\bm{1}_{E} + \underbrace{\E_{Y^n} \sum_{i=1}^n \big(\theta_{\hat{G}}(Y_i;\rho) - \theta_i\big)^2\bm{1}_{A_i}\bm{1}_{E^c}}_{\zeta_4}.
\end{align*}  
On the event $E$, we find a net $\{G_j, j=1,\ldots,N\}$ such that $H(f_{G_j}, f_{G_n}) \leq t_\ast \epsilon_n$ and for any $\tilde{G}$ such that $f_{\tilde{G}}$ lies in the $\epsilon_n$-Hellinger ball around $f_G$, there exists some $j\in[N]$ such that
\begin{align*}
\pnorm{\theta_{G_j}(\cdot;\rho) - \theta_{\tilde{G}}(\cdot;\rho)}{\infty, M} \equiv \sup_{y \in[0,M]\cap \mathbb{Z}} \big|\theta_{G_j}(y;\rho) - \theta_{\tilde{G}}(y;\rho)\big|\leq \eta,
\end{align*}
for some $\eta>0$ to be specified. Then
\begin{align*}
&\E_{Y^n} \sum_{i=1}^n \big(\theta_{\hat{G}}(Y_i;\rho) - \theta_i\big)^2\bm{1}_{A_i}\bm{1}_E\\
&\leq \underbrace{\E_{Y_n} \inf_{j\in [N]}  \Big|\sum_{i=1}^n \big(\theta_{\hat{G}}(Y_i;\rho) - \theta_i\big)^2\bm{1}_{A_i} - \sum_{i=1}^n \big(\theta_{G_j}(Y_i;\rho) - \theta_i\big)^2\bm{1}_{A_i}\Big|\bm{1}_E}_{\zeta_5}+\\
&\quad + \E_{Y_n} \max_{j\leq N} \sum_{i=1}^n \big(\theta_{G_j}(Y_i;\rho) - \theta_i\big)^2\bm{1}_{A_i}.
\end{align*}
For each $j\in[N]$, define the variable
\begin{align*}
Z_j = \sum_{i=1}^n  \Big[\big(\theta_{G_j}(Y_i;\rho) - \theta_i\big)^2 - \big(\theta_{G_n}(Y_i;\rho) - \theta_i\big)^2\Big]\bm{1}_{A_i}.
\end{align*}
Then
\begin{align*}
&\E_{Y^n} \max_{j\leq N} \sum_{i=1}^n \big(\theta_{G_j}(Y_i;\rho) - \theta_i\big)^2\bm{1}_{A_i}\\
&= \E_{Y^n} \max_{j\in N} Z_j + \E_{Y^n} \sum_{i=1}^n \big(\theta_{G_n}(Y_i;\rho) - \theta_i\big)^2\bm{1}_{A_i}\\
&\leq \underbrace{\E_{Y^n}  \max_{j\in N} |Z_j - \E_{Y^n} Z_j|}_{\zeta_6} + \max_{j\in N} \E_{Y^n} Z_j + \E_{Y^n} \sum_{i=1}^n \big(\theta_{G_n}(Y_i;\rho) - \theta_i\big)^2\bm{1}_{A_i}\\
&= \max_{j\in [N]} \E_{Y^n}  \sum_{i=1}^n \big(\theta_{G_j}(Y_i;\rho) - \theta_i\big)^2\bm{1}_{A_i} + \zeta_6.
\end{align*}
To summarize, we have
\begin{align*}
&\E_{Y^n} \pnorm{\theta_{\hat{G}}(Y^n;\rho) - \theta^n}{}^2  - \E_{Y^n} \pnorm{\theta_{G_n}(Y^n) - \theta^n}{}^2\\
&\leq \max_{j\leq N}\Big[\E_{Y_n} \pnorm{\theta_{G_j}(Y^n;\rho) - \theta^n}{}^2 - \E_{Y^n} \pnorm{\theta_{G_n}(Y^n) - \theta^n}{}^2\Big]+ \sum_{i=1}^6 \zeta_i\\
&= n\cdot \max_{j\leq N}\Big[\E_{G_n} \big(\theta_{G_j}(Y;\rho) - \theta\big)^2 - \E_{G_n} \big(\theta_{G_n}(Y) - \theta\big)^2\Big]+ \sum_{i=1}^6 \zeta_i\\
&= n\cdot \max_{j\leq N} \E_{G_n} \big(\theta_{G_j}(Y;\rho) - \theta_{G_n}(Y)\big)^2 + \sum_{i=1}^6 \zeta_i,
\end{align*}

Recall that $\{G_j\}_{j=1}^N$ satisfy $H^2(f_{G_j}, f_{G_n}) \leq (t_\ast\epsilon_n)^2$, so Theorem \ref{thm:regret_MLE_EB} yields that
\begin{align}\label{ineq:compound_regret_main}
\E_{G_n} \big(\theta_{G_j}(Y;\rho) - \theta_{G_n}(Y)\big)^2 \lesssim (\log n)^{12} \cdot \Big[n^{-\frac{2(p-1)}{2p+1}}m_p(G_n)^{\frac{3}{2p+1}} + n^{-1}\Big].
\end{align}

Next we bound $\zeta_1-\zeta_6$ and choose the parameters $(M,\eta)$ along the way. Using Lemma \ref{lem:bayes_form_upper} and calculations in (\ref{eq:mle_error}),
\begin{align*}
\zeta_1 &\lesssim \sum_{i=1}^n \E_{Y^n}\big[ \big(\theta_{\hat{G}}(Y_i;\rho) - (Y_i+1)\big)^2 +\big(\theta_i - (Y_i+1)\big)^2 \big] \bm{1}_{Y_i > M}\\
&\lesssim \log^2(1/\rho) \cdot \Big(\sum_{i=1}^n \E_{Y^n} Y_i\bm{1}_{Y_i \geq M} + \sum_{i=1}^n \E_{Y^n} (Y_i - \theta_i)^2\bm{1}_{Y_i > M}\Big)\\
&= n\log^2(1/\rho) \cdot \E_{G_n}[\big(Y + (Y-\theta)^2\big)\bm{1}_{Y > M}]\\
&\lesssim n\log^2(1/\rho)\Big(\frac{m_p(G_n)}{M^{p-1}} + m_p(G_n)^{1/p}\exp(-cM)\Big).
\end{align*}
For $\zeta_2$, note that the trivial bound $\max_{i\in[n]} \theta_i \leq (nm_p(G_n))^{1/p}$ and Lemma \ref{lem:bayes_form_upper} yields
\begin{align}\label{eq:bound_trivial}
\notag\big|\theta_{\hat{G}}(Y_i;\rho) - \theta\big|\bm{1}_{Y_i \leq M} &\leq \big(\big|\theta_{\hat{G}}(Y_i;\rho) - (Y_i+1)\big| + \big|Y_i+1-\theta_i\big|\big)\bm{1}_{Y_i\leq M}\\
&\lesssim \log(1/\rho) \cdot (M + (nm_p(G_n))^{1/p}).
\end{align}
For each $i\in[n]$ such that $\theta_i > 2M$, we also have by Lemma \ref{lem:poi_basic} that
\begin{align*}
\Prob_{Y^n}(Y_i < M) \leq \Prob_{Y^n}(Y_i - \theta_i \leq -\theta_i/2) \leq \exp(-c\theta_i) \leq \exp(-2cM), 
\end{align*}
so
\begin{align*}
\zeta_2 \lesssim n\log^2(1/\rho) \cdot (M + (nm_p(G_n))^{1/p})^2\exp(-2cM).
\end{align*}
For $\zeta_3$, note that for $\theta_i \leq 2M$, 
\begin{align*}
\Prob_{Y^n}(|Y_i - \theta_i| > \sqrt{M}(\log n)^2) \leq \exp(-c(\log n)^2), 
\end{align*}
so we have
\begin{align*}
\zeta_3 \lesssim n\log^2(1/\rho) \cdot (M + (nm_p(G_n))^{1/p})^2\exp\big(-c(\log n)^2\big).
\end{align*}
For $\zeta_4$, using again the trivial bound (\ref{eq:bound_trivial}), 
we have
\begin{align*}
\zeta_4 &\lesssim n\log^2(1/\rho) \cdot  \big(M + (nm_p(G_n))^{\frac{1}{p}}\big)^2 \Prob_{Y^n}(E^c)\\ 
&\lesssim n\log^2(1/\rho) \cdot  \big(M + (nm_p(G_n))^{\frac{1}{p}}\big)^2 \cdot\exp\big(-t_\ast^2(\log n)^2/4\big).
\end{align*}
For $\zeta_5$, on the event $E$, there exists some $j_0\in [N]$ such that $\pnorm{\theta_{\hat{G}}(\cdot;\rho) - \theta_{G_{j_0}}(\cdot;\rho)}{\infty, M} \leq \eta$, so using again the trivial bound (\ref{eq:bound_trivial}), 
\begin{align*}
\zeta_5 &\leq \sum_{i=1}^n \E \Big|\big(\theta_{\hat{G}}(Y_i;\rho) - \theta_{G_n}(Y_i;\rho)\big)^2 - \big(\theta_{G_{j_0}}(Y_i;\rho) - \theta_{G_n}(Y_i;\rho)\big)^2\Big|\bm{1}_{Y_i\leq M}\\
&\lesssim n\log(1/\rho)\big(M + (nm_p(G_n))^{\frac{1}{p}}\big) \cdot \eta.
\end{align*}
For $\zeta_6$, note that for any $j\in[N]$, by Bernstein's inequality we have 
\begin{align*}
\Prob(|Z_j - \E_{Y^n} Z_j|\geq t) \leq \exp(-C\frac{t^2}{\sigma_j^2} \wedge \frac{t}{B_j}),
\end{align*}
where
\begin{align*}
B_j &\equiv \max_{i\in[n]}\max_{y} \big|\big(\theta_{G_j}(y;\rho) - \theta_i\big)^2 - \big(\theta_{G_n}(y;\rho) - \theta_i\big)^2\big|\bm{1}_{y\leq M, |y - \theta_i| \leq \sqrt{M}(\log n)^2}\\
& \lesssim \log^2(n/\rho)M,
\end{align*}
and
\begin{align*}
\sigma_j^2 &\equiv \var_{Y^n}(Z_j)\\
&= \sum_{i=1}^n \var_{Y^n}\big(\big[\big(\theta_{G_j}(Y_i;\rho) - \theta_i\big)^2 - \big(\theta_{G_n}(Y_i;\rho) - \theta_i\big)^2\big]\bm{1}_{A_i}\big)\\
&\leq \sum_{i=1}^n \E_{Y^n}\Big[\Big(\big(\theta_{G_j}(Y_i;\rho) - \theta_i\big)^2 - \big(\theta_{G_n}(Y_i;\rho) - \theta_i\big)^2\Big)^2\bm{1}_{A_i}\Big]\\
& \lesssim \log^2(n/\rho)M \cdot \sum_{i=1}^n  \E_{Y^n} \big(\theta_{G_j}(Y_i;\rho) - \theta_{G_n}(Y_i;\rho)\big)^2\\
&= n\log^2(n/\rho)M \cdot \E_{G_n} \big(\theta_{G_j}(Y;\rho) - \theta_{G_n}(Y;\rho)\big)^2\\
&\lesssim (\log n)^{12}\log^2(n/\rho) \cdot \Big(n^{\frac{3}{2p+1}}m_p(G_n)^{\frac{3}{2p+1}} \vee 1\Big)M,
\end{align*}
using Theorem \ref{thm:regret_MLE_EB} (its variant with $\theta_{G_n}(\cdot)$ replaced by $\theta_{G_n}(\cdot;\rho)$) in the last step. Hence using the entropy bound in Lemma \ref{lem:net_cardinality}, we have
\begin{align*}
\zeta_6 \lesssim \sqrt{\log N} \cdot \max_j\sigma_j + \log N \cdot \max_jB_j \lesssim (\log (nM/\rho\eta))^{10} \Big[M^{3/4}\Big(n^{\frac{3}{4p+2}}m_p(G_n)^{\frac{3}{4p+2}} \vee 1\Big) + M^{3/2}\Big].
\end{align*}
Combining the bounds for $\zeta_1-\zeta_6$, we have
\begin{align*}
\sum_{i=1}^6 \zeta_i &\lesssim \frac{n\log^2(1/\rho)m_p(G_n)}{M^{p-1}} + n\log^2(1/\rho) \cdot  \big(M + (nm_p(G_n))^{\frac{1}{p}}\big)^2 \cdot\exp\big(-c'((\log n)^2 \wedge M)\big)\\
&\quad + n\log(1/\rho)\big(M + (nm_p(G_n))^{\frac{1}{p}}\big) \cdot \eta + (\log (nM/\rho\eta))^{10} \Big[M^{3/4}\Big(n^{\frac{3}{4p+2}}m_p(G_n)^{\frac{3}{4p+2}} \vee 1\Big)+ M^{3/2}\Big].
\end{align*}
By choosing $M = (\log n)^2 \cdot \big((nm_p(G_n))^{\frac{2}{2p+1}} \vee 1\big)$ and $\eta = n^{-100}$, we obtain the desired result. 
\end{proof} 

\begin{lemma}\label{lem:net_cardinality}
For any $\rho > 0$, let $\Theta_0(\rho) \equiv \{\theta_{G}(\cdot; \rho): G\subset\mathcal{P}(\R_+)\}$ be the set of all $\rho$-regularized Bayes forms. For any $\theta_{G}(\cdot;\rho), \theta_{H}(\cdot;\rho)\in\Theta_0(\rho)$ and $M > 0$, let
\begin{align*}
\pnorm{\theta_{G}(\cdot;\rho) - \theta_{H}(\cdot,\rho)}{\infty,M} \equiv \sup_{y\in\mathbb{Z}_+\cap[0,M]} \big|\theta_{G}(y;\rho) - \theta_{H}(y;\rho)\big|.
\end{align*}
Then for any $\eta \in (0,10^{-3})$ and $M \geq (\log(1/\rho\eta))^{\rho_M}$ for some sufficiently large $\rho_M > 0$, there exists some universal $K > 0$ such that
\begin{align*}
\log \mathcal{N}(\eta, \Theta_0, \pnorm{\cdot}{\infty, M}) \leq K\sqrt{M}\big(\log(M/\rho\eta)\big)^{5/2}.
\end{align*}
\end{lemma}
\begin{proof}
For any $y\in[0,M]$ and distributions $G, H$ such that $\pnorm{f_G - f_H}{\infty,2M} \leq \tau$, we have
\begin{align*}
\big|\theta_G(y;\rho) - \theta_H(y;\rho)\big| &= (y+1)\Big|\frac{\Delta f_G(y)}{f_G(y)\vee\rho} - \frac{\Delta f_H(y)}{f_H(y)\vee\rho}\Big|\\
&\leq (y+1)\cdot\Big[\Big|\frac{\Delta f_G(y)}{f_G(y)\vee\rho} - \frac{2\Delta f_G(y)}{f_G(y)\vee\rho + f_H(y)\vee \rho}\Big| + \Big|\frac{2\big(\Delta f_G(y) - \Delta f_H(y)\big)}{f_G(y)\vee\rho + f_H(y)\vee \rho}\Big|\\
&\quad + \Big|\frac{\Delta f_H(y)}{f_H(y)\vee\rho} - \frac{2\Delta f_H(y)}{f_G(y)\vee\rho + f_H(y)\vee \rho}\Big|\Big]\\
&\stackrel{(*)}{\lesssim} (y+1)\cdot\Big[\frac{1}{\sqrt{y+1}}\log(1/\rho)\cdot \frac{\pnorm{f_G - f_H}{\infty,M}}{2\rho} + \frac{2\pnorm{f_G - f_H}{\infty,2M}}{2\rho}\Big]\\
&\lesssim M\log(1/\rho)\frac{\tau}{\rho},
\end{align*} 
where $(*)$ follows from Lemma \ref{lem:bayes_form_upper} in the paper. Now by choosing $\tau = \eta\rho/(KM\log(1/\rho))$ for some large universal $K > 0$, the claim follows from Lemma \ref{lem:entropy} in the paper by adjusting the constants. 
\end{proof}

\subsection{Proof for Remark \ref{rmk:l2_regret}}\label{subsec:proof_remark_regret}
\begin{proof}
The proof is similar to that of Theorem \ref{thm:npmle_compound} so we omit some repetitive details. Same as Theorem \ref{thm:npmle_compound}, we directly prove the version with $\theta_{\hat{G}}(Y^n;\rho)$ and $\theta_{G_n}(Y^n;\rho)$.

Let $E$ denote the event $H(f_{\hat{G}}, f_{G_n}) \leq t_\ast\epsilon_n$, where $\epsilon_n$ is given by (\ref{eq:eps_compound}).  Let $M > 0$ be chosen later. Then
\begin{align*}
&\E_{Y^n} \pnorm{\theta_{\hat{G}}(Y^n;\rho) - \theta_{G_n}(Y^n;\rho)}{}^2\\
&= \E_{Y^n} \sum_{i=1}^n \big(\theta_{\hat{G}}(Y_i;\rho) - \theta_{G_n}(Y_i;\rho)\big)^2\bm{1}_{Y_i \leq M} + \underbrace{\E_{Y^n} \sum_{i=1}^n \big(\theta_{\hat{G}}(Y_i;\rho) - \theta_{G_n}(Y_i;\rho)\big)^2\bm{1}_{Y_i > M}}_{\xi_1}\\
&= \E_{Y^n} \sum_{i=1}^n \big(\theta_{\hat{G}}(Y_i;\rho) - \theta_{G_n}(Y_i;\rho)\big)^2\bm{1}_{Y_i \leq M}\bm{1}_E + \underbrace{\E_{Y^n} \sum_{i=1}^n \big(\theta_{\hat{G}}(Y_i;\rho) - \theta_{G_n}(Y_i;\rho)\big)^2\bm{1}_{Y_i \leq M}\bm{1}_{E^c}}_{\xi_2} + \xi_1.
\end{align*}
Now we take the same covering $\{G_j\}_{j=1}^N$ as in the proof of Theorem \ref{thm:npmle_compound} to obtain
\begin{align*}
&\E_{Y^n} \sum_{i=1}^n \big(\theta_{\hat{G}}(Y_i;\rho) - \theta_{G_n}(Y_i;\rho)\big)^2\bm{1}_{Y_i \leq M}\bm{1}_E\\
&\leq \underbrace{\E_{Y_n} \inf_{j\in [N]}  \Big|\sum_{i=1}^n \big(\theta_{\hat{G}}(Y_i;\rho) - \theta_{G_n}(Y_i;\rho)\big)^2\bm{1}_{Y_i \leq M} - \sum_{i=1}^n \big(\theta_{G_j}(Y_i;\rho) - \theta_{G_n}(Y_i;\rho)\big)^2\bm{1}_{Y_i\leq M}\Big|\bm{1}_E}_{\xi_3}+\\
&\quad + \E_{Y_n} \max_{j\leq N} \sum_{i=1}^n \big(\theta_{G_j}(Y_i;\rho) - \theta_{G_n}(Y_i;\rho)\big)^2\bm{1}_{Y_i\leq M}\\
&\leq \max_{j\leq N}\E \sum_{i=1}^n \big(\theta_{G_j}(Y_i;\rho) - \theta_{G_n}(Y_i;\rho)\big)^2 + \xi_3\\
&\quad \underbrace{\E_{Y_n} \max_{j\leq N} \bigg|\sum_{i=1}^n \big(\theta_{G_j}(Y_i;\rho) - \theta_{G_n}(Y_i;\rho)\big)^2\bm{1}_{Y_i\leq M} - \E \sum_{i=1}^n \big(\theta_{G_j}(Y_i;\rho) - \theta_{G_n}(Y_i;\rho)\big)^2\bm{1}_{Y_i\leq M}\bigg|}_{\xi_4}.
\end{align*}
To summarize, we have
\begin{align*}
&\E_{Y^n} \pnorm{\theta_{\hat{G}}(Y^n;\rho) - \theta_{G_n}(Y^n;\rho)}{}^2 \\ 
&\leq \max_{j\leq N} \E_{Y^n} \pnorm{\theta_{G_j}(Y^n;\rho) - \theta_{G_n}(Y^n;\rho)}{}^2+ \sum_{i=1}^4 \xi_i\\
&= n\cdot \max_{j\leq N} \E_{G_n} \big(\theta_{G_j}(Y;\rho) - \theta_{G_n}(Y;\rho)\big)^2 + \sum_{i=1}^4 \xi_i.
\end{align*}

The first term is bounded as in (\ref{ineq:compound_regret_main}) by Theorem \ref{thm:regret_MLE_EB} (its variant with $\theta_{G_n}(\cdot)$ replaced by $\theta_{G_n}(\cdot;\rho)$). Similarly, $\xi_1,\xi_2,\xi_3$ enjoy the same bounds as $\zeta_1,\zeta_4,\zeta_5$.
For $\xi_4$, note that for any fixed distribution $G$, by Bernstein's inequality we have 
\begin{align*}
\Prob\bigg(\bigg|\sum_{i=1}^n \big(\theta_G(Y_i;\rho) - \theta_{G_n}(Y_i;\rho)\big)^2\bm{1}_{Y_i\leq M} - \E\big(\theta_G(Y_i;\rho) - \theta_{G_n}(Y_i;\rho)\big)^2\bm{1}_{Y_i\leq M}\bigg| \geq t\bigg) \leq \exp(-C\frac{t^2}{\sigma^2} \wedge \frac{t}{B}),
\end{align*}
where by Lemma \ref{lem:bayes_form_upper},
\begin{align*}
B &\equiv \max_{y} \big(\theta_G(y;\rho) - \theta_{G_n}(y;\rho)\big)^2\bm{1}_{Y_i\leq M} \lesssim \log^2(1/\rho)M, \\
\sigma^2 &\equiv \sum_{i=1}^n \var(\big(\theta_G(Y_i;\rho) - \theta_{G_n}(Y_i;\rho)\big)^2\bm{1}_{Y_i\leq M})\\
&\leq \sum_{i=1}^n \E \big(\theta_G(Y_i;\rho) - \theta_{G_n}(Y_i;\rho)\big)^4\bm{1}_{Y_i\leq M}\\
& \lesssim \log^2(1/\rho)M \cdot \sum_{i=1}^n  \E \big(\theta_G(Y_i;\rho) - \theta_{G_n}(Y_i;\rho)\big)^2\\
&\lesssim (\log n)^{12}\log^2(1/\rho) \cdot\Big(n^{\frac{3}{2p+1}}m_p(G_n)^{\frac{3}{2p+1}} \vee 1\Big)M.
\end{align*}
Hence using the entropy bound in Lemma \ref{lem:net_cardinality}, we have
\begin{align*}
\xi_4 \lesssim \sqrt{\log N} \cdot \sigma + \log N \cdot B \lesssim (\log (nM/\rho\eta))^{10} \Big[M^{3/4}\Big(n^{\frac{3}{4p+2}}m_p(G_n)^{\frac{3}{4p+2}} \vee 1\Big) + M^{3/2}\Big].
\end{align*}
Now take the same choices of $(M,\eta)$ as in the proof of Theorem \ref{thm:regret_MLE_EB} to conclude. 
\end{proof}

\section*{Acknowledgment}
The authors thank Yutong Nie for useful input on the lower bound in \prettyref{thm:density_lower}. 
The authors are also grateful to Soham Jana and Yury Polyanskiy for helpful discussions.

\end{document}